%% file: main.tex
\input{preamble.tex}

\usepackage[style=alphabetic, maxnames = 5, maxalphanames=5]{biblatex}
\renewbibmacro{in:}{}

\title{Equivariant $K$-theory ring of the affine Grassmannian as deformation to normal cone}
\author{Jakub Löwit}
\date{}

\begin{document}

\maketitle

\begin{abstract}
We revisit fixed-point localization techniques in equivariant topological $K$-theory, presenting its fibers over varying points of the equivariant base as cohomology of the corresponding fixed-point schemes. We then employ this framework for an effective description of the equivariant $K$-theory ring of the affine Grassmannian $\Gr_G$. We do so by a detailed study of the relevant fixed points, extending known results to the action of any $(x, \zeta) \in T \times \Gm^{\rot}$. We obtain a description of the complexified topological $K$-theory ring $K^{\top, 0}_{T\times \Gm^{\rot}}(\Gr_G; \C)$ as the ring of functions on a family of affine blowups, extending and proving a conjecture of Roman Bezrukavnikov. While this ring is fairly big, we provide a preferred infinite set of topological generators. 
\end{abstract}

\tableofcontents

\input{0_introduction}
\input{1_groups_and_representations}
\input{2_equivariant_K_theory}
\input{3_cocharacter_graphs_and_resonance}

\input{4_fixed_points}
\input{5_bezrukavnikov_finkelberg_description_of_equivariant_K_theory}

\appendix
\input{A_some_explicit_computations}

\printbibliography

\bigskip
\noindent Jakub Löwit, \newline
Institute of Science and Technology Austria (ISTA), \newline
Am Campus 1, \newline 
3400 Klosterneuburg, \newline
Austria \newline
\texttt{jakub.loewit@ist.ac.at}

\end{document}

%% file: preamble.tex
\documentclass[a4paper, 10pt]{article}
\usepackage[a4paper, total={6in, 9.5in}]{geometry}
\usepackage[utf8]{inputenc}
\usepackage{amssymb}
\usepackage{amsthm}
\usepackage{amsmath}
\usepackage{amsfonts}

\usepackage[toc,page]{appendix}

\usepackage{graphicx}

\usepackage{titlesec}

\usepackage{comment}
\usepackage{titling}
\usepackage{tikz-cd}

\usepackage{enumitem}
\setlist{nosep} 

\usepackage{mathabx}

\usepackage{graphicx}
\usepackage{hyperref}
\usepackage{mathrsfs}
\usepackage{mathtools}

\usepackage[mathscr]{euscript}

\usepackage{stmaryrd}

\titlelabel{\thetitle.\quad}
\titleformat{\subsubsection}[runin]{\normalfont\bfseries}{\thesubsubsection.}{3pt}{}

\usepackage[style=alphabetic, maxnames = 5, maxalphanames=5]{biblatex}

\newcounter{cislo} \numberwithin{cislo}{section}

\numberwithin{equation}{subsection}

\usepackage{mathabx,epsfig}
\def\acts{\mathrel{\reflectbox{$\righttoleftarrow$}}}
\def\actsr{\righttoleftarrow}

\newtheorem{theorem}[cislo]{Theorem}
\newtheorem{lemma}[cislo]{Lemma}
\newtheorem{proposition}[cislo]{Proposition}
\newtheorem{corollary}[cislo]{Corollary}
\newtheorem{claim}[cislo]{Claim}

\newtheorem*{motto}{Motto}

\theoremstyle{definition}
\newtheorem{definition}[cislo]{Definition}
\newtheorem{notation}[cislo]{Notation}
\newtheorem{convention}[cislo]{Convention}
\newtheorem{remark}[cislo]{Remark}

\newtheorem{example}[cislo]{Example}
\newtheorem{assumption}[cislo]{Assumption}
\newtheorem{discussion}[cislo]{Discussion}

\newtheorem{construction}[cislo]{Construction}
\newtheorem{setup}[cislo]{Setup}
\newtheorem{aside}[cislo]{Aside}
\newtheorem{recollection}[cislo]{Recollection}

\theoremstyle{remark}

\let\top\relax
\newcommand{\top}{\mathrm{top}}
\newcommand{\btop}{\mathbf{top}}
\newcommand{\st}{\mathrm{st}}
\newcommand{\bet}{\mathrm{sing}}
\newcommand{\sing}{\mathrm{sing}}

\let\k\relax
\newcommand{\k}{\mathsf{k}}

\let\sc\relax
\newcommand{\sc}{\mathrm{sc}}
\newcommand{\der}{\mathrm{der}}

\newcommand{\ext}{\mathrm{ext}}
\newcommand{\aff}{\mathrm{aff}}

\DeclareMathOperator{\wt}{wt}

\DeclareMathOperator{\sk}{sk}

\newcommand{\llp}{(\!(}
\newcommand{\rrp}{)\!)}
\newcommand{\llb}{\llbracket}
\newcommand{\rrb}{\rrbracket}

\newcommand{\lL}{\mathbf{L}}

\newcommand{\ord}{\mathrm{ord}}

\DeclareMathOperator{\Cent}{Z}

\DeclareMathOperator{\map}{map}
\DeclareMathOperator{\Map}{Map}

\DeclareMathOperator{\GKM}{GKM}

\DeclareMathOperator{\Sch}{Sch}

\DeclareMathOperator{\Sh}{Sh}

\DeclareMathOperator{\Coh}{Coh}
\DeclareMathOperator{\Hom}{Hom}

\DeclareMathOperator{\RG}{R\Gamma}

\DeclareMathOperator{\Bl}{Bl}

\DeclareMathOperator{\CAlg}{CAlg}
\DeclareMathOperator{\IndSch}{IndSch}

\let\lim\relax
\DeclareMathOperator*{\lim}{lim}
\DeclareMathOperator*{\colim}{colim}

\let\ker\relax
\DeclareMathOperator{\ker}{ker}
\DeclareMathOperator{\coker}{coker}

\newcommand{\Ga}{\mathbb{G}_a}

\newcommand{\eI}{\mathscr{I}}

\newcommand{\rot}{\mathrm{rot}}

\newcommand{\uni}{\mathrm{uni}}

\DeclareMathOperator{\dL}{\mathcal{L}}

\DeclareMathOperator{\id}{id}

\DeclareMathOperator{\Spec}{Spec}
\DeclareMathOperator{\red}{red}
\DeclareMathOperator{\chara}{char}

\DeclareMathOperator{\A}{\mathbb{A}}
\let\P\relax
\DeclareMathOperator{\P}{\mathbb{P}}
\DeclareMathOperator{\Gr}{\mathscr{G}r}
\DeclareMathOperator{\Iw}{\mathrm{Iw}}
\DeclareMathOperator{\Fl}{\mathscr{F}l}

\DeclareMathOperator{\g}{\mathfrak{g}}
\let\t\relax
\DeclareMathOperator{\t}{\mathfrak{t}}

\newcommand{\D}{D}

\newcommand{\eN}{\mathscr{N}}

\newcommand{\sr}{\mathrm{s}}

\DeclareMathOperator{\Stab}{Stab}

\DeclareMathOperator{\ad}{ad}

\DeclareFontFamily{OT1}{pzc}{}
\DeclareFontShape{OT1}{pzc}{m}{it}{ <-> s*[1.2] pzcmi7t }{}
\DeclareMathAlphabet{\mathpzc}{OT1}{pzc}{m}{it}
\newcommand{\mm}{\mathpzc{m}}

\DeclareMathOperator{\Fix}{Fix}

\newcommand{\Gm}{\mathbb{G}_m}

\newcommand{\GL}{GL}
\newcommand{\PGL}{PGL}

\newcommand{\lo}{\mathrm{L}}
\newcommand{\lop}{\mathrm{L^+}}

\newcommand{\lol}{\mathrm{L}_{\ell}}
\newcommand{\lopl}{\mathrm{L^+_{\ell}}}

\let\O\relax
\DeclareMathOperator{\O}{\mathscr{O}}
\DeclareMathOperator{\eL}{\mathscr{L}}
\DeclareMathOperator{\eE}{\mathscr{E}}

\DeclareMathOperator{\eG}{\mathscr{G}}
\DeclareMathOperator{\eP}{\mathscr{P}}

\DeclareMathOperator{\eX}{\mathscr{eX}}

\DeclareMathOperator{\eZ}{\mathscr{Z}}

\DeclareMathOperator{\slant}{slt}

\DeclareMathOperator{\Z}{\mathbb{Z}}
\DeclareMathOperator{\Q}{\mathbb{Q}}
\DeclareMathOperator{\N}{\mathbb{N}}
\DeclareMathOperator{\C}{\mathbb{C}}

\DeclareMathOperator{\pr}{pr}
\DeclareMathOperator{\ev}{ev}

\DeclareMathOperator{\Pic}{Pic}

\DeclareMathOperator{\rank}{rank}

\DeclareMathOperator{\pt}{pt}

\newcommand{\qcqs}{\mathrm{qcqs}}

\newcommand{\ver}{\mathrm{ver}}

\newcommand{\act}{\mathrm{act}}

\DeclareMathOperator{\Perf}{Perf}

\newcommand{\loc}{\gamma}
\newcommand{\bloc}{\upsilon}

%% file: 0_introduction.tex
\section{Introduction}\label{section: introduction}

\subsection{Aims of this paper}

The aim of this paper is to revisit fixed-point localization techniques in equivariant topological $K$-theory and employ them for an effective description of the equivariant $K$-theory ring of the affine Grassmannian. The fibers of this $K$-theory ring are given by the cohomology of varying fixed-points. Using this viewpoint, we describe $K^{\top, 0}_{T \times \Gm^{\rot}}(\Gr_G; \C)$ as a suitable deformation to normal cone (alternatively an affine blowup algebra).

The cohomology of the same varying fixed points is known to match centers of various categories $\O$ in pure representation theory. Remarkably, all this information is packaged in a single $K$-theory ring, which we completely describe in terms of topological generators. In the joint companion work \cite{HL25}, we relate the present results to classical and quantum representation theory.

Let us now explain our main results in more detail.

\subsubsection{Localization techniques in equivariant $K$-theory via derived loop stacks.}

To set up the stage, we revisit the relationship between equivariant topological $K$-theory, equivariant periodic cyclic homology, functions on derived fixed-point schemes and singular cohomology of the varying fixed points. This stems from classical fixed-point localization techniques in equivariant topology and their reinterpretation in terms of equivariant Hochschild homology along the lines of \cite{Chen20}.

\begin{motto}
The fibers of torus-equivariant $K$-theory are given by singular cohomology of the fixed-points. In other words, equivariant $K$-theory puts the cohomologies of varying fixed points into an algebraic family. Globally, equivariant $K$-theory is given by equivariant periodic cyclic homology, which has a geometric model via functions on the derived fixed-point scheme.  
\end{motto}

Combining different techniques from the literature, we arrive at the following comprehensive statement.

\begin{theorem}\label{introtheorem: localization in equivariant K-theory}
Let $T \acts X$ be a qcqs scheme over $\C$ with an action of a torus $T$. Then we have natural multiplicative $\O(T)\llp u \rrp$-linear identifications
\begin{equation*}
    K^{\top}_{T}(X; \C) \xrightarrow{\simeq} HP_{T}(X / \C) \xrightarrow{\simeq} \RG(\Fix^{\lL}_{\frac{T}{T}}(X), \O)^{tS^1}.
\end{equation*}
Taking derived fibers over a closed point $t \in T$, this equivalence specializes to
\begin{equation*}
    K^{\top}_{T}(X; \C)_{\lL t} \xrightarrow{\simeq} H_{\sing}(\Fix_{t}(X); \C)\llp u \rrp.
\end{equation*}
If $X$ is equivariantly formal, the final identification is underived, and in particular gives an isomorphism $K^{\top, 0}_{T}(X; \C)_{t} \cong H^{2\bullet}_{\sing}(\Fix_{t}(X); \C)$ of commutative rings.
\end{theorem}
\begin{proof}
\S \ref{section: K-theory and HP and cohomology of fixed points}; Theorem \ref{theorem: fibers of topological K-theory, HP and fixed-points} and Corollary \ref{corollary: K-theory in the flat situation}. 
\end{proof}

\subsubsection{Fixed-points on affine Grassmannian and  affine Schubert varieties.}
In particular, we will be interested in the action of the extended torus $T \times \Gm^{\rot}$ on the affine Grassmannian $\Gr_G$ associated to a reductive group $G$ over $k=\C$.

We can fully compute the varying fixed-points of $(x, \zeta) \in T \times \Gm^{\rot}$ in $\Gr_G$, and in fact in any affine Schubert variety $\Gr_{\leq \mu}$ inside it. These examples are equivariantly formal -- via Theorem \ref{introtheorem: localization in equivariant K-theory} above, we can effectively compute the corresponding fibers of equivariant $K$-theory.

\begin{motto}
    The fixed points of an arbitrary element $(x, \zeta)$ in the extended torus on the affine Grassmannian $\Gr_G$ are given as an infinite union of partial flag varieties (if $\zeta$ generic) resp. finite union of affine flag varieties (if $\zeta$ root of unity) for a resonant subgroup $L \leq G$. The fixed points on an affine Schubert variety $\Gr_{\leq \mu}$ inside are given by a closed union of Borel orbits resp. closed union of Iwahori orbits.
\end{motto}

Generalizing well-known computations, we first describe the reduced fixed-points $\Fix^{\red}_{(x, \zeta)}(\Gr_G)$ on the whole affine Grassmannian. We construct an associated resonant subgroup $L = L_{[x, \zeta]} \leq G$ and show that these fixed-points are given as follows.
\begin{theorem}\label{introthm: fixed points on affine grassmannians}
The reduced fixed points $\Fix^{\red}_{(x, \zeta)}(\Gr_G)$ of $(x, \zeta)\in T \times \Gm^{\rot}$ on $\Gr_G$ are given as
\begin{enumerate}
    \item[(i)] an infinite disjoint union $\coprod_{\vartheta \in X_{\bullet}/(W_L, \bullet_{\omega})} L/P_{\vartheta}$ of finite flag varieties for $L$, if $\zeta$ is generic,
    \item[(ii)] a finite disjoint union of dilated affine flag varieties for $L^{\sc}$, namely $\coprod_{\vartheta \in X_{\bullet}/(W_L, \bullet^{\ell}_{\omega})} \lo_{\ell} L^{\sc}/\eP^{\vartheta}_{\ell, L^{\sc}}$, if $\zeta$ is a root of unity of primitive order $\ell$.
\end{enumerate}
\end{theorem}
\begin{proof}
\S \ref{section: zeta generic} and \S \ref{section: zeta root of unity}; Theorems \ref{lemma: fixed points of (c, zeta) on Gr_G -- zeta generic} and \ref{lemma: fixed points of (c, zeta) on Gr_G -- zeta root of unity}.    
\end{proof}

\begin{remark}
When $x=1$ and $\zeta$ is generic, this statement is classical \cite[\S 2]{Zhu15}. When $x=1$ and $\zeta$ is a root of unity of primitive order $\ell$, this was observed in \cite[\S 4]{RW22}, \cite[\S 2]{BBSV22}. The general description of Theorem \ref{introthm: fixed points on affine grassmannians} was missing; we expect it to be useful beyond the applications in this paper. 
\end{remark}

Carefully intersecting the above with given $\Gr_{\leq \mu}$, we arrive at the following result.
\begin{theorem}
Take any affine Schubert variety $\Gr_{\leq \mu}$.
Let $(x, \zeta) \in T \times \Gm^{\rot}$ and $L=L_{[x, \zeta]} \leq G$ the corresponding resonant subgroup. Then the reduced fixed point scheme $\Fix^{\red}_{(x, \zeta)}(\Gr_{\leq \mu})$ is given by
\begin{enumerate}
    \item[(i)] a closed union of of finite Schubert varieties for $L$, if $\zeta$ is generic,
    \item[(ii)] a closed union of Iwahori orbits on $\ell$-dilated affine flag varieties $\Fl^{\vartheta}_{\ell, L^{\sc}}$ of $L^{\sc}$, if $\zeta$ is a root of unity of primitive order $\ell$.
\end{enumerate}
In both cases, we can concretely describe which finite Schubert varieties resp. closure of Iwahori orbits actually appear.
\end{theorem}
\begin{proof}
\S \ref{section: zeta generic} and \S \ref{section: zeta root of unity}; Theorems \ref{theorem: generic zeta -- fiber of fixed points and classical schubert varieties} and \ref{theorem: fiber of fixed points and classical schubert varieties general case at root of unity}.    
\end{proof}

As we explain in the companion work \cite{HL25}, this geometric statement perfectly matches known representation-theoretical phenomena in classical and quantum categories $\O$. This was, in fact, our original motivation in \cite{HL25} for conjecturing such a description in the first place. It is remarkable how accurately the resulting comparison holds.

\subsubsection{Description of equivariant $K$-theory ring of the affine Grassmannian.}
Using fixed-point localization together with the description of fixed-points on the whole $\Gr_G$ from Theorem \ref{introthm: fixed points on affine grassmannians}, we compute the $T \times \Gm^{\rot}$-equivariant complexified $K$-theory ring of the affine Grassmannian in terms of the completed ring of functions on a certain scheme ${\eN}_{\mu_{\infty}}(\eZ, T \times (T \sslash W))$ given by a gluing of deformations to normal cones of quasi-diagonals $Z_{\ell}$ inside $T \times T \sslash W$ over each root of unity $\zeta$.

\begin{motto}
The complexified equivariant $K$-theory of the affine Grassmannian is given by the completed ring of functions on infinitely many affine blowups of $T \times \Gm^{\rot} \times T\sslash W$ over each root of unity $\zeta$, centered at the $\ell$-torsion subscheme $Z_{\ell}$ of $T \sslash W$.    
\end{motto}

Such a description was originally conjectured by Roman Bezrukavnikov; it is motivated by the analogous description of equivariant cohomology from \cite[Theorem 1]{BF07}. After properly accounting for the completion, we arrive at the following. 
\begin{theorem}\label{introtheorem: bezrukavnikov description of equivariant K-theory: sc}
Let $G$ be a simply connected reductive group over $k=\C$. Then there is a functorial ring isomorphism
\begin{equation*}
   \widehat{\O}({\eN}_{\mu_{\infty}}(\eZ, T \times (T \sslash W))) \xrightarrow{\cong} K^{\top, 0}_{T \times \Gm^{\rot}}(\Gr_G; \C).
\end{equation*}
\end{theorem}
\begin{proof}
\S \ref{section: bezrukavnikov--finkelberg description of equivariant K-theory}; Theorem \ref{theorem: equivariant K-theory as deformation to normal cone -- G sssc}.    
\end{proof}

In particular, the ring $K^{\top, 0}_{T \times \Gm^{\rot}}(\Gr; \C)$ is not finitely generated, even before completing. Nevertheless, it is quite explicit, as the following remark shows.

\begin{remark}
Unraveling Theorem \ref{introtheorem: bezrukavnikov description of equivariant K-theory: sc} in terms of commutative algebra, the $K$-theory ring of the affine Grassmannian $K^{\top, 0}_{T \times \Gm^{\rot}}(\Gr_G; \C)$ is given by the completion of
\begin{equation}\label{equation: explicit K-theory introduction remark}
    \C[t_1^{\pm 1}, \dots, t_r^{\pm 1}][s^{\pm 1}_1, \dots, s^{\pm 1}_r]^W \big[\tfrac{e_j(t_1^{\ell}, \dots, t_r^{\ell}) - e_j(s_1^{\ell}, \dots, s_r^{\ell})}{1-q^{\ell}} \mid \ell \in \N, j = 1, \dots, r\big],
\end{equation}
with respect to an explicit Schubert filtration.
More precisely, the uncompleted ring \eqref{equation: explicit K-theory introduction remark} naturally surjects onto the equivariant $K$-theory of each affine Schubert variety $\Gr_{\leq \mu}$; the equivariant $K$-theory of the whole $\Gr_G$ is then given by completing along the induced filtration by kernels. 

In particular, we get a concrete countable infinite set of topological generators, 
\begin{equation*}
    b_{j, \ell} := \tfrac{e_j(t_1^{\ell}, \dots, t_r^{\ell}) - e_j(s_1^{\ell}, \dots, s_r^{\ell})}{1-q^{\ell}}, \qquad j =1, \dots, r, \ \ell \in \N
\end{equation*}
which we call \textit{fundamental Bezrukavnikov classes}. 
These generators have an interpretation in terms of the tautological vector bundle and are related by Adams operations. We deduce that $K^{\top, 0}_{T \times \Gm^{\rot}}(\Gr_G; \C)$ is finitely generated as a complete topological $\lambda$-ring, see Corollary \ref{corollary: topological Adams generation}. 
\end{remark}

\begin{remark}
The above description goes against some naive analogies with cohomology: the equivariant cohomology of the affine Grassmannian with complex coefficients is a polynomial (or power-series) ring on finitely many generators.

Let us explain this apparent discrepancy. In their computation of cohomology, \cite{BF07} pass to one affine blowup over the origin. (In our terms, this is due to the fact that equivariant cohomology sees only the formal completion of equivariant $K$-theory at the origin $(1,1) \in T \times \Gm^{\rot}$.) In any case, the result of this single blowup is still coincidentally isomorphic to a polynomial ring (on shifted generators). This final coincidence breaks down once we are forced to blow up infinitely many times.
\end{remark}

\begin{remark}
We want to emphasize that many interesting classes -- such as the class of the \textit{determinant line bundle} $\eL_{\det}$ -- are only present in the completion; see Lemma \ref{lemma: necessity of Schubert completion}. In other words, the Schubert completion is important and nontrivial.
\end{remark}

\begin{remark}\label{remark: variant of the K-theory of affine Grassmannian}
Our method is quite robust. In particular, it can be used to describe
    \begin{enumerate}
        \item[(1)] the corresponding result for any reductive $G$ with respect to $T^{\sc} \times \Gm^{\rot}$, 
        \item[(2)] the result with respect to an isogenous torus $T$ to $T^{\sc}$,
        \item[(3)] the $\lop G \rtimes \Gm^{\rot}$-equivariant variant,
        \item[(4)] equivariant $K$-theory of arbitrary affine flag varieties.
    \end{enumerate}    
We address these generalizations in \S \ref{section: variants and generalizations}, see in particular Theorem \ref{theorem: equivariant K-theory as deformation to normal cone -- general case}.
\end{remark}

\begin{remark}
We believe that our method may be further adapted to describe
    \begin{enumerate}
        \item[(5)] the equivariant $K$-theory with integral coefficients.
    \end{enumerate}
We will discuss this elsewhere; one description will be given in \cite{HL25}.
\end{remark}

\begin{remark}
Even for reductive $G$, it is sometimes preferable to consider the action of $G^{\sc}$ on $\Gr_G$. This has several reasons.
\begin{itemize}
    \item The equivariant $K$-theory with respect to this action is what usually appears as centers in representation theory of quantum groups; also see \cite{CK18} for such appearance.
    \item When $G$ is not simply connected, the quotient $T \sslash W$ is often singular \cite[\S 7]{BZ00}. On the other hand, $T^{\sc} \sslash W$ is always isomorphic to the affine space $\A^d = \t \sslash W$. 
    \item The simply connected variant is compatible with the theory of Kac--Moody groups. In particular, the ample generator of the Picard group of $\Gr_G$ may not be $G$-equivariant otherwise \cite[\S 3.2]{CK18}; see \cite[Remarks 4.3 and 5.8]{YZ09} for more details.
\end{itemize}
\end{remark}

\begin{remark}
The results of this paper concern topological $K$-theory of $\Gr_G$. However, we expect them to hold algebraically on the level of homotopy-invariant algebraic $K$-theory $KH$.

In \cite{Low25}, we already answered this affirmatively in type A. Namely, we showed that for affine Schubert varieties $\Gr_{\leq \mu}$ in the $GL_n$ affine Grassmannian, $KH$ is equivariantly formal and its degree zero part agrees with topological $K$-theory:
$KH^0_{T \times \Gm^{\rot}}(\Gr_{\leq \mu}) = K^{\top, 0}_{T \times \Gm^{\rot}}(\Gr_{\leq \mu})$.
Hence the results of the present paper work algebraically in this case.
\end{remark}

\begin{remark}\label{question: decompletion of K-theory}
We view the ring of functions $\O({\eN}_{\mu_{\infty}}(\eZ, T \times (T \sslash W)))$ as a useful \emph{decompletion} of $K^{\top, 0}_{T \times \Gm^{\rot}}(\Gr_G; \C)$.
However, it is not clear to us what is its $K$-theoretical meaning -- can this decompletion be realized from algebraic $K$-theory?
\end{remark}

\subsection{Companion work}
In a companion joint work in progress \cite{HL25}, we describe the tight relationship between both decompleted resp. actual equivariant $K$-theory ring $K^{\top, 0}_{T \times \Gm^{\rot}}(\Gr_G)$, and the centers of a suitable integral quantum group resp. its quantum category $\O$. We precisely pin down this quantum group. The results of \cite{HL25} work with $\Z$-coefficients, and give a new genuine application of $K$-theoretical techniques in representation theory.

\subsection{Plan of the paper}
After the current \S \ref{section: introduction}, this paper is structured as follows.

In \S \ref{section: groups and schemes} we recall some preliminary material on algebraic geometry, reductive groups, and affine flag varieties.

In \S \ref{section: equivariant K-theory and fixed points} we set up notations regarding equivariant $K$-theory, related invariants and their algebraic counterparts. The nonstandard content is our interpretation of equivariant localization in \S \ref{section: K-theory and HP and cohomology of fixed points}, in particular its formulation in Theorem \ref{theorem: fibers of topological K-theory, HP and fixed-points}.

In \S \ref{section: cocharacter graphs and resonance} we introduce the notion of resonance for $(x, \zeta) \in T \times \Gm^{\rot}$ and consequently the resonant subgroup $L_{[x, \zeta]} \leq G$. We then discuss centralizers and stabilizers attached to the point $(x, \zeta)$. We further define the {cocharacter graphs} $\Gamma$ inside $T \times \Gm^{\rot} \times T$ and explain how they put the resonant combinatorics into a family over $T \times \Gm^{\rot}$. There are no deep results, but the discussion in this section serves as a useful tool in the rest of the paper.

In \S \ref{section: fixed points on affine grassmannians}, we fully describe the varying fixed-points $\Fix^{\red}_{(x, \zeta)}(-)$ for the whole affine Grassmannian $\Gr_G$ in Theorems \ref{lemma: fixed points of (c, zeta) on Gr_G -- zeta generic} and \ref{lemma: fixed points of (c, zeta) on Gr_G -- zeta root of unity}, as well as their restriction to affine Schubert varieties $\Gr_{\leq \mu}$ inside it in Theorems \ref{theorem: generic zeta -- fiber of fixed points and classical schubert varieties} and \ref{theorem: fiber of fixed points and classical schubert varieties general case at root of unity}. These results are of independent interest.

In \S \ref{section: equivariant K-theory ring of the affine grassmannian} we move on to compute the equivariant $K$-theory ring $K^{\top, 0}_{T \times \Gm^{\rot}}(\Gr_G; \C)$ in terms of infinitely many affine blowups at roots of unity $\zeta \in \Gm^{\rot}$. The main computation for simply connected $G$ is done in Theorem \ref{theorem: equivariant K-theory as deformation to normal cone -- G sssc}. Generalizations of this result are then covered in Theorem \ref{theorem: equivariant K-theory as deformation to normal cone -- general case}.

Appendix \ref{appendix: some explicit formulas in type A} provides some concrete localization formulas in $GL_n$ and discussion of the determinant line bundle.

\subsection{Acknowledgements}

\subsubsection{Acknowledgements.}
I would like to thank to Tamás Hausel for discussions regarding both this project and our joint ongoing work relating equivariant $K$-theory to centers of quantum groups.
At the same time, I would like to thank to Roman Bezrukavnikov for sharing and discussing his conjectural expectations on equivariant topological $K$-theory of the affine Grassmannian.

Moreover, I would like to thank to Vasily Krylov for pointing out a gap in an earlier version related to appearance of pseudo-Levi subgroups in computations of fixed points, Anton Mellit for suggesting a description of the relevant $K$-theory classes in terms of the tautological vector bundle and Adams operations, Catharina Stroppel for explaining to us at an early stage of this project the appearance of cohomology rings of classical flag varieties and finite Schubert varieties as endomorphism algebras in the classical category $\O$.
Finally, in connection to the companion work on quantum groups, I would like to thank Quan Situ for discussing extensions of his work on quantum groups and centers of deformed categories $\O$, and Josef Svoboda for helping me notice the original independent appearance of hybrid quantum groups in quantum topology.

\subsubsection{Funding.}
This work is a revised version of the results covered in the third part of author's PhD thesis \cite[Chapter III]{LowPhD} at the Institute of Science and Technology Austria (ISTA). The results were obtained between February 2025 -- February 2026; further clarifications were implemented until September 2026. The author was supported by the DOC Fellowship of the Austrian Academy of Sciences and by the ERC grant ViaFiPoS, No 101199663.

\subsubsection{Declaration of the use of generative AI.}
The author has not used any AI tools in the writing of this article. All mathematical content was obtained independently, based on author's expertise, discussions with his colleagues and available literature.

%% file: 1_groups_and_representations.tex
\section{Groups and schemes}\label{section: groups and schemes}
This section serves as a recollection of concepts from algebraic geometry and representation theory which we freely use later in the paper.

\subsection{Schemes}

\subsubsection{Base rings.}
Let $k$ be a base ring for our geometry. The most relevant case is when $k$ is an algebraically closed field of $\chara k = 0$, for example $k=\C$. The reader is invited to focus on this case.

On the other hand, let $\k$ be an arbitrary ring of coefficients. Again, one may want to take $\k =\C$ for simplicity. Nevertheless, a majority of our results works over fields $\k$ of arbitrary characteristic. Even better, most of our constructions work over $\k = \Z$.

\subsubsection{General conventions.}
Given a scheme $X$, we write $\O(X)$ for the ring of global sections of its structure sheaf $\O_X$.
We write $|X|$ for the set of closed points of $X$ -- we prefer to use this concise notation for closed points only for the sake of readability; this should not lead to any confusion.

In general, we write $\coprod$ for actual coproduct (disjoint union) of schemes. On the other hand, we write $\bigsqcup$ for locally closed stratifications.

\subsubsection{Closed subschemes and fibers.}
Given a closed subscheme $Z \hookrightarrow X$, we write $X - Z$ for the open complement. The ideal sheaf of $Z$ in $X$ is denoted by $\eI_{Z}$; it sits in the short exact sequence
\begin{equation*}
    0 \to \eI_{Z} \to \O_X \to \O_Z \to 0.
\end{equation*}
We further write $X_{Z, n} = {\Spec}_{\O_X}(\O_X / \eI^n_Z) \hookrightarrow X$ for the $n$-th infinitesimal thickening of $Z$ inside $X$. We denote $X^{\wedge}_Z := \lim_{n} X_{Z, n}$ the formal completion of $X$ along $Z$.

Given a map of schemes $X \to Y$ and a point $y \in Y$, we write $X_y$ for the fiber. Furthermore, we denote by $X^{\wedge}_{y}$ the formal completion of $X$ around $y$. By definition, $X^{\wedge}_y := X^{\wedge}_{X_y}$ is the completion of $X$ along the fiber $X_y$.

\subsubsection{Quotients.}
Let $G$ be a group scheme. Suppose $G \acts X$. We write $X/G$ for the stack quotient (taken in the fpqc topology). 

On the other hand, we denote by $X \sslash G$ the GIT quotient (in general this depends on further choices). We only consider GIT quotients for $X = \Spec R$ affine, when $X\sslash G$ is also affine with ring of functions $\O(X \sslash G) = \O(X/G) = \O(X)^G= R^G$.

\subsubsection{Ind-schemes.}
Consider the Yoneda embedding $\Sch_k \hookrightarrow \Sh(\Sch_k)$ into the category of sheaves (say of sets) on $\Sch_k$ (say in the étale topology). This embedding commutes with all limits, but not with colimits. The category $\Sch_k$ of $k$-schemes is in fact complete, but not cocomplete. In other words, colimits of schemes may only exists as objects of the bigger category $\Sh(\Sch_k)$. 

In particular, suppose $X$ is given as a filtered colimit $X = \colim_{i \in I} X_i$ of schemes $X_i$ along closed immersions. Such $X$ are called \textit{ind-schemes}, and span a full subcategory $\IndSch_k$. Any filtered system $(X_i, i \in I)$ as above together with a chosen isomorphism $X = \colim_{i \in I} X_i$ is called an \textit{ind-presentation} of $X$. Clearly, $\Sch_k \subseteq \IndSch_k$ via constant systems.

Many construction and properties of schemes and their morphisms extend to ind-schemes by checking on a ind-presentation. In particular, the notions of ind-proper and ind-closed immersion make sense. 

\subsubsection{Functions on ind-schemes.}
The global functions on an ind-scheme are given by the inverse limit of global functions on any ind-presentation
\begin{equation*}
    \O(X) := \Hom(X, \A^1) = \Hom(\colim_{i\in I} X_i, \A^1) = \lim_{i \in I}(X_i, \A^1) = \lim_{i \in I} \O(X_i).
\end{equation*}
Similarly for other representable functors in place of $\A^1$.

\begin{remark}
    Given a closed immersion of ind-schemes $Z \hookrightarrow X$, the induced map on global functions $\O(X) \to \O(Z)$ need not be surjective. For an illustrative example, let $X = \A^1_k$ over a field $k$ of characteristic zero and $Z = \coprod_{n \in \Z} \pt = \colim_{i \in \N} \coprod_{n \in [-i, i] } \pt$, with the embedding $Z \hookrightarrow X$ given by the discrete set of integer points in $\A^1(k)$. This is clearly a closed ind-subscheme, since the map at each finite stage is a closed immersions of schemes. However, the induced map on global functions $\O(\A^{1}) \to \prod_{n \in \Z} k$ is not surjective: not every function $\Z \to k$ may be interpolated by a polynomial.

    Given a closed immersion $Z \hookrightarrow X$ of ind-schemes such that $Z$ is a scheme, the map $\O(X) \to \O(Z)$ is surjective. 
\end{remark}

\subsection{Deformations to normal cones and affine blowups}\label{section: deformations to normal cones and affine blowups}

\subsubsection{Blowups.}\label{section: blowups}
Given a scheme $X$ with a closed subscheme $Z \hookrightarrow X$, we denote $\Bl_Z(X)$ the blowup of $X$ at $Z$. It has a natural map $\pi: \Bl_Z(X) \to X$. The fiber $E:=\pi^{-1}(Z) \subseteq \Bl_Z(X)$ is an effective Cartier divisor. The pair $(\Bl_Z(X), E)$ is the universal pair mapping to $X$ characterized by the property that $E \hookrightarrow \Bl_Z(X)$ is an effective Cartier divisor \cite[Tag 085U]{Sta}.

In particular $\Bl_X(X) = \emptyset$. In an ambient scheme $X$, suppose we are given closed subschemes $Z \hookrightarrow X_1, X_2 \hookrightarrow X$. Inside the ambient blowup $\Bl_Z(X)$, we then have
\begin{equation*}
    \Bl_{Z}(X_1) \cap \Bl_{Z}(X_2) = \Bl_{Z}(X_1 \cap X_2).
\end{equation*}

\subsubsection{Strict transforms.}\label{section: strict transforms}
See \cite[Tag 080C]{Sta} for a general discussion. Consider a closed subscheme $Z \hookrightarrow X$, and the associated blowup $\Bl_Z(X)$. Given any other closed subscheme $Y \hookrightarrow X$, we have its strict transform
\begin{equation*}
    Y^{\st} \hookrightarrow \Bl_Z(X).
\end{equation*}
which may be described as the scheme-theoretic closure of the preimage of $Y \cap (X - Z)$, namely
\begin{equation*}
    Y^{\st} = \overline{\pi^{-1}(Y \cap (X - Z))} \hookrightarrow \Bl_Z(X).
\end{equation*}
In terms of $Y$, the strict transform is given as
\begin{equation*}
    Y^{\st} = \Bl_{Y \cap Z} Y \hookrightarrow \Bl_{Z} X.
\end{equation*}
Here, the embedding comes from the functoriality of blowups. See \cite[Tag 080E]{Sta}.

\subsubsection{Deformations to normal cones.}
See \cite[\S 5]{Ful84}.
Let $Z \hookrightarrow X$ be a closed subscheme. Consider the embedding $Z \times \{0\} \hookrightarrow X \times \A^1$. We can consider the blowup $\Bl_{Z \times \{0\}}(X \times \A^1)$. The \textit{deformation of $X$ to the normal cone of $Z$} is then given by the open complement to the strict transform of $X \times \{0\}$, namely
\begin{equation*}
    N(Z, X) := \Bl_{Z \times \{0\}}(X \times \A^1) - (X \times \{0\})^{\st}.
\end{equation*}
It is an affine flat scheme over $\A^1$. Denoting by $C(Z, X)$ the normal cone of $Z \subseteq X$, meaning the relative spectrum of the sheaf of algebraic differentials $\Omega_{Z/X}$, the deformation $N(Z, X)$ fits in the pullback diagram
\begin{equation*}
    \begin{tikzcd}
      C(Z, X) \arrow[d] \arrow[r]&  N(Z,X) \arrow[d] & \arrow[l] \arrow[d] X \times (\A^1 - \{ 0 \}) \\
      \{0\} \arrow[r] & \A^1 & \arrow[l] \A^1 - \{ 0 \}. 
    \end{tikzcd}
\end{equation*}
In words, $N(Z, X)$ carries a structure map $N(Z,X) \to \A^1 = \Spec(k[h])$, whose fibers over $h \neq 0$ are isomorphic to $X$, while the fiber over $h=0$ is given by the normal cone $C(Z, X)$ of $Z$ inside $X$.

More explicitly, the structure sheaf $\O_X$ has a descending filtration $(\eI^j_{Z})_{j = 0}^{\infty}$ by powers of the ideal sheaf $\eI_Z$. The associated Rees construction
\begin{equation*}
 \O_X[\tfrac{f}{h^j} \mid f \in \eI^j_Z]    
\end{equation*}
turns it into a graded algebra, whose relative spectrum gives the above deformation:
\begin{equation*}
 N(Z, X) = \Spec_{\O_X}(\O_X[\tfrac{f}{h^j} \mid f \in \eI^j_Z]).  
\end{equation*}

\subsubsection{Basic properties.}

The deformation to the normal cone is functorial in the input morphism $Z \to X$.
It is further compatible with pullbacks: given $Z \hookrightarrow X$ and $Y \to X$, we have a natural identification $N(Z, X) \times_X Y = N(Z\times_X Y, Y)$.
If a finite group $W$ acts on $Z \hookrightarrow X$, it also acts on $N(Z, X)$ and we have $N(Z \sslash W, X\sslash W) = N(Z, X) \sslash W$.

Suppose we are given a finite set of generators of the ideal sheaf $\eI_Z = \langle f_1, \dots, f_m  \rangle$. Then 
\begin{equation*}\label{section: generators of the deformation}
    \O_X[\tfrac{f}{h^j} \mid f \in \eI^j_Z] = \O_X[\tfrac{f_i}{h} \mid i = 1, \dots, m]. 
\end{equation*}
In other words, to generate the structure sheaf of $N(Z, X)$, we only need to add finitely many ring generators
$\tfrac{f_i}{h}$ for $i = 1, \dots, m$.

\subsubsection{Affine blowups.}\label{subsubsection: affine blowups.}
More generally, suppose we have a family of schemes $X \to \A^1$ over the base $\A^1$. Let $\zeta \in \A^1$ be a point, and suppose $Z \hookrightarrow Y_{\zeta}$ is a closed subscheme of the fiber of $X$ at $\zeta$.
Then the \textit{affine blowup of $X$ at $Z$ with respect to $X \to \A^1$}, denoted $\Bl^{\circ}_Z(X)$ is the open complement in $\Bl_{Z}(X)$ of the strict transform of $X_{\zeta}$, namely
\begin{equation*}
    \Bl^{\circ}_Z(X) := \Bl_Z(X) - X^{\st}_{\zeta}.
\end{equation*}
We have a natural affine morphism $\Bl^{\circ}_Z(X) \to X$.
By design, the deformation to the normal cone is a special case: $N(Z, X) = \Bl^{\circ}_{Z \times \{0\}}(X \times \A^1)$.

\subsubsection{Strict transforms inside affine blowups.}\label{section: embedding of strict transform into affine blowup}
Assume that we have another closed subscheme $Y \hookrightarrow X$ such that $Z = Y \cap X_{\zeta}$. Then the strict transform of $Y$ inside $\Bl_Z(X)$ lies inside the affine blowup, namely
\begin{equation*}
    Y^{\st} \hookrightarrow \Bl^{\circ}_{Z}(X).
\end{equation*}
Indeed, it is enough to show that $Y^{\st}$ does not intersect $X^{\st}_{\zeta}$. Using the properties from \S \ref{section: blowups}, \ref{section: strict transforms}, this follows from
\begin{equation*}
  Y^{\st} \cap  X^{\st}_{\zeta} = \Bl_{Z}(Y) \cap \Bl_Z(X_{\zeta}) = \Bl_{Z}(Y \cap X_{\zeta}) = \Bl_{Z}(Z) = \emptyset.
\end{equation*}

\subsubsection{Deformations centered at a set of points.}\label{subsubsection: deformations centered at a set of points}
We will often perform the affine blowup construction over a set of closed points of $\Gm \hookrightarrow \A^1$. Namely, suppose we have a scheme $X$ and consider $X \times \Gm$. Suppose we are given a set $(z_i \mid i \in I)$ of pairwise distinct closed points of $\Gm$ and a set of closed subschemes $\eZ = (Z_i \subseteq X \mid i \in I)$. We regard $Z_i = Z_i \times \{z_ i\} \hookrightarrow X \times \Gm$ as a closed subscheme.

We the may blow up independently each $Z_i$ inside $X \times \Gm$; we denote the result by
\begin{equation*}
    N_{q=z_i}(Z_{z_i}, X) := \Bl^{\circ}_{Z_{i} \times \{z_i\}}(X \times \Gm^{\rot})
\end{equation*}
where the subscript $q=z_i$ emphasizes that the deformation is centered at the point $z_i$ of $\Gm = \Spec(k[q^{\pm 1}])$. 

We denote the associated gluing of deformations over the set of points $(z_i \mid i \in I)$ by
\begin{equation*}
    \eN_{I}(\eZ, X) := \Bl^{\circ}_{\eZ}(X).
\end{equation*}
This is defined as the affine blowup of $X \times \A^1$ centered in $\coprod_{i \in I} Z_{i} \times \{z_i\} \hookrightarrow X \times \A^1$. However, note the when $I$ is infinite, the center of the blowup is merely an ind-scheme. 

Explicitly, if $I$ is finite, this gluing is given by inductively blowing up at $Z_i \times \{ z_i \}$ in some order; the order does not matter as these are disjoint. For infinite $I$, one takes the limit over all finite subsets of $I$ -- these are schemes, hence their limit also is.

\subsection{Reductive groups}
\subsubsection{Reductive groups.}\label{subsubsection: reductive groups}
Let $G$ be a connected reductive group over an algebraically closed field $k$. 
Denote by $T$ its universal maximal torus; we also write $T = T_G$ whenever we need to stress the ambient group. For the sake of concreteness, we fix an embedding $T \leq G$. We write $X_{\bullet} = X_{\bullet}(T)$ resp. $X^{\bullet}=X^{\bullet}(T)$ for the cocharacter resp. character lattice of $T$. There is a natural identification $T(k) = X_{\bullet}(T) \otimes _{\Z} k^{\times}$.

We choose a Borel $B \leq G$ containing $T$ and denote by $X^{+}_{\bullet} = X^{+}_{\bullet}(G)$ the set of dominant cocharacters with respect to $B$.
We denote $B' \leq G$ the Borel subgroup opposite to $B$; all parabolics and parahorics will be taken with respect to $B'$.

\subsubsection{Chevalley forms.}
We use the same notation as in \S \ref{subsubsection: reductive groups} over the ring of coefficients $\k$ for our cohomology theories; for example we may take $\k = \Z$. In particular, we denote by $G := G_{\k}$ the Chevalley form over $\k$. We further denote by $T := T_{\k}$ its universal maximal torus, $B := B_{\k}$ a chosen Borel subgroup and so on. We omit the subscript $\k$ when there is no risk of confusion. All reductive groups in this paper are split.

\subsubsection{Root data.}
We write $\Phi_{\bullet} \subseteq X_{\bullet}$ resp. $\Phi^{\bullet} \subseteq X^{\bullet}$ for the coroots resp. roots of $G$. They respectively generate the coroot lattice $Q_{\bullet} := \Z\Phi_{\bullet} \leq X_{\bullet}$ and root lattice $Q^{\bullet} :=\Z\Phi^{\bullet} \leq X^{\bullet}$.
Altogether, $(X_{\bullet}, \Phi_{\bullet}, X^{\bullet}, \Phi^{\bullet})$ denotes the root datum of $G$; its automorphism group is the Weyl group $W$.
Whenever the ambient reductive group needs to be emphasized, we also write $X_{\bullet, G}$, $\Phi_{\bullet, G} $, $X^{\bullet}_G$, $\Phi^{\bullet}_G$, $W_G$ and so on.

We write $\langle -, - \rangle$ for the natural pairing $X_{\bullet} \times X^{\bullet} \to \Z$. We denote by $(-)^\vee: \Phi^{\bullet} \to \Phi_{\bullet}$, $\beta \mapsto \widecheck{\beta}$ the unique bijection such that $\langle \widecheck{\beta}, \beta \rangle = 2$ for each root $\beta \in \Phi^{\bullet}$. Denote further $\Phi_{\bullet}^{+}$ resp. $\Phi^{\bullet}_{+}$ the sets of positive coroots and positive roots with respect to $B$. We let $\Phi_{\bullet}^{\sr} = \{\widecheck{\beta}_1, \dots, \widecheck{\beta}_d \}$ resp. $\Phi^{\bullet}_{\sr} = \{\beta_1, \dots, \beta_d \}$ denote the subsets of simple coroots and simple roots. 

We write $\leq$ for the usual partial order on $X_{\bullet}$. By definition, $\lambda' \leq \lambda$ if and only if $\lambda - \lambda' \in \Z_{\geq 0} \Phi^{+}_{\bullet}$. Similarly in the dual case of $X^{\bullet}$.

We denote $X_{\bullet,\Q} := X^{\bullet}(T) \otimes_{\Z} \Q$ and $X^{\bullet}_{\Q} := X^{\bullet}(T) \otimes_{\Z} \Q$.

\subsubsection{Langlands dual and its representations.}\label{section: lnglands dual and its representations}
If $G$ is a reductive group with a root datum $(X_{\bullet}, \Phi_{\bullet}, X^{\bullet}, \Phi^{\bullet})$, its Langlands dual is the reductive group $\widecheck{G}$ with the flipped root datum given by $(\widecheck{X}_{\bullet}, \widecheck{\Phi}_{\bullet}, \widecheck{X}^{\bullet}, \widecheck{\Phi}^{\bullet}) := (X^{\bullet}, \Phi^{\bullet}, X_{\bullet}, \Phi_{\bullet})$.

Given any $\mu \in X^{+}_{\bullet} = \widecheck{X}^{\bullet}_+$, we denote $V_{\mu}$ the corresponding irreducible representation of $\widecheck{G}$ of highest weight $\mu$. 
We further denote 
\begin{equation}\label{notation: weights in mu}
    \wt(\mu) := W \cdot \{ \lambda \in X^{+}_{\bullet} \mid \lambda \leq \mu \} \subseteq X_{\bullet}
\end{equation}
the union of $W$-orbits of elements $\lambda \in X^{+}_{\bullet}$ with $\lambda \leq \mu$. In terms of $\widecheck{G}$, it is the subset of characters of nonzero weight spaces in $V_{\mu}$.

\subsubsection{Weyl groups.}
Denote $W = N_G(T)/T$ the Weyl group of $T$ in $G$. Equivalently, it is the Weyl group of the root system $(X_{\bullet}, \Phi_{\bullet}, X^{\bullet}, \Phi^{\bullet})$. It acts on $T$ and all the above objects in the standard way. 

Given any root $\beta \in \Phi^{\bullet}$, we denote the corresponding reflection
\begin{equation*}
    s_{\beta}: X_{\bullet} \to X_{\bullet}, \qquad \lambda \mapsto \lambda - \langle \beta, \lambda \rangle \widecheck{\beta}.
\end{equation*}
It induces an automorphism  of $T$, in multiplicative notation given by
\begin{equation*}
s_{\beta}: T \to T, \qquad s_{\beta}(c) = c \cdot \widecheck{\beta}(\beta(c))^{-1}.    
\end{equation*} 
The reflections $s_{\beta}$ induce elements of the Weyl group $W$. Moreover, $W$ has the structure of a Coxeter group, generated by the simple reflections $s_{\beta}$ for $\beta \in \Phi^{\bullet}_{s}$.

\subsubsection{Simply connected covers and adjoint groups.}\label{subsubsection: simply connected covers and adjoint groups}
Denote $G^{\der}$ the derived group of $G$. This is a semisimple group. It is often useful to consider the finitely many (semisimple) groups isogeneous to $G^{\der}$. 
In particular, let $G^{\sc}$ be the simply connected cover of $G^{\der}$; its maximal torus is denoted $T^{\sc}$. Similarly, let $G^{\ad}$ be the adjoint group of $G^{\der}$, with maximal torus $T^{\ad}$. The adjoin group is given directly by the quotient of $G$ by its center: $G^{\ad} = G/Z(G)$. Concisely, we have the diagram
\begin{equation*}
\begin{tikzcd}
G^{\sc} \arrow[d] \arrow[rd]\\
G^{\der} \arrow[r, hook] \arrow[d] & G \arrow[ld, two heads]\\
G^{\ad}
\end{tikzcd}    
\end{equation*}
The groups in the left column are all semisimple; the vertical maps between them are central isogenies. Each connected reductive group isogenous to $G^{\der}$ sits in between $G^{\sc}$ and $G^{\ad}$.
When $G=G^{\sc}$, we say $G$ is simply connected group (implicitly assuming it is reductive); similarly when $G=G^{\ad}$ is adjoint.

The above diagram also induces corresponding maps on the universal maximal tori, and consequently on other root theoretic data. In particular, these maps naturally identify the root resp. coroot lattices $\Z\Phi^{\bullet}$ resp. $\Z\Phi_{\bullet}$ of all the groups; we will implicitly use these identifications. 

We write $X^{\sc}_{\bullet} = X_{\bullet}(T^{\sc})$ and $X_{\sc}^{\bullet} = X^{\bullet}(T^{\sc})$ for the cocharacter and character lattice of $T^{\sc}$, as well as $X_{\bullet}^{\ad} = X_{\bullet}(T^{\ad})$ and $X_{\ad}^{\bullet} = X^{\bullet}(T^{\ad})$ for the cocharacter and character lattices of $T^{\ad}$, and so on.
For any cocharacter $\lambda \in X_{\bullet}$ of $T$, we write $\lambda^{\ad} \in X^{\ad}_{\bullet}$ for the cocharacter of $T^{\ad}$ given by the image of $\lambda$ under the natural projection $T \twoheadrightarrow T^{\ad}$.

\subsubsection{Algebraic fundamental group.}\label{section: borovoi fundamental group}
Denote
\begin{equation}
 \pi_1(G) := X_{\bullet} / X^{\sc}_{\bullet} = X_{\bullet} / \Z\Phi_{\bullet}  
\end{equation}
the \textit{algebraic fundamental group of $G$}, see \cite[Definition 1.3]{Bor98}. It recovers the usual fundamental group of the topological space $G(\C)$ over $\C$. We have $G=G^{\sc}$ if and only if $\pi_1(G)=1$. The natural action of $W$ on $\pi_1(G)$ is trivial and the fundamental group $\pi_1(G)$ is functorial in $G$; see \cite[Lemma 1.2 and its proof]{Bor98}.
Given $\lambda \in X_{\bullet}$, we denote by $\overline{\lambda} \in \pi_1(G)$ the corresponding class.

We choose a family of representatives for elements of the fundamental group under the surjection $X_{\bullet}(T) \twoheadrightarrow \pi_1(G)$: given an element $\sigma \in \pi_1(G)$, let $ \chi(\sigma) \in X_{\bullet}(T)$ be the chosen representative.

\subsubsection{Shifted actions.}\label{section: shifted actions}
Given $\alpha \in X_{\bullet}$, we denote $\bullet_{\alpha}$ the $\alpha$-shifted action
\begin{align*}
    (W, \bullet_{\alpha}) & \acts X_{\bullet} \\
    w \bullet_{\alpha} \mu & = w(\mu - \alpha) + \alpha, \qquad \forall \ w \in W.
\end{align*}
It identifies with the natural action $W \acts \alpha + X_{\bullet}$ of $W$ on the $\alpha$-shifted copy of $X_{\bullet}$.
When $\alpha = \rho$, this is the \emph{dot action}. We use the same notation for the induced action on $T$.

This notation naturally extends to the following setup. Assume that $\omega \in X^{\ad}_{\bullet}$ is a cocharacter of the adjoint group $G^{\ad}$. We then pick any lift $\alpha \in X_{\bullet, \Q}$ of $\omega$ under the natural projection $X_{\bullet, \Q} \twoheadrightarrow X^{\ad}_{\bullet, \Q}$. 
The associated $\alpha$-shifted action $(W, \bullet_{\alpha}) \acts X_{\bullet, \Q}$ on the vector space $X_{\bullet, \Q}$, given by $w \bullet_{\alpha} \mu = w(\mu - \alpha) + \alpha = w(\mu) - w(\alpha) + \alpha$, preserves the subset $X_{\bullet}$ and is independent of the choice of the lift $\alpha$. 
Indeed, it is enough to check that for any $w \in W$, the element $\alpha - w(\alpha)$ lies in $X_{\bullet}$ and is independent of the choice of $\alpha$. By induction on the length of $w$, both statements may be checked on simple reflections, where they are clear since $\alpha - s_{\beta}(\alpha) = \langle \beta, \alpha \rangle \widecheck{\beta} = \langle \beta, \omega \rangle \widecheck{\beta}$.

Altogether, given $\omega \in X_{\bullet}^{\ad}$, we obtain a well-defined action
\begin{equation*}\label{equation: action shifted by adjoint cocharacter}
    (W, \bullet_{\omega}) \acts X_{\bullet}. 
\end{equation*}
given by the same formula with respect to any lift of $\omega$.


\subsubsection{Levi subgroups.}\label{section: levi subgroups}
Given a cocharacter $\lambda: \Gm \to G$, we denote
$L_{\lambda}$ the centralizer of the image of $\lambda$ inside $G$. 
Then $L_{\lambda} \leq G$ is a Levi subgroup; it is in particular connected reductive. We denote by $P_{\lambda} \leq G$ the corresponding parabolic subgroup with respect to $B'$ (containing $L_{\lambda}$ and $B'$). See \cite[\S 21.i and Theorem 13.33]{Mil17}.

We call $L_{\lambda} \leq P_{\lambda}$ the Levi and parabolic associated to $\lambda$; note that this construction depends on the ambient reductive group $G$. 
Note that $T \leq L_{\lambda}$ is a maximal torus and $X_{\bullet} = X_{\bullet}(T)$ is the cocharacter lattice of $L$. The intersections of $B$ resp. $B'$ with $L$ define a pair of opposite Borel subgroups $B_L$ resp. $B'_L$ of $L$. 

For any reductive group $G$, the Levi subgroups and parabolic subgroups of $G$ and $G^{\ad}$ are in natural bijection. Consequently, we may consider the subgroups $L_{\omega}$, $P_{\omega}$ of $G$ labeled by $\omega \in X^{\ad}_{\bullet}$.

\subsubsection{Coordinates and invariant polynomials.}\label{subsubsection: coordinates and invariant polynomials}
Let $G$ be an affine group scheme over a field $k$.
The representation ring of $G$ is denoted $R(G)$. Given a coefficient ring $\k$, its $\k$-linearization is denoted $R(G;\k):= R(G) \otimes_{\Z} \k$.

Assume $G = G$ is a split reductive group over $k$ of rank $n$ with maximal torus $T$. 
We then have its Chevalley form $G_{\k}$ over $\k$, and we denote $T_{\k}$ is maximal torus. The representation ring of $G$ is given by
\begin{equation*}
    R(G) = \O(T_{\Z} \sslash W) \qquad \text{and} \qquad R(G; \k) = \O(T_{\k} \sslash W).
\end{equation*}
As a module, $R(G)$ is generated by classes of irreducible representations of $G$, denoted $V_{\mu}$ for $\mu \in X^{\bullet}_{+}$. For ring generators, one may take the $[V_{\mu_j}]$ with $\mu_j$ for $j =1, \dots, n$, where $V_j$ are chosen generators of $X^{\bullet}_{+}$. We denote these $n$ elements by
\begin{equation*}
    e_j := e_j^{G} := [V_{\mu_j}] \qquad j = 1, \dots, n.
\end{equation*}
In particular:
\begin{itemize}
\item 
If $G=T$ to start with, we denote these coordinates by
$\O(T_{\Z}) = \Z[X^{\bullet}] = \Z[t_1^{\pm 1}, \dots, t_n^{\pm 1}]$.
In a coordinate free way, we also denote by $t_{\nu}$ the coordinate coming from the character $\nu$. 
\item
If $G = GL_n$, we have $\O(T_{\Z} \sslash W) = \Z[e_1, \dots, e_{n-1}, e_n^{\pm 1}]$. Here, $e_j$ is the $j$-th elementary symmetric polynomial on $T_{\Z}$.
\item
For general $G$, interpreting any $V_{\mu}$ as a group morphism $\rho_{\mu}: G \to \GL_n$ for a suitable $n$, we get the above classes as pullbacks of the standard representation: $e^{G}_j = r^*_{\mu_j} (e_1)$.
\item 
For $G$ simply connected, we take $e_j := [V_{\omega_j}]$ the classes of the fundamental representations with highest weights $\omega_j$, for $j = 1, \dots, n$.
\end{itemize}

\begin{remark}\label{remark: representation ring in the simply conneted case}
For reductive $G$, the ring $\O(T_{\Z} \sslash W)$ may be subtle. Namely, when $\pi_1(G) \neq 1$, it will usually acquire finite quotient singularities -- see \cite[\S 7]{BZ00} for the case of $G = \PGL_3$.

On the other hand, for $G$ simply connected, the above ring is a polynomial ring: $\O(T_{\Z} \sslash W) = \Z[e_1, \dots, e_n]$ with distinguished coordinates given by the fundamental representations $e_j \in \O(T \sslash W)$.
\end{remark}

\subsubsection{Lie algebras.}
Given a split reductive group $G$ over $\k$, we denote its Lie algebra by $\g$. In particular, the Lie algebra of its maximal torus $T$ is denoted $\t$. When $G$ is simply connected, we have $T \sslash W = \t \sslash W$, an affine space from Remark \ref{remark: representation ring in the simply conneted case}.

\subsection{Loop groups and affine flag varieties}\label{section: loop groups and affine flag varieties}

\subsubsection{Affine root data.}
We denote by $\Phi^{\bullet}_{\aff} = \Phi^{\bullet} \oplus \Z\hbar$ the set of {\it affine real roots} associated to the root system of $G$. We write elements of this set as $\beta + m \hbar$ with $\beta \in \Phi^{\bullet}$ and $m \in \Z$.

\subsubsection{Affine and extended Weyl groups.}
Let $\Z\Phi_{\bullet} \leq X_{\bullet}$ be the coroot lattice inside the cocharacter lattice.
We denote by
\begin{equation*}
W_{\aff} := \Z\Phi_{\bullet} \rtimes W
\qquad
\text{and}
\qquad
W_{\ext} := X_{\bullet} \rtimes W
\end{equation*}
the \textit{affine Weyl group} and \textit{extended affine Weyl group} associated to the root system of $G$.
These semi-direct products are constructed from the natural action of $W$ on $\Phi_{\bullet}$ resp. $X_{\bullet}$. They thus have preferred splittings, allowing us to regard $W$ as a subgroup.
Clearly, $W_{\aff} \leq W_{\ext}$ is a normal subgroup; the quotient identifies as $W_{\ext} / W_{\aff} = \pi_1(G)$. In particular $W_{\ext} = W_{\aff}$ if $G$ is simply connected.
Also note that $W_{\aff}$ only depends on the isogeny class of $G$, since this is the case for both $W$ and $\Phi_{\bullet}$. On the other hand, $W_{\ext}$ depends on the actual group $G$.

The image of a coroot $\widecheck{\beta} \in \Phi_{\bullet}$ is denoted $\tau_{\widecheck{\beta}}:= (\widecheck{\beta}, 1) \in W_{\aff}$. We also have the finite reflections $s_{\beta} \in W \leq W_{\aff}$ labeled by $\beta \in \Phi^{\bullet}$.
The affine Weyl group $W_{\aff}$ has the structure of a Coxeter group, with simple reflections given by
\begin{equation*}
    s_{\beta, m} := (\tau_{m\widecheck{\beta}}, s_{\beta}) \in \Z\Phi_{\bullet} \rtimes W, \qquad \beta \in \Phi_{\bullet}, \ m \in \Z. 
\end{equation*}
Similar discussion applies to $W_{\ext}$.

\subsubsection{Shifted affine Weyl group actions.}\label{section: shifted affine weyl group actions}
Consider the following family of actions of $W_{\aff}$ on $X_{\bullet}$.
Let $\ell \in \N$ be a natural number and $\alpha \in X_{\bullet}$ any cocharacter. Then we define the \emph{$\ell$-dilated $\alpha$-shifted action}
\begin{align*}
    (W_{\aff}, \bullet^{\ell}_{\alpha}) & \acts X_{\bullet} \\
    (\lambda, w ) \bullet^{\ell}_{\alpha} \mu & = w(\mu - \alpha - \ell \lambda) + \alpha, \qquad \forall \lambda \in \Z\Phi_{\bullet}, \ w \in W
\end{align*}

For $\ell = 1$ and $\alpha$ trivial, this is the natural of $W_{\aff}$ on $X_{\bullet}$. For general $\ell$, it is the \emph{$\ell$-dilated action}. For $\alpha = \rho$ and general $\ell$, we get the \emph{$\ell$-dilated dot action}. We use the same notation for the induced action on $T$.

Given any cocharacter $\omega \in X^{\ad}_{\bullet}$ of the adjoint group $G^{\ad}$, we obtain the $\omega$-shifted action
\begin{equation*}
   (W_{\aff}, \bullet^{\ell}_{\omega})  \acts X_{\bullet} 
\end{equation*}
given by the same formula with respect to any representative of $\omega$ inside $X_{\bullet, \Q}$. It is well-defined by the argument from \S \ref{section: shifted actions}.

\subsubsection{Loop rotation torus.}
We denote by $\Gm^{\rot}$ the one dimensional torus with coordinate ring $\O(\Gm^{\rot}) = k[q^{\pm 1}]$ and call it the \textit{loop rotation torus}. It plays a special role, so we include the superscript $\rot$ to distinguish from other tori. We denote its closed points by $\zeta \in \Gm^{\rot}$.

For $\ell \in \N$, we write $\mu_{\ell} \leq \Gm^{\rot}$ for the subgroup scheme of order $\ell$ roots of unity. We write $\mu_{\infty} := \colim_{\ell \in \N} \mu_{\ell} \leq \Gm^{\rot}$ for the ind-subgroup scheme of all roots of unity. This gives a decomposition $\Gm^{\rot} = (\Gm^{\rot} - \mu_{\infty}) \sqcup \mu_{\infty}$.

A closed point $\zeta \in \Gm^{\rot}$ is called {\it generic} if it is not a root of unity, i.e. if it is torsionfree element of the underlying multiplicative group. Otherwise, $\zeta \in \mu_{\infty}$ is a root of unity.
For any $\zeta \in \Gm^{\rot}$, we write
\begin{equation*}
\ord(\zeta) =
\begin{cases}
\ell, \qquad \text{if} \ \zeta \in \mu_{\infty} \ \text{is a root of unity of primitive order} \ \ell \in \N, \\
0, \qquad \text{if} \ \zeta \in \Gm^{\rot} - \mu_{\infty} \ \text{is generic}.
\end{cases}
\end{equation*}
We write $\Sigma_{\ell} \hookrightarrow \mu_{\ell}$ for the cyclotomic subscheme of roots of unity of primitive order $\ell$. 

Denoting $\hbar$ the real affine root in the Lie algebra sense, our group coordinate $q$ is given as $q = e^{\hbar}$. Since we work multiplicatively throughout, we do not need the exponential notation.

\subsubsection{Loop groups with loop rotation.}
Given a connected reductive group $G$ over $k$, we denote by $\lo G$ its loop group and by $\lop G$ its positive loop group, with functors of points
\begin{equation*}
    \lo G: R \mapsto G(R\llp t \rrp) \qquad \text{and} \qquad \lop G: R \mapsto G(R\llb t \rrb).
\end{equation*}
We have $\lop G \leq \lo G$. The group of connected components of the loop group $\lo G$ identifies with $\pi_1(G)$, while the positive loop group $\lop G$ is connected. We have an embedding $G \leq \lop G \leq \lo G$ of the original reductive group $G$ via constant loops.

The torus $\Gm^{\rot}$ acts on $\lo G$ by the loop rotation 
\begin{align*}
\Gm^{\rot} \times \lo G \to \lo G  , \\
(\zeta, g(t)) \mapsto g(\zeta t).
\end{align*}
This action clearly restricts to $\lop G$. 
The extended loop groups are given by the semidirect products along loop rotation
\begin{equation*}
\lop G \rtimes \Gm^{\rot} \leq \lo G \rtimes \Gm^{\rot}.   
\end{equation*}
These semidirect products become trivial after restriction to the constant loops $G \leq \lop G$.

Given $\ell \in \N$, we write $\lo_{\ell} G \leq \lo G$ and $\lopl G \leq \lop G$ the $\ell$-dilated subgroups; they represent the subfunctors
\begin{equation*}
    \lol G: R \mapsto G(R\llp t^{\ell} \rrp) \qquad \text{and} \qquad \lopl G: R \mapsto G(R\llb t^{\ell} \rrb).
\end{equation*}
While abstractly isomorphic to the usual loop groups, the geometry of their embeddings is useful.
For $\ell = 0$, we denote by convention $\lo_0 G = \mathrm{L^+_{0}} G = G$, embedded in $\lo G$ as constant loops.

\subsubsection{Extended tori.}\label{subsubsection: extended tori}
In particular, we have the extended torus $T \times \Gm^{\rot} \leq \lo G \rtimes \Gm^{\rot}$ embedded via constant loops.
Many computations will take place within this ambient extended loop group.
For notational convenience, we pick coordinates
\begin{equation*}
    \O(T\times \Gm^{\rot}) = k[t_1^{\pm 1}, \dots, t_n^{\pm 1}][q^{\pm 1}].
\end{equation*}

Any character of $X^{\bullet}(T)$, and in particular any root $ \beta \in \Phi^{\bullet} \leq X^{\bullet}(T)$, can be tautologically regarded as a function $\beta \in \O(T)$ given by the composition $\iota \circ \beta: T \to \Gm \to \A^1$, where $\iota: \Gm^{\rot} \hookrightarrow \A^1$ is the obvious inclusion.
Similarly, any character of $X^{\bullet}(T\times \Gm^{\rot})$, and in particular any affine root $\beta + m \hbar$, can be tautologically regarded as a function $\beta + m\hbar \in \O(T \times \Gm^{\rot})$ given by the composition $\iota \circ (\beta + m\hbar): T \times \Gm^{\rot} \to \Gm \to \A^1$.
It is sometimes customary to denote these functions as
\begin{equation*}
   e^{\beta} = \beta \in \O(T) \qquad \text{and} \qquad e^{\beta + m\hbar} = e^{\beta} \cdot e^{m\hbar} = {\beta} \cdot q^m \in \O(T \times \Gm^{\rot}).
\end{equation*}
Since we are working multiplicatively, we will avoid this exponential notation.

On the other hand, given any cocharacter $\lambda \in X_{\bullet}(T)$, regarded as a homomorphism $\lambda: \Gm \to T$, we denote $t^{\lambda} := \lambda(t) \in \lo T$ the image of the formal variable $t$ under the map $\lo \lambda: \lo \Gm \to \lo T$. More generally, given any element $a \in k \llb t \rrb$, we denote $a^{\lambda} := \lambda(a) \in \lo T$ its image under $\lo \lambda$.

\subsubsection{Root spaces in loop groups.}\label{paragraph: root spaces in loop groups}
Given an affine root $\beta + m\hbar \in \Phi^{\bullet} \oplus \Z \hbar$, we get the homomorphism $\Ga \to \lo G$ given by $x \mapsto u_\beta(t^mx)$ and denote $U_{\beta + m\hbar} \leq \lo G$ the corresponding root subgroup of $\lo G$ given by its image; see \cite[\S 4.3]{RW22}.

\begin{lemma}\label{lemma: weights of extended torus on affine root groups}
The adjoint action of $T \times \Gm^{\rot}$ restricts to each $U_{\beta + m\hbar}$, where it has weight
\begin{equation*}
(x, \zeta) \mapsto \beta(x) \cdot \zeta^m.  
\end{equation*}
\end{lemma}
\begin{proof}
Given any point $(x, \zeta) \in T \times \Gm^{\rot}$, compute
\begin{align*}
   (x, \zeta) \cdot (u_{\beta}(z t^m), 1) \cdot (x, \zeta)^{-1} 
   = (x, 1) \cdot (u_{\beta}(\zeta^m \cdot zt^m), 1) \cdot (x^{-1}, 1)
   = (u_{\beta}(\beta(x) \zeta^m \cdot z t^m), 1).
\end{align*}
\end{proof}

\begin{lemma}\label{lemma: conjugation of affine root spaces}
For any $\lambda \in X_{\bullet}$ we have 
\begin{equation*}
 t^{-\lambda} \cdot U_{\beta + m\hbar} \cdot t^{\lambda} = U_{\beta + (m- \langle \lambda, \beta \rangle)\hbar}
\end{equation*}
\end{lemma}
\begin{proof}
    See \cite[Lemma 4.3. (2)]{RW22}.
\end{proof}

\subsubsection{Parahoric subgroups.}\label{subsubsection: parahoric subgroups}
The preimage of the Borel subgroup $B' \leq G$ under the evaluation map $\lop G \to G$ gives rise to the standard Iwahori subgroup $\Iw \leq \lop G$. Given any proper subset $J$ of the extended Dynkin diagram of $G$, we denote the associated standard parahoric subgroup $\eP_J$. We have $\Iw \leq \eP_J \leq \lop G$. To avoid too many subscripts, we sometimes put $J$ in the superscript and write $\eP^J := \eP_J$.

A parahoric subgroup $\eP_{J}$ sits in a canonical short exact sequence
\begin{equation*}
1 \to \eP_{J}^{\uni} \to \eP_{J} \to L_{J} \to 1    
\end{equation*}
where $\eP_{J}$ is its \textit{pro-unipotent radical} and $L_{J}$ is its reductive \textit{Levi quotient}. 
For all questions regarding equivariant topological invariants, the difference between $\eP_{J}$ and $L_{J}$ is irrelevant.

For more details on parahoric subgroups, see \cite[\S 2.1.5]{Yun15} and references there.

\subsubsection{Affine flag varieties.}
Given a parahoric subgroup $\eP \leq \lop G$, the associated affine flag variety is the fpqc quotient
\begin{equation*}
    \Fl_{\eP} = \lo G / \eP = \lo G \rtimes \Gm^{\rot} / \eP \rtimes \Gm^{\rot}.
\end{equation*}
It is an ind-projective ind-scheme. It carries a left transitive action of $\lo G \rtimes \Gm^{\rot}$; we often restrict the above action to various subgroups of $\lop G \rtimes \Gm^{\rot}$, such as $T \times \Gm^{\rot}$.

When $\eP= \eP_J$ is the standard parahoric, we also write $\Fl^J := \Fl_{\eP_J}$.
For $\eP = \lop G$, we have the affine Grassmannian
\begin{equation*}
    \Gr_{G} = \lo G / \lop G = \lo G \rtimes \Gm^{\rot} / \lop G \rtimes \Gm^{\rot}.
\end{equation*}

For each $\lambda \in X_{\bullet}$, the element $t^{\lambda}$ gives a distinguished point on $\Fl_{\eP}$ and in particular on $\Gr_G$. These are precisely the fixed-points under the whole torus $T \times \Gm^{\rot}$. We call them the \emph{coordinate fixed-points}.

\subsubsection{Local Hecke stacks.}\label{section: local quantum hecke stacks}
Taking the quotient by the above action, we obtain
\begin{equation*}
   \lop G \rtimes \Gm^{\rot} \backslash \Fl_{\eP} = \lop G \rtimes \Gm^{\rot} \backslash \lo G / \eP = \lop G \rtimes \Gm^{\rot} \backslash \lo G \rtimes \Gm^{\rot} / \eP \rtimes \Gm^{\rot}.
\end{equation*}
In the case of affine Grassmannian, we call this quotient the \emph{quantum Hecke stack}
\begin{equation*}
   \lop G \rtimes \Gm^{\rot} \backslash \Gr_{G} = \lop G \rtimes \Gm^{\rot} \backslash \lo G / \lop G = \lop G \rtimes \Gm^{\rot} \backslash \lo G \rtimes \Gm^{\rot} / \lop G \rtimes \Gm^{\rot}.
\end{equation*}

\subsubsection{Affine Schubert varieties.}
The $\lop G \rtimes \Gm^{\rot}$ orbits on $\Gr_G$ are the affine Schubert cells $\Gr_{\mu}$, equipped with the reduced scheme structure. By the Iwasawa decomposition, they are labeled by the discrete set $X^{+}_{\bullet}$. The affine Schubert cell $\Gr_{\mu}$ is an affine bundle over the corresponding partial flag variety $G/P_{\mu}$.

The affine Schubert varieties $\Gr_{\leq \mu}$ are the closures of the affine Schubert cells, equipped with the reduced scheme structure. They have a locally closed decomposition into a disjoint union of the affine Schubert cells
\begin{equation*}
    \Gr_{\leq \mu} = \bigsqcup_{\mu' \leq \mu} \Gr_{\mu'}.
\end{equation*}
Each $\Gr_{\leq \mu}$ is a singular projective variety, acted on by $\lop G \rtimes \Gm^{\rot}$. We refer to \cite{Zhu15} for an excellent introduction.

\begin{remark}
The affine Grassmannian $\Gr_{G}$ has a natural affine paving. Indeed, we have an action $\Iw \acts \Gr_{G}$ of the Iwahori group scheme. Its fixed points are precisely $t^{\lambda}$ with $\lambda \in X_{\bullet}$. The action of $\Iw$ on $t^{\lambda}$ factors through a finite-dimensional quotient. Since the reductive quotient $T$ of $\Iw$ stabilizes $t^{\lambda}$, its orbit under $\Iw$ is an affine space. Consequently, we deduce that Iwahori orbits provide an affine paving of $\Gr_G$, and in particular of each affine Schubert variety $\Gr_{\leq \mu}$. 

Also see \cite[(5.2)]{HHH05} and references there.
\end{remark}

\subsubsection{Dilated variants.}\label{subsubsection: dilated variants}
We will use the above constructions in the \textit{$\ell$-dilated} context. Also, the ambient reductive group will in fact be a suitable pseudo-Levi subgroup $L$ of $G$. We spell this out in detail.

Let $\lambda \in X_{\bullet} = X_{\bullet}(T)$. Consider the subset $I_{\ell, \lambda}$ of affine reflections fixing $\lambda$. Denote $\eP^{\lambda}_{L, \ell} \leq \lo_{\ell} L$ the standard parahoric subgroup corresponding to $I_{\ell, \lambda}$ with respect to the opposite Borel $B'$ -- up to passage to the opposite Borel, see \cite[\S 2.1.4]{BBSV22}).
In terms of Bruhat--Tits theory, $\eP^{\lambda}_{L, \ell}$ is the parahoric subgroup of $\lo_{\ell} L$ corresponding to the facet $\mathrm{f}$ containing $\lambda$, see \cite{RW22}. 
Explicitly, by \cite[equation (2.8)]{BBSV22} together with our opposite Borel convention, it is given as
\begin{equation*}
    \eP^{\lambda}_{L, \ell} = \lo_{\ell} L \cap t^{\lambda} \cdot \lop L \cdot t^{-\lambda}.
\end{equation*}
See \cite[\S 4.2 -- \S 4.3]{RW22} and \cite[\S 2.1]{BBSV22}.
We denote 
\begin{equation*}
\Fl^{\lambda}_{L, \ell} := \lo_{\ell} L / \eP^{\lambda}_{L, \ell}    
\end{equation*}
the corresponding partial affine flag variety of the $\ell$-dilated loop group of $L$.
Note that these constructions depend only on the stabilizer of $\lambda$ in $W_{\ell, \aff}$.

\subsubsection{Functoriality.}
The constructions in \S \ref{section: loop groups and affine flag varieties} are functorial in group homomorphisms $G \to G'$ of reductive groups. Specifically, we have an induced map $\Gr_{G} \to \Gr_{G'}$ compatible with the equivariance, Schubert stratifications, and so on. Consequently, many constructions can be reduced to the case of $GL_n$.

\subsection{Root subsystems and centralizers}\label{section: root subsystems and centralizers}
In order to effectively manipulate centralizers of semisimple elements inside loop groups, we recall some technical material on subsystems of finite and affine root systems. We then briefly recall the structure of such centralizers. In the finite case, these are called pseudo-Levi subgroups. 
We work over an algebraically closed field $k$. For simplicity, we assume $k$ is of characteristic zero.

\subsubsection{Closed subsystems of root systems.}

By an \emph{abstract root system} $\Psi^{\bullet}$ we mean a root system of a Kac--Moody algebra $\g_{\Psi^{\bullet}}$ in the sense of \cite{Kac83}; the notations below make sense in this generality.

\begin{setup}\label{setup: finite and affine root systems}
For our purposes, we will assume that $\Psi^{\bullet}_1$ is one of the following.
\begin{itemize}
    \item[(i)] $\Phi^{\bullet}_G$ is a finite root system, realized in a reductive group $G$ with respect to the chosen maximal torus $T \leq G$. We denote $\g$ its Lie algebra and $\t \subseteq \g$ the corresponding Cartan subalgebra.
    \item[(ii)] $\Phi^{\bullet}_{G, \aff} = \Phi^{\bullet}_G \oplus \Z\hbar$ is the real untwisted affine root system, realized by the affine real roots in the loop group $\lo G$ of $G$ with respect to the extended torus $T \times \Gm^{\rot}$. We denote $\g_{\aff}$ its Lie algebra and $\t_{\aff} = \t \oplus \Ga^{\rot}$ the corresponding Cartan subalgebra.
\end{itemize}
\end{setup}

Following \cite[Definition 2.4.1]{RV19}, we recall some terminology. Also see \cite[\S 1.2]{DL11a}.
\begin{notation}\label{notation: closed root subsystems}
Let $\Psi^{\bullet}_1$ be any ambient root system.
\begin{enumerate}
\item[(i)] 
A subset $\Psi^{\bullet}_2 \subseteq \Psi^{\bullet}_1$ is called a \emph{root subsystem} if it is itself a root system. 
Explicitly, this means $\Psi^{\bullet}_2$ is closed on reflections in its elements: $\forall \alpha, \beta \in \Psi^{\bullet}_2$, we have $s_{\alpha}(\beta) \in \Psi^{\bullet}_2$.
\item[(ii)] 
A subset $\Psi^{\bullet}_2 \subseteq \Psi^{\bullet}_1$ is called \emph{closed} if $\forall \alpha, \beta \in \Psi^{\bullet}_2$, also $ \alpha + \beta \in \Psi^{\bullet}_2$.
\end{enumerate}
A subset satisfying both (i) and (ii) is called a \emph{closed root subsystem}.
\end{notation}

\subsubsection{Projection and affinization.}\label{section: projection and affinization}
Affine and finite root systems compare through the projection map 
\begin{equation}\label{equation: projection of root systems}
\Phi^{\bullet}_{G, \aff} \twoheadrightarrow \Phi^{\bullet}_G, \qquad \beta + m\hbar \mapsto \beta.    
\end{equation}
Any subset $\Psi^{\bullet} \subseteq \Phi^{\bullet}_{G, \aff}$ can be projected to $\overline{\Psi}^{\bullet} \subseteq \Phi^{\bullet}_G$. If $\Psi^{\bullet}$ is a root subsystem, then also its projection is a root subsystem \cite[Definition 2.4.3]{RV19}. Moreover, if $\Psi^{\bullet} \subseteq \Phi^{\bullet}_{G, \aff}$ is a closed root subsystem, so is $\overline{\Psi}^{\bullet}$ by \cite[Proposition 4.1.1]{RV19}.

Vice versa, any subset $\Psi^{\bullet} \subseteq \Phi^{\bullet}_G$ can be lifted to a subset $ \widehat{\Psi}^{\bullet} \subseteq \Phi^{\bullet}_{G, \aff}$ by taking preimage \cite[Definition 2.4.4]{RV19}. This operation clearly preserves root subsystems, as well as the closed condition.

\subsubsection{Regular subalgebras.} 
From now on, we restrict attention to closed root subsystems; these parametrize a particularly useful class of subalgebras.
\begin{notation}
Let $\g_{\Psi^{\bullet}_1}$ be a Kac--Moody algebra. A subalgebra $\g' \subseteq \g_{\Psi^{\bullet}_1}$ is called \emph{regular} if it is stable under the adjoint action of some Cartan subalgebra $\t_{\Psi^{\bullet}_1}$ of $\g_{\Psi^{\bullet}_1}$. 
See \cite[Definition 12.1.3 and text below]{RV19} and \cite[\S 1]{FRT08}, as well as \cite[Chapter II, \S 2]{Dyn57} for the case of classical Lie algebras.
\end{notation}

Fix a Cartan subalgebra $\t_{\Psi^{\bullet}_1}$ of $\g_{\Psi^{\bullet}_1}$.
To a closed subroot system $\Psi^{\bullet}_2 \subseteq \Psi^{\bullet}_1$, we can associate the regular subalgebra $\g_{\Psi^{\bullet}_2}$ generated by the root spaces $\g_{\alpha}$, for $\alpha \in \Psi^{\bullet}_2$ as in \cite[Definition 12.1.3 and below]{RV19}. 
In the relevant cases, these give all regular subalgebras with respect to the chosen Cartan.

\begin{proposition}\label{proposition: closed root subsystems and regular subalgebras}
Assume that $\Psi^{\bullet}_1$ is either the finite root system $\Phi^{\bullet}_G$ or the affine real root system $\Phi^{\bullet}_{G, \aff}$ as in Setup \ref{setup: finite and affine root systems}. Denote $\t_{\Psi^{\bullet}_1} \subseteq \g_{\Psi^{\bullet}_1}$ the chosen Cartan subgroup in its Lie algebra. Then the above induces a bijection
    \begin{equation*}
    \{ \Psi^{\bullet}_2 \subseteq \Phi^{\bullet}_{\Psi^{\bullet}_1} \ \text{closed root subsystem}  \} \leftrightarrow \{ \g_{\Psi^{\bullet}_2} \subseteq \g_{\Psi^{\bullet}_1}  \text{regular subalgebra with respect to} \t_{\Psi^{\bullet}_1} \}.
\end{equation*}
\end{proposition}
\begin{proof}
The case of finite root systems is classical, going back to \cite{Dyn57}, \cite{BD49}. For the case of affine real root systems, see \cite[Corollary 12.1.5]{RV19}.
\end{proof}

Also see \cite[Theorem 12.3.1]{RV19} for other equivalent characterization of regular subalgebras of affine Kac--Moody algebras in terms of combinatorial data.

\begin{notation}
Since $k$ is algebraically closed of characteristic zero, to each regular subalgebra of $\g$ resp. $\g_{\aff}$ we can associate the corresponding connected subgroup of $G$ resp. $\lo G$.
We call these \emph{regular subgroups}.
\end{notation}

By construction, Proposition \ref{proposition: closed root subsystems and regular subalgebras} also parametrizes regular subgroups. The projection and affinization of closed root subystems from \S \ref{section: projection and affinization} induce corresponding constructions on regular subalgebras and regular subgroups.

\subsubsection{Centralizers of semisimple elements.}
Regular subalgebras and closed subroot systems in particular arise from centralizers of semisimple elements in the ambient Kac--Moody group. See \cite[Chapter II, \S 6, No.21, Corollaries to Theorem 6.7]{Dyn57} for the finite case.
\begin{lemma}\label{lemma: centralizers give closed subroot systems}
Assume we are in the Setup \ref{setup: finite and affine root systems}. Let $G_{\Psi^{\bullet}}$ be the corresponding group and $T_{\Psi^{\bullet}}$ its maximal torus.
Let $y \in T_{\Psi^{\bullet}}$ be any closed point. Then its connected centralizer $\Cent^{\circ}_{G_{\Psi^{\bullet}}}(y)$ is a regular subgroup of $G_{\Psi^{\bullet}}$ with associated closed root subsystem 
\begin{equation*}
   \Psi^{\bullet}_{\Cent^{\circ}(y)} = \{ \beta \in \Psi^{\bullet} \mid \beta(y) = 1 \} \subseteq \Psi^{\bullet}.
\end{equation*}
\end{lemma}
\begin{proof}

Since we are working in characteristic zero and $\Cent^{\circ}(y)$ is connected, it is sufficient to describe its Lie algebra. This is clearly a regular subalgebra -- it is stable under the action of the Cartan subalgebra corresponding to the maximal torus $T$ containing $y$. Hence it comes from a closed root subsystem by Proposition \ref{proposition: closed root subsystems and regular subalgebras}. The subsystem is given by $\Psi^{\bullet}_{\Cent^{\circ}(y)}$, as can be checked root-by-root: the adjoint action of $y$ fixes the root space $U_{\beta}$ labeled by $\beta$ if and only if $\beta(y)=1$.
\end{proof}

\begin{remark}
It is also easy to check that $\Psi^{\bullet}_{\Cent^{\circ}(y)}$ is a closed root subsystem by hand. Indeed, assume $\alpha, \beta \in \Psi^{\bullet}_{\Cent^{\circ}(y)}$. Then $\alpha(y) = \beta(y) =1$.
To see that $\Psi^{\bullet}_{\Cent^{\circ}(y)}$ is a root subsystem, note that $s_{\alpha}(\beta) = \beta {-\langle \beta, \alpha^{\vee} \rangle}\cdot \alpha$. Hence also $s_{\alpha}(\beta)(y) = \beta(y) \cdot \alpha(y)^{-\langle \beta, \alpha^{\vee} \rangle} = 1$, showing $s_{\alpha}(\beta) \in \Psi^{\bullet}_{\Cent^{\circ}(y)}$. 
To see that $\Psi^{\bullet}_{\Cent^{\circ}(y)}$ is closed, note that $(\alpha + \beta)(y) = \alpha(y) \cdot \beta(y) = 1$, showing $\alpha + \beta \in \Psi^{\bullet}_{\Cent^{\circ}(y)}$.
\end{remark}

We review the structure of these centralizers in more detail in the two cases from Setup \ref{setup: finite and affine root systems} separately in \S \ref{section: finite case and pseudo-levi subgroups} and \S \ref{section: centralizers in the affine case} below.

\subsubsection{Finite case and pseudo-Levi subgroups.}\label{section: finite case and pseudo-levi subgroups}
We now specialize to the finite case.
Let $G$ be a reductive group over an algebraically closed field $k$. Given any semisimple element $x \in G$, we wish to recall the structure of its centralizer $\Cent_{G}(x)$. 
Let $T \leq G$ be a maximal torus. Clearly, $\Cent_G(x)$ depends only on the conjugacy class of $x$. Conjugating inside $G$, we may thus without loss of generality assume $x \in T$.

\begin{proposition}\label{proposition: pseudo-levi subgroups}
Let $G$ be a connected reductive group over and algebraically closed field $k$. Given any $x \in T$, we have the following.
\begin{enumerate}
    \item[(i)] The centralizer $\Cent_{G}(x) \leq G$ of $x$ in $G$ is a reductive group.
    \item[(ii)] It is generated by $T$, the root subgroups $U_{\beta}$ for those roots $\beta \in \Phi^{\bullet}$ such that $\beta(x)=1$, and those Weyl group representatives $n_w$ that commute with $x$.
    \item[(iii)] The identity component $\Cent^{\circ}_{G}(x)$ is a connected reductive group with root system determined by $(X^{\bullet}, \Phi_{\Cent^{\circ}(x)})$, where $X_{\bullet} = X_{\bullet}(T)$ and $\Phi^{\bullet}_{\Cent^{\circ}(x)} = \{ \beta \in \Phi^{\bullet}_G \mid \beta(x)=1\}$. Its Weyl group $W_x \leq W$ is the subgroup generated by the reflections $s_{\beta}$, $\beta \in \Phi^{\bullet}_{L_x}$.   
    \item[(iv)]  If $G$ is moreover simply connected, $\Cent_{G}(x)$ is connected.     
\end{enumerate}
\end{proposition}
\begin{proof}
See \cite[\S 2]{Hum95} for an excellent overview. 
Statements (i), (ii), (iii) are \cite[\S 2.2. Theorem]{Hum95}; (iv) is \cite[\S 2.11. Theorem]{Hum95}.
Also see \cite[\S 4, \S 5]{Ste68}, \cite[\S 4, in particular \S 4.1 and \S 4.2]{SS70}, \cite[\S 6]{MS03}.
\end{proof}

\begin{notation}
The connected centralizers $\Cent^{\circ}(x) \leq G$ for varying $x\in T$ are called \emph{pseudo-Levi subgroups} of $G$. These are regular subgroups of $G$.
\end{notation}

\begin{remark}
Let $G$ be a reductive group and $x \in T$.
\begin{enumerate}
\item[(1)] Assume $G$ is of type A. Then any pseudo-Levi $\Cent^{\circ}(x)$ is in fact a Levi subgroup of $G$. Its Weyl group $W_x$ is generated by a subset of simple reflections.
\item[(2)] In other types, $\Cent^{\circ}(x)$ might not be a Levi subgroup of $G$. See \cite[Table 1]{RV19}, \cite[\S 12.1]{Kan01}. for a complete list of pseudo-Levi subgroups that appear in a given $G$.
This classification was done in \cite{BD49}, by describing all (closed) subroot systems of $\Phi^{\bullet}_G$.
\item[(3)] If $G$ is not simply connected, the centralizer $\Cent(x)$ might not be connected. For example \cite[4.3.Examples]{SS70}, the centralizer of $(i, -i)$ in $PGL_2$ is given by two discrete points corresponding to monomial matrices.
\end{enumerate}  
\end{remark}

\subsubsection{Centralizers in the affine case.}\label{section: centralizers in the affine case}
In the affine case, the above discussion immediately gives the following.

\begin{lemma}\label{proposition: affine centralizer as regular subgroup}
Let $\lo G \rtimes \Gm^{\rot}$ be the loop group of a reductive group $G$ extended by loop rotation. Take $(x, \zeta) \in T \times \Gm^{\rot}$ and denote $\ell =\ord(\zeta) \geq 0$. 
The identity component $\Cent^{\circ}_{\lo G \rtimes \Gm^{\rot}}(x, \zeta)$ is a regular subgroup of $\lo G$ associated to the closed root subsystem $\Phi^{\bullet}_{\Cent^{\circ}(x, \zeta)} = \{ \beta + m\hbar \in \Phi^{\bullet}_{G, \aff} \mid \beta(x) \cdot \zeta^m = 1 \}$.
\end{lemma}
\begin{proof}
Firstly, note that 
\begin{equation*}
    \Cent^{\circ}_{\lo G \rtimes \Gm^{\rot}}(x, \zeta) = \Cent^{\circ}_{\lo G}(x, \zeta) \rtimes \Gm^{\rot}.
\end{equation*}
because for any $\zeta' \in \Gm^{\rot}$, the elements $(x, \zeta)$ and $(1, \zeta')$ lie in the same torus $T \times \Gm^{\rot}$ and thus commute. Hence we can restrict to the centralizer inside $\lo G$. 

Since we consider the connected centralizer, it is enough to describe its affine root spaces.
The conjugation by $(x, \zeta)$ acts on such root space $U_{\beta + m\hbar}$ by the scalar $\beta(x) \cdot \zeta^m$ as in Lemma \ref{lemma: weights of extended torus on affine root groups}.
The fixed points are the whole $U_{\beta + m\hbar}$ if $\beta(x) \cdot \zeta^m =1$ and a single point else. The statement follows.   
\end{proof}
We will return to Lemma \ref{proposition: affine centralizer as regular subgroup} and give a more detailed description of $\Cent^{\circ}_{\lo G}(x, \zeta)$ and it root system later in Lemma \ref{lemma: root spaces of J} and Lemmata.

\begin{remark}
Alternatively, \cite{RV19} classified all closed root subsystems of $\Phi^{\bullet}_{G, \aff}$, see in particular \cite[Proposition 4.1.1 and Theorem 4.2.1]{RV19}; also see \cite[Theorem 5.3]{FRT08}.
The computations above can be also bootstrapped from \cite[Proposition 3.2.1 with notation of \S 3.1]{RV19} 
\end{remark}

\begin{remark}
The closely related classification problem for all (not necessarily closed) root subsystems of $\Phi^{\bullet}_{G, \aff}$ was resolved much earlier \cite{Dye90, DL11a, DL11b}. The answer here is similar, but slightly more complicated \cite[Theorem 5.1 and \S 5.5]{Dye90}, \cite[Theorem 3]{DL11a}.
\end{remark}

\subsubsection{Stabilizers in Weyl groups as reflection subgroups.}\label{section: stabilizers in weyl groups as reflection subgroups}
Let $W_{\Psi^{\bullet}_1}$ denote the Weyl group of the ambient root system $\Psi^{\bullet}_1$. It has the structure of a Coxeter group, whose reflections are labeled by the elements of $\Psi^{\bullet}_1$.
\begin{notation}
A subgroup $W'$ of $W_{\Psi^{\bullet}_1}$ is called a \emph{reflection subgroup} if it is generated by a subset of reflections $S' \subseteq (s_{\beta} \mid \beta \in \Psi^{\bullet}_1)$.    
\end{notation}

\begin{remark}
Reflection subgroups of $W_{\Psi^{\bullet}_1}$ correspond to (not necessarily closed) root subsystems of ${\Psi^{\bullet}_1}$. Indeed, to such $W'$ we can associate the subset $\{\beta \in \Psi^{\bullet}_1 \mid s_{\beta} \in W'\}$, which is a root subsystem of $\Psi^{\bullet}_1$. Vice versa, to any root subsystem we can associate the reflection subgroup generated by reflections in its roots. This defines a bijection between root subsystems $\Psi^{\bullet}_2$ of $\Psi^{\bullet}_1$ and reflection subgroups $W_{\Psi^{\bullet}_2}$ of its Weyl group $W_{\Psi^{\bullet}_1}$. See \cite[\S 1.2]{DL11a}, \cite[\S 2]{FRT08}). 
\end{remark}

Assume we are in the Setup \ref{setup: finite and affine root systems}.
For simplicity, assume that $G=G^{\sc}$ is simply connected.
The Weyl group $W_{\Psi^{\bullet}_1}$ naturally acts on the maximal torus $T_{\Psi^{\bullet}}$ of the corresponding group -- in the two respective cases we have $W_G \acts T$ and $W_{G, \aff} \acts T \times \Gm^{\rot}$. The following result is classical.

\begin{proposition}\label{proposition: steinberg fixed point theorem}
Assume we are in the Setup \ref{setup: finite and affine root systems} and $G=G^{\sc}$. Let $ z \in T_{\Psi^{\bullet}}$ be any closed point of the maximal torus. Then the stabilizer of $z$ under the action $W_{\Psi^{\bullet}_1} \acts T_{\Psi^{\bullet}_1}$ is a reflection subgroup of $W_{\Psi^{\bullet}_1}$.
\end{proposition}
\begin{proof}
This is due to Steinberg; variants of this result are called \emph{Steinberg fixed point theorems}. In the case of finite Weyl group $W$, this is basically \cite[Theorem 1.5]{Ste64}. The statement in terms of tori, including the generalization to the affine case of $W_{\aff} = \Z\Phi^{\bullet}_G \rtimes W$, follows from \cite[\S 3, \S 4, \S 5]{Ste68} -- see in particular \cite[\S 4.2]{Ste68} for the relevant statement on fixed-points on real topological tori and \cite[\S 5.1 and \S 5.3 Theorem]{Ste68} for the translation to algebraic tori. The simply-connectedness is needed so that $X_{\bullet} = \Z\Phi_{\bullet}$.

For alternative approaches to the finite case, see \cite[chap. V]{Bou68} or \cite[Theorem 1.1]{Leh04}. For the general case, see \cite[Theorem 2.2]{PS19}, \cite[Introduction]{Pue23}. Further see \cite[\S 11, \S 12]{Kan01}.
\end{proof}

\begin{remark}
Recall the we assume $G=G^{\sc}$.
In Proposition \ref{proposition: steinberg fixed point theorem}, the stabilizer is the reflection subgroup corresponding to the closed root subsystem $\Phi^{\bullet}_{\Cent^{\circ}(z)}$, namely
\begin{equation}\label{equation: affine stabilizer as reflection group}
    \Stab_{W_{\Psi^{\bullet}_1}}(z) = W_{\Phi^{\bullet}_{\Cent^{\circ}(z)}}.
\end{equation}

Indeed, knowing Proposition \ref{proposition: steinberg fixed point theorem}, it suffices to describe the subset of reflections which lie inside the stabilizer. (Alternatively, we may lift elements of $W_{\Psi^{\bullet}_1}$ to $G_{\Psi^{\bullet}_1}$ and use the description of the stabilizer there.) We discuss this concretely in Lemmata \ref{lemma: stabilizer in finite weyl group}, \ref{lemma: affine stabilizer in weyl group}.
However, note that \eqref{equation: affine stabilizer as reflection group} may fail without the simply connected assumption; see Proposition \ref{proposition: pseudo-levi subgroups}, (ii). 
\end{remark}

%% file: 2_equivariant_K_theory.tex
\section{Equivariant \texorpdfstring{$K$}{K}-theory and fixed-points}\label{section: equivariant K-theory and fixed points}

We now set conventions regarding equivariant topological $K$-theory. Here, we recall both the usual topological approach through homotopy theory \S \ref{section: generalized cohomology theories in topology}, as well as the algebraic approach via localizing invariants \S \ref{section: algebraic cohomology theories as localizing invariants}. We also take this as an opportunity to recall how these approaches relate -- they are equivalent on finite type schemes over $k=\C$. We emphasize that we work with genuine equivariance. 

After recalling our notion of fixed-point schemes in \S \ref{section: fixed-point schemes}, we finally present in \S \ref{section: K-theory and HP and cohomology of fixed points} our reinterpretation of classical fixed-point localization in terms of the trace to equivariant periodic cyclic homology. The main statement is Theorem \ref{theorem: fibers of topological K-theory, HP and fixed-points}. We find this formulation both conceptually clear and useful in practice. In particular, the result holds without restrictions for any (possibly singular) $X$ with respect to an arbitrary action of $T$. This formulation further gives the localization map a clear global meaning.

\subsection{Generalized cohomology theories in topology}\label{section: generalized cohomology theories in topology}

\subsubsection{Generalized (co)homology theories.}
Recall the notion of a generalized (co)homology theory in topology \cite{May96}. These are representable by spectra. Multiplicative generalized cohomology theories are represented by ring spectra. Given such spectrum $E$, one may evaluate it on any topological space $X$ by taking the mapping space $\map(E, -)$ resp. $\map(-, E)$ and collecting its homotopy groups into a $\Z$-graded abelian group:
\begin{equation*}
    E_{\bullet}(X) := \pi_{\bullet}(\map(E, X)) \qquad \text{resp.} \qquad E^{\bullet}(X) := \pi_{\bullet}(\map(X, E)).
\end{equation*}
These respectively give the $E$-homology resp $E$-cohomology of $X$.
Our exposition is biased towards the contravariant cohomological picture, which carries a natural ring structure when $E$ is multiplicative. We focus on this case from now on.

\begin{remark}\label{remark: sum and product convention for generalized cohomology theories}
To regard the above $E^{\bullet}(X)$ as actual ring, one needs to choose how to interpret the grading direction: for example, it is a matter of convention whether to take $\bigoplus_{\bullet \in \Z} \pi_{\bullet}(-)$ or $\prod_{\bullet \in \Z} \pi_{\bullet}(-)$.
In many situations -- such as ordinary cohomology of finite CW complexes -- this distinction is vacuous. However, in infinite-dimensional contexts and more general cohomology theories, one needs to treat this seriously. The correct cohomological convention is to take $\prod_{\bullet \in \Z}$. 
\end{remark}

\begin{notation}
Here and below, we use the following convention. For any coefficient ring $\k$, we write $E(-; \k)$ for the version with $\k$-coefficients, represented by the smash product of $E$ with the Moore space of $\k$.  When no coefficients are specified, we have $E(-) = E(-; \Z)$.    
\end{notation}

Given a Lie group $G$, one can axiomatically consider multiplicative $G$-equivariant generalized cohomology theories \cite[\S XIII]{May96}. These are representable by $G$-equivariant ring spectra.

\subsubsection{Generalized cohomology of colimits.}\label{subsubsection: K-theory of colimits}
Recall the following principle \cite[\S 19.4]{May99}, valid for any generalized cohomology theory.
\begin{recollection}
Let $j \in \Z$ and let $X$ be a topological space, given as an $\N_0$-indexed colimit
\begin{equation*}
    X = \colim_{i} X_i.
\end{equation*}
Then there is a short exact sequence
\begin{equation*}
0 \to {\lim_i}^1 (E^{j-1}(X_i)) \to E^j(X) \to \lim_i E^j(X_i) \to 0.    
\end{equation*}
Moreover, if the system $(E^{j-1}(X_i))_{i}$ satisfies the Mittag--Leffler condition, the ${\lim}^1$-term vanishes and we deduce
\begin{equation*}
    E^j(X) \xrightarrow{\cong} \lim_i E^j(X_i).
\end{equation*}
\end{recollection}

\begin{remark}\label{remark: Mittag--Leffler and limit of rings}
Under the Mittag--Leffler assumption and using the correct (product) convention from Remark \ref{remark: sum and product convention for generalized cohomology theories}, we deduce a ring isomorphism
\begin{equation*}
    E^{\bullet}(X) \cong \lim_i E^{\bullet}(X_i).
\end{equation*}
\end{remark}

\begin{remark}\label{remark: topological cellularity}
In particular, assume that $X$ is a CW complex with even cells only, together with its filtration by skeleta $\sk_i X$. Then, for any fixed $j$, all the maps in the inverse system $(E^{j}(X_i), i \in \N_0)$ are surjective, and the Mittag--Leffler condition is trivially satisfied.
\end{remark}

\begin{remark}
For a scheme $X$ over $\C$, consider its complex points $X(\C)$ with analytic topology. Given a generalized cohomology theory $E$, we then denote $E(X) := E(X(\C))$. We use similar convention in the equivariant case.

Assume that $X$ is a proper scheme over $\C$ that possesses a paving by affine spaces. Then its underlying topological space $X(\C)$ a CW complex with even cells only, and Remark \ref{remark: topological cellularity} applies.    
\end{remark}

\subsubsection{Cohomology.}
In particular, we denote by $H\Z$ the Eilenberg--Mac-Lane spectrum representing singular cohomology, and $H\k$ the version with coefficients. It represents singular cohomology; for a topological space $X$, we denote it $H(X) := H_{\sing}(X) := \map(X, H\Z)$. Similarly with other coefficients.

\begin{convention}
Given a topological space $X$, we write
\begin{equation*}
    \widehat{H}^{\bullet}(X) := \prod_{i \geq 0} H^{i}(X).
\end{equation*}
Here, the convention is that we take the direct product of all cohomology groups rather then their direct sum; see Remarks \ref{remark: sum and product convention for generalized cohomology theories}, \ref{remark: Mittag--Leffler and limit of rings}. 
It would be for the best to write simply $H^{\bullet}$ without any adornments; however, since the opposite convention is quite pervasive in geometric representation theory, we write $\widehat{H}^{\bullet}$ to avoid confusion.
\end{convention}

Cohomology extends to the equivariant setup through the Borel construction. Given a topological group $G$, denote $EG \to BG$ the classifying space of $G$ with its universal $G$-bundle. For any $G$-space $X$, its equivariant cohomology is defined by $H_G(X) := H(EG \times^G X)$.

\subsubsection{Topological $K$-theory.}
We denote by $BU$ the complex topological $K$-theory spectrum, i.e. the ring spectrum representing topological $K$-theory. Given a topological space $X$, we denote its topological $K$-theory
\begin{equation*}
    K^{\top}(X) := KU(X) := \map(X, BU).
\end{equation*}
We write $K^{\top, \bullet}(X) := \pi_{\bullet}(K^{\top}(X))$ with $\bullet \in \Z$ for its homotopy groups, which form a graded-commutative ring.

Topological $K$-theory genuinely extends to the equivariant context. 
\begin{recollection}[Genuine equivariant topological $K$-theory]\label{recollection: genuine topological K-theory}
Let $G$ be a compact topological group. The genuine equivariant $K$-theory was firstly defined by \cite{Seg68}. These ideas were further developed in terms of the representing spectrum $BU_G$, see \cite[\S XIV, in particular \S XIV.3 and \S XIV.4]{May96}. Given any $G$-space $X$, we thus have equivariant topological $K$-theory
\begin{equation*}
    K^{\top}_G(X) := KU_G(X) := \map_G(X, BU_G).
\end{equation*}
We will use this genuine definition. The discussion of previous paragraphs works in this context. The definition extends to the case when the group $G$ is a complex Lie group by passing to its maximal compact subgroup $M$; see also \cite[\S 2.1.2]{HLP20}. For finite $G$-CW complexes, the above can be also defined directly from the category of equivariant topological vector bundles on $X$ as in \cite{Weib13}. However, this later interpretation does not hold when $X$ is not finite (and does not behave like a generalized cohomology theory).
\end{recollection}

\begin{remark}[Borel-type theories]
There are also cheaper ways of extending topological $K$-theory to the equivariant context -- one can use the non-equivariant definition on the Borel construction. The resulting theory is simpler, but carries considerably less information. Already for the equivariant point, it returns only the completion of its representation ring at the origin \cite{AS69}.
Also compare to \cite{Tho86} and the discussion \cite[\S 5]{Tho88b}. We will not use the Borel-type version.  
\end{remark}

\subsubsection{Goresky--Kottwitz--McPherson descriptions.}\label{section: GKM descriptions}

Let $G$ be a topological group and $E_G$ a multiplicative $G$-equivariant cohomology theory. Let $X$ be a topological space with $G \acts X$. Under favorable assumptions, the equivariant generalized cohomology $E^{\bullet}_G(X)$ can be computed combinatorially from the fixed points of $G \acts X$ by solving a specific system of congruences. For $T$-equivariant singular cohomology $H^{\bullet}_T(X; \C)$ of a complex algebraic variety $X$, this was realized in \cite[\S 7]{GKM97}. It was then extended to generalized cohomology theories in \cite{HHH05}; we will follow their excellent treatment, to which we refer for details.
\begin{assumption}\label{assumption: HHH assumptions}
Let $G \acts X$ be a topological group acting on a topological space. Assume the following \cite[\S 3, Assumptions 1-4]{HHH05}:
\begin{enumerate}
    \item The space $X$ admits a $G$-equivariant stratification $X = \bigsqcup_{i \in I} X_i$ labeled by a poset $I$, whose graded pieces $X_i / X_{<i}$ are Thom spaces of $G$-equivariant vector bundles $V_i$ over some $G \acts F_i$.
    \item The bundles $V_i$ are $E$-orientable and admit a decompositions $V_i = \bigoplus_{j < i} V_{i, j}$ by $G$-equivariant $E$-orientable vector bundles.
    \item There is a system of $G$-equivariant maps $a_{i, j}: F_i \to F_j$ for each $j < i$, such that the following diagrams commute
    \begin{equation*}
    \begin{tikzcd}
        S(V_{i, j}) \arrow[r] \arrow[d] & X_{<i} \\
        F_i \arrow[r, "a_{i, j}"] & F_j \arrow[u, hook]
    \end{tikzcd}
    \end{equation*}
    In words, the attaching map $S(V_i) \to X_{<i}$ from the sphere bundle of $V_i$ is, on each subbundle $S(V_{i, j})$, induced by $a_{i, j}$.
    \item Fix any $i \in I$. As $j<i$ varies, the Euler classes $e(V_{i, j})$ are not zero divisors and pairwise relatively prime in $E^{\bullet}_G(F_i)$.
\end{enumerate}
\end{assumption}

The main result of \cite{HHH05} then reads as follows.
\begin{proposition}\label{proposition: GKM of HHH}
Assuming (1), the following restriction map is injective
\begin{equation*}
    E^{\bullet}_G(X) \hookrightarrow \prod_{i \in I} E^{\bullet}_G(F_i).
\end{equation*}
Assuming (1)-(4), its image is given by the GKM ring
\begin{equation*}
    \GKM^{\bullet}_G(X) := \big\{ (f_i) \in \prod_{i \in I} E^{\bullet}_G(F_i) \mid f_i - a^*_{i, j} (f_j) \in e(V_{i, j}) E^{\bullet}_G(F_i)  \ \forall j < i \big\}.
\end{equation*}   
\end{proposition}
\begin{proof}
   See \cite[Theorems 3.1 and 2.3]{HHH05}.
\end{proof}

\begin{remark}\label{remark: GKM in the BB case}
The most common use case is that of a torus action on a complex algebraic variety $T \acts X$. It induces a Białynicki-Birula decomposition of $X$ by attractors of fixed points. Assume that the fixed points $X^T$ are isolated and the strata are affine spaces. One then takes $(F_i, i\in I)$ to be the set of fixed-points, $V_{i}$ the attractor cells, $V_{i, j}$ their weight decompositions, $a_{i, j}$ the obvious maps. In many real-life examples -- namely when the $T$-action is sufficiently non-degenerate -- the conditions (3) and (4) hold.
\end{remark}

We have followed \cite{HHH05}. Also see \cite[Appendix A]{Ros03}, \cite{HHRW16}, \cite{Zho25}, \cite{KR21} for related subsequent work.

\subsection{Algebraic cohomology theories as localizing invariants}\label{section: algebraic cohomology theories as localizing invariants}

\subsubsection{Localizing invariants.}
Certain generalized (co)homology theories do have a purely algebraic counterpart as localizing invariants $E(-)$ of stable categories. Given a space $X$ of algebro-geometric origin, one can then compute $E(X)$ via its derived category. More precisely, such theories again come in pairs: the covariant version $E(\Coh(X))$ and the contravariant version $E(\Perf(X))$. 

Again, we are biased towards the latter.
As in the topological case, one good reason for this is the ring structure. Namely, assume $E(-)$ is lax monoidal. Since $\Perf(X)$ has a symmetric monoidal structure via $\otimes$, the output $E(\Perf(X))$ is an $\mathbb{E}_{\infty}$-ring spectrum.

Given an affine group scheme $G$ acting on a qcqs scheme $X$, we put
\begin{equation*}
    E_G(X) := E(\Perf(G \backslash X)),
\end{equation*}
in line with \cite{Tho88, TT90}.
Taking homotopy groups, we obtain the $\Z$-graded abelian group $\pi_{\bullet}(E_G(X))$. If $E(-)$ was lax monoidal, this is a graded-commutative ring.

\begin{convention}\label{remark: regrettable indices}
In this paper, we use the following nonstandard convention. Given a localizing invariant $E$, we denote 
\begin{equation*}
E^{i}(\eX) := \pi_{i}(E(\eX)) := \pi_{i}E(\Perf(\eX))),   
\end{equation*}
in other words, we put the homotopy degree to the upper index. While this is non-standard from the viewpoint of localizing invariants, we are forced to do so in order to match the homology-cohomology conventions from algebraic topology: our invariants are computed from $\Perf(\eX)$ and hence compare to cohomological versions in topology. (The lower-index convention would be given by $\Coh(\eX)$ and correspond to homology.)
\end{convention}

The known localizing invariants either come from vector bundles (e.g. algebraic $K$-theory and its variants) or from differential forms (e.g Hochschild homology $HH$ and its variants). 
We write $K(\eX) := K(\Perf(\eX))$ for non-connective algebraic $K$-theory.

Perhaps surprisingly, purely topological invariants of qcqs schemes over $\C$ can be recovered from this machinery. The relevant examples are Blanc's topological $K$-theory $K^{\btop}$ and periodic cyclic homology $HP$.

\subsubsection{Equivariant topological $K$-theory.}
We work over $k=\C$. We denote by $K^{\btop}$ Blanc's non-connective topological $K$-theory \cite{Bla16}. This is a multiplicative localizing invariant of $\C$-linear categories.
In particular, given any affine group $G$ action on a qcqs scheme $X$ over $\C$, we obtain
\begin{equation*}
    K^{\btop}_G(X) := K^{\btop}(\Perf(X/G)).
\end{equation*}
It is an $\mathbb{E}_{\infty}$-ring spectrum.

\begin{remark}
For $X$ qcqs scheme over $\C$, Blanc's topological $K$-theory $K^{\btop}(X)$ agrees with the equivariant topological $K$-theory $K^{\top}(X) \simeq KU(X)$ of the analytification from \S \ref{section: generalized cohomology theories in topology}. More generally, for an action of an algebraic group $G \acts X$ where $X$ is qcqs, it agrees with the genuine equivariant topological $K$-theory $K^{\top}_G(X) = KU_G(X)$ of on the analytification $X(\C)$ in the sense of \S \ref{recollection: genuine topological K-theory}. See \cite{Bla16}, \cite{HLP20}, \cite[Remark 2.27]{Low25} for this identification.

Consequently, we may (and will) identify $K^{\btop}_G(X) \simeq K^{\top}_G(X)$ in such situations, in particular when $X$ is a $G$-equivariant projective variety. However, note that for ind-schemes we use the purely topological $K^{\top}_G$.
\end{remark}

\begin{notation}
We denote by $u \in K_{2}(\pt) = K^{2}(\pt)$ the Bott element in algebraic $K$-theory. 

It maps into topological $K$-theory, giving an an invertible element $u \in K^{\btop, 2}(\pt)$ with inverse $u^{-1} \in K^{\btop, -2}(\pt)$. In particular, multiplication by $u$ is an isomorphism on homotopy groups shifting degree by $2$, making $K^{\btop, \bullet}(X)$ 2-periodic. We write $K^{\btop, *}(X)$ with $* \in \Z/2$ for this $\Z/2$-graded, graded-commutative ring.
\end{notation}

\subsubsection{Equivariant periodic cyclic homology.}
Consider now the localizing invariant $HP(-/k)$ of $k$-linear categories. Following our conventions, we write $HP^G(X / k) := HP(\Perf(X/G) / k)$ for the $G$-equivariant periodic cyclic homology of $X \in \Sch^G_{k}$. 
We write $u \in HP_{2}(\pt / k) = HP^{2}(\pt / k)$ for the periodicity operator with inverse $u^{-1} \in HP^{-2}(\pt)$ as in \cite[\S 5.1.4]{Lod97}.

Let us specialize to $k=\C$. Then the Dennis trace induces a map of localizing invariants $K^{\btop}(-) \to HP(- / \C)$ called the topological Chern character. The trace map $K^{\btop}(\pt; \C) \to HP(\pt / \C)$ identifies the Bott element on the left-hand side with the periodicity operator on the right-hand side. See \cite{Lod97} for classical exposition; see \cite{Chen20, EKS25} for the equivariant case. We will recall its geometric incarnation in terms of fixed points in \S \ref{section: fixed-point schemes} below.

\subsubsection{Equivariant base.}
For any group scheme $G$ over $k$, we tautologically identify the equivariant $K$-theory of a point with the representation ring
\begin{equation*}
    K^0_G(\pt) = K^{\btop, 0}_G(\pt) = R(G)
\end{equation*}
ans similarly with coefficients.
Working over $k = \C$, the topological Chern character identifies
\begin{equation*}
    K^0_G(\pt; \C) = K^{\btop, 0}_G(\pt; \C) = R(G; \C) = HP^0_G(\pt / \C).
\end{equation*}

\subsubsection{Derived fibers over $T$.}
Working with a torus $T$ over $k$, the equivariant base identifies with its ring of functions.
\begin{equation*}
 K^{0}_T(\pt; k) = K^{\btop, 0}_T(\pt; k) = \O(T).   
\end{equation*}
Given any closed point $t \in T$, we denote $t: \O(T)\to k$ the corresponding algebra homomorphism. To emphasize the $\O(T)$-algebra structure on the right-hand side, we sometimes write $k = k_t$. We further denote $(-)_{\lL t} = (-) {\otimes^{\lL}_{\O(T)}} k_t$ the derived tensor product along this map. 

In particular, the functor $(-)_{\lL t}$ can be applied to $K^{\btop}_T(-; k)$ and $HP^{\btop}_T(- / k)$, both regarded as algebra objects in the derived category of $\O(T)$-modules, getting
\begin{equation*}
    K^{\btop}_T(X; k)_{\lL t} \qquad \text{and} \qquad HP_T(X / k)_{\lL t}.
\end{equation*}

\subsubsection{Remarks.}
Let us outline how the current section parallels the topological discussion.

\begin{remark}[Comparisons]
Comparisons of algebraic $K_G$ and topological $K^{\top}_G$ have a long history -- algebraic invariants usually encode the topological ones. Algebraic $K$-theory of coherent sheaves (usually called $G$-theory) relates to topological $K$-homology; algebraic $K$-theory of perfect complexes relates to topological $K$-cohomology. See \cite{Tho88, Tho88b}.
\end{remark} 

\begin{remark}[Kan extended theories]
In the same way as in topology, there is a cheaper way of producing equivariant invariants. Instead of using $\Perf(G \backslash X)$, one can Kan-extend from the non-equivariant case. The result then parallels the Borel-type variant in topology: it is much simpler and sees only the completion at the augmentation ideal. See \cite{Tho86}, or also \cite{TV24}. Again, we do not use these variants.    
\end{remark}

\begin{remark}[Motivic filtrations]\label{remark: motivic filtrations}
Algebraic $K$-theory and its variants of any $X \in \Sch_{k}$ carry the \textit{motivic filtration} \cite{EM23}. In characteristic zero and after taking rational coefficients, this filtration naturally splits. In particular, $K^0(X; \C)$ is naturally a graded ring. There is a compatible filtration on $HP(-/ \C)$. Interpreting the latter in terms of de Rham cohomology, this is precisely the de Rham filtration. Again, it naturally splits, making $HP^0(X/ \C)$ into a graded ring. The trace map is compatible with these filtrations.

Such filtrations also appear in the equivariant setup. For example, one can get a filtration on $K_G(X)$ by Kan extension from the motivic filtration on non-equivariant $K$-theory as in \cite{EKS25}. Also see \cite[(2.6)]{TV24}, \cite{KN25}, \cite{EKS25} for more discussion of Atyiah--Segal completion phenomena in this context. However, it is not well understood whether such a filtration exists genuinely \cite[Question 1.0.2]{EKS25}, in other words ``relatively over $BG$".
\end{remark}

\subsection{Fixed-point schemes}\label{section: fixed-point schemes}

Let $G$ be a group scheme over $k$ and denote $\Sch_k^G$ the category of $G$-equivariant $k$-schemes. We recall the notion of the associated fixed-point scheme from \cite[\S 1]{Low24}.

\subsubsection{Fixed point schemes.}
Suppose $G$ acts on $X \in \Sch^G_k$ via $m: G\times X \to X$.  We define the associated \textit{fixed-point scheme} as the fiber product of schemes
\begin{equation}\label{equation: definition of fixed-point schemes}
\begin{tikzcd}
\Fix_G(X) \arrow[r] \arrow[d] & G\times X \arrow[d, "m \times \pr_2"] \arrow[r, "\pr_1"] & G \\
X \arrow[r, "\Delta_X"] & X\times X
\end{tikzcd}
\end{equation}
Its functor of points is given by
\begin{equation*}
    \Fix_G(X): R \mapsto \{ (g, x) \in G(R) \times X(R) \mid gx = x \}.
\end{equation*}
The horizontal composition in \eqref{equation: definition of fixed-point schemes} gives a map $\pi: \Fix_G(X) \to G$. For any $g \in G(k)$, the fiber $\pi^{-1}(g)$ parametrizes those $x \in X(R)$ which are fixed by $g$. 
Given any $S \in \Sch_k$ with a map $S \to G$, we obtain
\begin{equation}
\Fix_S(X) := S {\times}_G \Fix_G     
\end{equation}
and view it as the fixed point family of the restriction of the $G$-action to $S$. In particular, given any closed point $g \in G$, we have its scheme-theoretic fixed-points $\Fix_g(X) = X^g$.

\subsubsection{Residual action.}\label{section: residual action on Fix}
The fixed-point scheme carries a natural action $G \acts \Fix_G(X)$ induced from the original action on $G \acts X$ and conjugation action $G \acts G$ on the defining diagram  \eqref{equation: definition of fixed-point schemes}. We denote the quotient by $\Fix_{\tfrac{G}{G}}(X)$.

In particular, suppose $h, g \in G$ and denote $g' = hgh^{-1}$. Then we have a natural identification
\begin{equation*}
    \Fix_{g}(X) = \Fix_{hgh^{-1}}(h\cdot X) \xrightarrow{=} \Fix_{g'}(X)
\end{equation*}
where the latter arrow is induced by the translation $h^{-1}: X \to X$.
For fixed $g \in G$, we thus obtain an action
\begin{equation*}
    \Cent_G(g) \acts \Fix_g(X).
\end{equation*}

\subsubsection{Fixed-point schemes preserve abstract blowups.}
We will need the following elementary observation about fixed-point schemes.
\begin{lemma}\label{lemma: fixed-point schemes of abstract blowup squares}
Let $G$ be a group over $k$, and $S \to G$ any map from some $S \in \Sch_k$. Then the classical fixed-point scheme functor
\begin{equation*}
    \Fix_S(-): \Sch_k^G \to \Sch_S
\end{equation*}
preserves abstract blowup squares. In particular, this applies to $\Fix_g(-)$ for any point $g \in G$.
\end{lemma}
\begin{proof}
The classical fixed-point scheme functor $\Fix_S(-)$ preserves isomorphisms, closed immersions, open immersions, closed-open decompositions, proper maps \cite[Propositions 1.2 and 1.3]{Low24}. It commutes with fiber products (in fact with all limits \cite[Lemma 1.7]{Low24}). Therefore, it also preserves abstract blowup squares.
\end{proof}

\subsubsection{Reduced and derived variants.}\label{subsubsection: reduced and derived variants}
If $\Fix_\frac{G}{G}(X)$ is non-reduced, we can take its reduction $\Fix^{\red}_{\frac{G}{G}}(X)$. On the other hand, we may take the defining fiber product \eqref{equation: definition of fixed-point schemes} in derived schemes, giving a derived enhancement $\Fix^{\lL}_{\frac{G}{G}}(X)$. We get closed immersions
\begin{equation}\label{equation: reduced, classical and derived fixed-point schemes}
    \Fix^{\red}_{\frac{G}{G}}(X) \hookrightarrow \Fix_{\frac{G}{G}}(X) \hookrightarrow \Fix^{\lL}_{\frac{G}{G}}(X).
\end{equation}
For the applications in this paper, the reduced version will be sufficient.

On the other hand, the derived variant has the following interpretation. 
For any derived stack $\eX$ over $k$, its derived loop space $\dL (\eX)$ is defined by the derived self-intersection of the diagonal $\dL (\eX) = \eX \times_{\eX \times_k \eX} \eX = \Map(S^1, \eX)$. Applied to $\eX = G \backslash \eX$, there is an equivalence
\begin{equation*}
    \dL (G \backslash X) \simeq \Fix^{\lL}_{\frac{G}{G}}(X).
\end{equation*}
This interpretation allows to compare to periodic cyclic homology, see Proposition \ref{proposition: periodic cyclic homology and derived fixed-point schemes} below. If $G$ is an affine group acting on a quasi-projective variety $X$, the derived structure on $\Fix^{\lL}_{\frac{G}{G}}(X)$ is trivial if and only if $G\acts X$ acts with finitely many orbits by the criterion \cite[Proposition 3.3]{BN13}.

\subsection{\texorpdfstring{$K$}{K}-theory, periodic cyclic homology, and cohomology of fixed points}\label{section: K-theory and HP and cohomology of fixed points}

\subsubsection{Localization theorem.}
We work over the base field $k=\C$. Let $T = \Gm^{n}$ be an algebraic torus over $k$. 
Throughout this section, we write $\Sch^T_{\C} := \Sch^{T, \qcqs}_{\C}$ for the category of $T$-equivariant, quasi-compact quasi-separated schemes over $\C$  -- we omit the superscript $\qcqs$ for readability.

Let $X \in \Sch^{T}_{\C}$ be any $T$-equivariant qcqs scheme over $\C$; it is in particular allowed to be singular. 
The purpose of this section is to recall the (maybe not so well) known relationship between its complexified equivariant topological $K$-theory $K^{\btop}_T(X; \C)$, equivariant periodic cyclic homology $HP^T(X / \C)$ and the singular cohomologies of the varying fixed-point subschemes $\Fix_t(X)$ for $t \in T$.

\begin{theorem}\label{theorem: fibers of topological K-theory, HP and fixed-points}
Let $X \in \Sch^{T}_{\C}$. Then we have equivalences
\begin{equation}\label{equation: K, HP, Fix}
    K^{\btop}_{T}(X; \C) \xrightarrow{\simeq} HP_T(X / \C) \xrightarrow{\simeq} \RG(\Fix^{\lL}_{\frac{T}{T}}(X), \O)^{tS^1}
\end{equation}
in $\CAlg(\O(T)[u^{\pm 1}])$. 
Taking fibers over a closed point $t \in T$, this equivalence specializes to
\begin{equation}\label{equation: fibers and singular cohomology of fixed-points}
    K^{\btop}_T(X; \C)_{\lL t} \xrightarrow{\simeq} H_{\sing}(\Fix_t(X); \C)\llp u \rrp
\end{equation}
in $\CAlg(D(\C))$.
\end{theorem}
\begin{proof}
Proposition \ref{theorem: fibers of HP and cohomlogy of fixed points}, Proposition \ref{proposition: periodic cyclic homology and derived fixed-point schemes} and Proposition \ref{theorem: fibers of equivariant K-theory nd cohomology of fixed points} below.
\end{proof}

\begin{remark}
In plain english: the fibers of equivariant topological $K$-theory over various closed points $t$ of the equivariant base $T$ are isomorphic with the singular cohomologies $H_{\sing}(\Fix_t(X); \C)\llp u \rrp$ of the corresponding scheme-theoretical fixed-points $\Fix_t(X)$. Reformulating, $T$-equivariant topological $K$-theory of $X$ puts the singular cohomologies of the fixed-point varieties $\Fix_t(X)$ into an algebraic family.
This algebraic family is geometrically computed by 
$$\RG(\Fix^{\lL}_{\frac{T}{T}}(X), \O))^{tS^1},$$
the Tate fixed points of the circle action on global functions on the derived fixed-point stack $\Fix^{\lL}_{\frac{T}{T}}(X)$.
\end{remark}

We also have the following formally deformed variant on the final assertion of the above theorem.
\begin{proposition}\label{remark: the completed variant of equivariant localization}
Let $X \in \Sch^{T}_{\C}$. Taking completions $(-)^{\wedge}_{\lL t}$ along $t \in T$, we obtain
\begin{equation}\label{equation: completed version K-theory}
  K^{\btop}_T(X; \C)^{\wedge}_{\lL t} \xrightarrow{\simeq} H_{T}(\Fix_t(X); \C) \llp u \rrp^{\wedge} 
\end{equation}
Here, the completion of the right-hand side it with respect to the derived Hodge filtration in the Cartan model of equivariant cohomology. 
\end{proposition}
\begin{proof}
Take Proposition \ref{theorem: topological K-theory and HP}, then take the formal derived completion $(-)^{\wedge}_{\lL t}$ over $t \in T$, and finally apply Proposition \ref{proposition: fibers of HP and singular cohomology -- deformed}.
\end{proof}

\subsubsection{Case of reductive groups.}

\begin{remark}[Case of reductive groups.]
Let $G$ be any reductive group acting on a quasi-projective variety $X$ over $\C$. Then the equation \eqref{equation: K, HP, Fix} holds $G$-equivariantly:
\begin{equation*}
    K^{\btop}_{G}(X; \C) \xrightarrow{\simeq} HP^G(X / \C) \xrightarrow{\simeq} \RG(\Fix^{\lL}_{\frac{G}{G}}(X), \O)^{tS^1}.
\end{equation*}
Moreover, let $g \in G$ be an element with semisimple part $\overline{g} \in \Spec(R_{\C}(G))$. Then the identification \eqref{equation: fibers of HP and singular cohomology of fixed-points} goes through as:
\begin{equation*}\label{equation: fibers and singular cohomology of fixed-points semisimplification}
    K^{\btop}_T(X; \C)_{\lL \overline{g}} \xrightarrow{\simeq} H_{\sing}(\Fix_{\overline{g}}(X); \C)\llp u \rrp
\end{equation*}
Note that this does {not} work without semisimplicifation.
\end{remark}
\begin{proof}[Proof of Remark.]    
This version follows by the same arguments, combining results of \cite{Chen20} in the smooth case with the recent progress on equivariant cdh descent \cite{LS25} with respect to reductive groups $G$. Also see \cite[Appendix B]{Low25}.
\end{proof}

In particular, even if we consider $G$-equivariant $K$-theory instead of $T$-equivariant one, there is no extra information in its fibers.

\subsubsection{$K$-theory and periodic cyclic homology.}
Let us first recall the relationship between equivariant topological $K$-theory and equivariant periodic cyclic homology. 

\begin{proposition}\label{theorem: topological K-theory and HP}
For any $X \in \Sch^{T}_{\C}$ we have an equivalence
\begin{equation*}
K^{\btop}_T(X; \C) \xrightarrow{\simeq} HP_T(X / \C)    
\end{equation*}
as algebras over $K_T^{\btop}(\pt; \C) \simeq HP_T(\pt / \C)$.
\end{proposition}
\begin{proof}
For a sufficient reference covering (possibly singular) stacks with nice stabilizers, see \cite{Kha23}. This heavily builds on the previous works \cite{Bla16, Kon21, HLP20, ES21}. 
\end{proof}

\subsubsection{Periodic cyclic homology and fixed-points.}
Periodic cyclic homology $HP(-/\C)$ over $\C$ is further geometrically modeled by suitable functions on the derived fixed-point scheme.
\begin{proposition}\label{proposition: periodic cyclic homology and derived fixed-point schemes}
There is a natural equivalence
    \begin{equation*}
        HP_T(X/\C) \xrightarrow{\simeq} \RG(\Fix^{\lL}_{\frac{T}{T}}(X), \O)^{tS^1}.
    \end{equation*}
Here, $(-)^{tS^1}$ refers to the Tate construction with respect to the natural $S^1$-action on the derived fixed-point scheme.
\end{proposition}
\begin{proof}
Apply the Tate construction to \cite[Discussion 1.12]{Low24} and \cite[Example 2.2.20]{Chen20}
\end{proof}

\subsubsection{Periodic cyclic homology computes cohomology of fixed-points.}
Let us now explain how fibers of equivariant periodic cyclic homology relate to the singular cohomology of fixed-points.
\begin{proposition}\label{theorem: fibers of HP and cohomlogy of fixed points}
Let $t \in T$ be a closed point. We then have an equivalence
\begin{equation}\label{equation: fibers of HP and singular cohomology of fixed-points}
    HP_T(X/ \C)_{\lL t} \xrightarrow{\simeq} H_{\bet}(\Fix_t(X); \C)\llp u \rrp.
\end{equation}
\end{proposition}
\begin{proof}
First recall the identification of derived de Rham cohomology with singular cohomology of the analytification \cite[Theorem 4.2.9 and Corollary 4.2.11]{Chen20} and \cite[Theorem 4.10]{Bha12}, valid for any derived stack of finite type over $\C$ (the case of quasi-projective varieties is sufficient for us).

We first assume that the $T$-variety $X$ is smooth. Then \eqref{equation: fibers of HP and singular cohomology of fixed-points} holds by \cite[Corollary 1.0.4]{Chen20} (also see \cite[Theorem D]{Chen20} for a finer statement), using the above discussion to identify the right-hand side with singular cohomology. We use here the semisimplicity assumption on $t$.

Now, both sides of \eqref{equation: fibers of HP and singular cohomology of fixed-points} satisfy proper excision on $T$-equivariant abstract blowup squares.
Indeed, the left-hand side is given by the composition
\begin{equation*}
    \Sch^{T}_{\C} \xrightarrow{HP_T(-/ \C)} \D(HP_T(\pt / \C)) \xrightarrow{(-)_{\lL t}} \D(\C)
\end{equation*}
Here, $HP(-/\C)$ is a truncating invariant, so $HP^T(-/\C)$ has proper excision by \cite{ES21}. The same holds after composition with the derived functor $(-)_{\lL t} = (-)\otimes^{\lL}_{HP_T(\pt / \C)} \C_t$.
On the other hand, the right-hand side can be factored as
\begin{equation*}
    \Sch^{T}_{\C} \xrightarrow{\Fix_t(-)} \Sch_{\C} \xrightarrow{H_{\bet}(-; \C)\llp u \rrp} D(\C).
\end{equation*}
Here, the functor $\Fix_t(-)$ preserves abstract blowup squares by Lemma \ref{lemma: fixed-point schemes of abstract blowup squares}. 
Moreover, singular cohomology (with any coefficients) has proper excision (this being the case for any generalized cohomology theory with open-closed long exact sequence). Alternatively, this last claim can be seen on the level of derived de Rham cohomology, which is a canonical direct summand of $HP(-/\C)$ and hence has cdh descent \cite[Lemma 4.5]{EM23}.

Altogether, by equivariant resolution of singularities in characteristic zero and induction on dimension, the singular case of \eqref{equation: fibers of HP and singular cohomology of fixed-points} follows from the smooth one.
\end{proof}

\subsubsection{Completions of equivariant cohomology.}
The above proposition has the following formally deformed variant.
\begin{proposition}\label{proposition: fibers of HP and singular cohomology -- deformed}
Let $X \in \Sch^{T}_{\C}$. For any closed point $t \in T$, we have
\begin{equation}\label{equation: completed version HP}
  HP_T(X; \C)^{\wedge}_{\lL t} \xrightarrow{\simeq} H_{T, \sing}(\Fix_t(X); \C) \llp u \rrp^{\wedge}.
\end{equation}
\end{proposition}
\begin{remark}
Here, the completion of the left-hand side is the derived formal completion over $t \in T$. The completion of the right-hand side is with respect to the derived Hodge filtration in the Cartan model of equivariant cohomology -- see \cite[Example 2.1.10, Definition 4.2.3]{Chen20} for a discussion (using the cotangent complex in general).
\end{remark}
\begin{proof}
If $X$ is smooth, the statement is (a special case of) \cite[Theorem D, Theorem 4.3.2]{Chen20}. We again deduce the general case by proper excision. For the left-hand side, $HP(-/\C)^{\wedge}_{\lL t}$ satisfies proper excision -- it is given by composing the truncating invariant $HP(-/ \C)$ with $(-)^{\wedge}_{\lL t}$. For the right-hand side, we again use that $\Fix_t(-)$ preserves abstract blowups. Moreover, since the derived de Rham complex satisfies proper excision, the associated filtrations on $H_{\sing}(-; \C) \llp u \rrp$ are also compatible with proper excision -- so taking $H_{\sing}(-; \C)\llp u \rrp^{\wedge}$ satisfies proper excision. By existence of equivariant resolutions in characteristic zero, we thus reduce to the smooth case.
\end{proof}

\subsubsection{Fibers of equivariant $K$-theory and cohomology of fixed-points.}
Putting the above discussion together, we obtain the second part of the theorem.
\begin{proposition}\label{theorem: fibers of equivariant K-theory nd cohomology of fixed points}
    Let $X \in \Sch^T_{\C}$ and $t \in T$ be a closed point. Then we have an equivalence
\begin{equation*}
    K_T^{\btop}(X; \C)_{\lL t} \xrightarrow{\simeq} H_{\sing}(\Fix_t(X); \C)\llp u \rrp
\end{equation*}
in $\CAlg(D(\C))$.
\end{proposition}
\begin{proof}
Let $t: \C[t_1^{\pm 1}, \dots, t_n^{\pm 1}] \to \C$ be a classical point of the equivariant base. Taking the fiber $(-)_t$ of the equivalence from Proposition \ref{theorem: topological K-theory and HP} and composing with the equivalence from Proposition \ref{theorem: fibers of HP and cohomlogy of fixed points}, we deduce
\begin{equation*}
K_T^{\btop}(X; \C)_{\lL t} \xrightarrow{\simeq} HP_T(X / \C)_{\lL t} \xrightarrow{\simeq} H_{\bet}(\Fix_t(X); \C)\llp u \rrp.   
\end{equation*}
\end{proof}

\subsubsection{Equivariant formality.}
Under extra flatness assumptions, the above derived identifications simplify. In particular, this happens under suitable equivariant formality hypothesis. Let us hide the formal parameter $u$ by regarding $K_{T}^{\btop, *}(X)$ as graded by $* \in \Z/2$.

\begin{corollary}\label{corollary: K-theory in the flat situation}
Assume that $K_{T}^{\btop, *}(X)$ is a flat module over the equivariant base and let $t \in T$.
Then we have an isomorphism of $\Z/2$-graded-commutative rings
\begin{equation*}
    K_{T}^{\btop, *}(X; \C) \underset{\O(T)}{\otimes} \C_t = H^{\bullet}_{\bet}(\Fix_t(X); \C)
\end{equation*}
where $* \in \Z/2$ and $\bullet \in \N_{\geq 0}$. 
\end{corollary}
\begin{proof}
Take $\pi_{\bullet}(-)$ in Proposition \ref{theorem: fibers of equivariant K-theory nd cohomology of fixed points}.
By the flatness assumption, the base change along $\O(T) \to \C_t$ is underived, hence commutes with the passage to homotopy groups on the left-hand side. 

Both sides are $2$-periodic; the trace map is compatible with the natural filtrations; the element $u$ has degree $-2$. This explicitly identifies
\begin{equation*}
    K_{T}^{\btop, 0}(X; \C) \underset{\O(T)}{\otimes} \C_t = \bigoplus_{i \geq 0} H^{2i}_{\sing}(\Fix_t(X); \C) \cdot u^i
\end{equation*}
\begin{equation*}
    K_{T}^{\btop, 1}(X; \C) \underset{\O(T)}{\otimes} \C_t = \bigoplus_{i\geq 0} H^{2i+1}_{\sing}(\Fix_t(X); \C) \cdot u^i
\end{equation*}
\end{proof}

\begin{remark}\label{remark: K-theory in the superflat situaton}
Even more specifically, suppose that $K_{T}^{\btop, *}(X)$ is concentrated in the even part, where it is given by a flat module over the equivariant base. Then we have a ring isomorphism
$$K_{T}^{\btop, 0}(X; \C)\underset{\O(T)}{\otimes} \C_t =  H^{\bullet}_{\bet}(\Fix_t(X); \C),$$
the right-hand size is concentrated in even degrees, and the motivic grading matches the cohomological grading. This is the case relevant for our applications.
\end{remark}

\begin{aside}
An appropriate definition of the equivariant motivic filtration (Remark \ref{remark: motivic filtrations}) on the left-hand side would match the cohomological filtration on the right-hand side. For example, this is the case over $1 \in T$ via the Kan-extended motivic filtration.
\end{aside}

\subsubsection{Bibliographical remarks.}
The groundwork for the presented statements is contained in the literature. Moreover, weaker versions of some of these claims can be recovered from the classical literature on completion and localization techniques in equivariant topology, pioneered by Segal \cite{Seg68} and Goresky--Kottwitz--MacPherson \cite{GKM97}. These weaker statements would be also sufficient for our main applications. Nevertheless, we hope our exposition gives a conceptually clear and useful viewpoint. 

Our discussion above is based on the following two lines of development.
\begin{itemize}
    \item The relationship between topological $K$-theory $K^{\btop}$ and periodic cyclic homology $HP$ comes, in various degrees of generality, from \cite{Bla16}, \cite{HLP20}, \cite{Kon21}, \cite{ES21}, \cite{Kha23}, \cite{LS25}. 
    \item The relationship between $HH$ and functions on derived fixed-point schemes originates in \cite{BFN10}; the topological version with $HP$ follows by taking the Tate construction. We refer to \cite{Chen20} for a throughout and highly recommendable discussion of equivariant $HP$ with its relation to fixed-point schemes and ultimately to singular cohomology of fixed-points. Also see \cite{Bha12}, \cite{ACH19}, \cite{HSS17}, \cite{BN12, BN13}, \cite{Toe14} for more context.
\end{itemize}
An additive analogue of our fixed-point schemes together with its relation to equivariant cohomology was studied in \cite{HR23} on examples with trivial derived structure.

%% file: 3_cocharacter_graphs_and_resonance.tex
\section{Resonace, centralizers and cocharacter graphs}\label{section: cocharacter graphs and resonance}

We first introduce in \S \ref{section: resonance} the notion of \textit{resonance} inside $T \times \Gm^{\rot}$. Intuitively, two points in $T$ are resonant with respect to a scalar $\zeta \in \Gm^{\rot}$ if they multiplicatively differ by integer powers of $\zeta$. This gives rise to the \textit{resonant Weyl group} $W_{[x, \zeta]}$ and the associated \textit{resonant subgroup} $L_{[x, \zeta]}$. 
The notion of resonance appears when describing fixed-points of $(x, \zeta)$ in the affine Grassmannian and computing the corresponding fibers of its equivariant $K$-theory, as well as in the context of linkage principles in classical and quantum categories $\O$.

We then in \S \ref{subsection: centralizers in loop groups} describe the centralizer of $(x, \zeta)$ in $\lo G$ via its closed root subsystem, and discuss stabilizers of coordinate fixed points $t^{\lambda} \in \Gr_G$ inside it. We further consider the adjoint action of the resonant cocharacter $\omega$ on these objects, which we call \emph{slanting}. 
For simply connected $G$, we futher describe in \S \ref{subsection: weyl group stabilizers} the stabilizer of $(x, \zeta)$ inside the affine Weyl group.

We finally discuss in \S \ref{section: cocharacter graphs} the \textit{cocharacter graphs} $\Gamma$ inside $T \times \Gm^{\rot} \times T$ for a simply connected $G$. We explain how $\Gamma$ puts the resonant combinatorics into a family over $T \times \Gm^{\rot}$. Ultimately, $\Gamma$ turns out to be essential for our description of the whole $K$-theory ring of the affine Grassmannian and affine Schubert varieties inside it, because restriction to $\Gamma$ realizes equivariant localization to fixed points. We discuss the variant for reductive $G$ afterwards in \S \ref{section: cocharacter graphs for reductive groups}.

Altogether, the discussion in this section serves as a useful tool in the rest of the paper.

\subsection{Resonance and centralizers}\label{section: resonance}
To eventually describe fixed-points of elements $(x, \zeta) \in T \times \Gm^{\rot}$ on the affine Grassmannian, we introduce the notion of \emph{resonance}, \emph{resonant subgroups} and \emph{resonant cocharacter}.

\subsubsection{Resonance.}\label{subsubsection: resonance}
Let $(x, \zeta) \in T \times \Gm^{\rot}$. 
A point $x \in T$ is called \emph{$\zeta$-integral} if it is of the form $x=\lambda(\zeta)$ for some $\lambda \in \Z\Phi_{\bullet} \leq X_{\bullet}$. 
We say that $x, y \in T$ are \emph{$\zeta$-resonant} if there is some $\lambda \in \Z\Phi_{\bullet}$ such that $y = \lambda(\zeta) \cdot x$. In words, two points are $\zeta$-resonant if their ratio is $\zeta$-integral. 

We denote the equivalence relation given by $\zeta$-resonance as $x \sim_{\zeta} y$. For given $(x, \zeta) \in T \times \Gm^{\rot}$, we denote the equivalence class of $x$ under $\zeta$-resonance by $[x]_{\zeta} = [x, \zeta]$. We will use the latter notation to avoid too many subscripts.

The resonance relation is induced by the orbits of the \emph{$\zeta$-integral action} $\Z\Phi^{\bullet} \acts T$, $\lambda \cdot x = \lambda(\zeta) \cdot x$.

\subsubsection{Resonant subgroups.}
Recall the discussion of root subsystems from \S \ref{section: root subsystems and centralizers}.
Depending on $(x, \zeta) \in T \times \Gm^{\rot}$, consider the closed root subsystems
\begin{equation}
    \Phi^{\bullet}_x \subseteq \Phi^{\bullet}_{[x, \zeta]} \subseteq \Phi^{\bullet}_G
\end{equation}
given by
\begin{equation}\label{equation: resonant root system}
    \Phi^{\bullet}_x := \{\beta \in \Phi^{\bullet}_G \mid \beta(x)=1 \},
    \qquad \text{and} \qquad
    \Phi^{\bullet}_{[x, \zeta]} := \{\beta \in \Phi^{\bullet}_G \mid \exists m \in \Z, \ \beta(x)= \zeta^m \}.
\end{equation}
Indeed, both subsets $\Phi^{\bullet}_x$, $\Phi^{\bullet}_{[x, \zeta]}$, clearly satisfy (i), (ii) from Notation \ref{notation: closed root subsystems}.

We denote the corresponding regular subgroups by
\begin{equation}
L_x \leq L_{[x, \zeta]} \leq G.   
\end{equation}
We call them the \emph{stabilizer subgroup} and \emph{resonant subgroup}.
Their Weyl groups are denoted
\begin{equation}\label{equation: resonant weyl group}
    W_x \leq W_{[x, \zeta]} \leq W.
\end{equation}
We call them the \emph{stabilizer Weyl group} and \emph{resonant Weyl group}.
These are the reflection subgroups associated to the above closed root subsystems $\Phi^{\bullet}_{x}$ and $\Phi^{\bullet}_{[x, \zeta]}$.

The resonant root subsystem $\Phi^{\bullet}_{[x, \zeta]}$, resonant subgroup $L_{[x, \zeta]}$ and resonant Weyl group $W_{[x, \zeta]}$ clearly depend, up to canonical isomorphism, only on the resonance class $[x, \zeta]$ of the element $(x, \zeta)$. 
The notation is thus consistent.
When the point $(x, \zeta)$ is clear from the context, we abbreviate
\begin{equation*}
L := L_{[x, \zeta]} \leq G, \qquad \qquad \Phi^{\bullet}_L := \Phi^{\bullet}_{[x, \zeta]} \subseteq \Phi^{\bullet}_G, \qquad \qquad W_L := W_{[x, \zeta]} \leq W. 
\end{equation*}

\subsubsection{Resonant cocharacter.}
Consider $(x, \zeta)$ as above.
Let $\beta \in \Phi^{\bullet}_L$ be a resonant root. By definition of $L$, there exists an integer $m_{\beta} \in \Z$ such that $\beta(x) = \zeta^{m_{\beta}}$. Put differently, we have $m_{\beta} \in \log_{\zeta}(\beta(x))$. The integer $m_{\beta}$ is unique modulo $\ell = \ord(\zeta)$. We may package such integers as follows.

\begin{construction}\label{notation: additive function m}
We choose an additive function   
\begin{equation*}
   \mm : \Phi^{\bullet}_L \to \Z, \qquad \beta \mapsto m_{\beta}
\end{equation*}
such that $\beta(x) = \zeta^{m_{\beta}}$ for any $\beta \in \Phi^{\bullet}_L$.
\end{construction}

\begin{proof}
The above function is constructed as follows. 
When $\zeta$ is generic, meaning $\ell=0$, all its integer powers are distinct, so the condition $\beta(x) = \zeta^{m_{\beta}}$ determines $m_{\beta} \in \Z$ uniquely. On the other hand, when $\zeta$ is a root of unity of order $\ell \geq 1$, this condition is $\ell$-periodic, so $m_{\beta}$ is unique modulo $\ell$. In both cases, we obtain a map
\begin{equation*}
    \Phi^{\bullet}_L \to \Z / \ell \Z, \qquad \beta \mapsto [m_{\beta}].
\end{equation*}
This map is additive. 
Indeed, if $\alpha(x) = \zeta^{m_{\alpha}}$ and $\beta(x) = \zeta^{m_{\beta}}$, then $(\alpha + \beta)(x) = \alpha(x)\cdot \beta(x) = \zeta^{-(m_{\alpha}+m_{\beta})}$, so $[m_{\alpha + \beta}] \equiv [m_{\alpha}] + [m_{\beta}]$ modulo $\ell$ as desired.
We now choose an additive lift
\begin{equation*}
   \Phi^{\bullet}_L \to \Z, \qquad \beta \mapsto m_{\beta}. 
\end{equation*}
When $\zeta$ is generic, there is nothing to choose. When $\zeta$ is a root of unity of $\ord(\zeta) = \ell \geq 1$, we choose arbitrarily the lifts $m_{\alpha_i} \in \Z$ of $[m_{\alpha_i}]$ for all simple roots $\alpha_i \in \Phi^{\bullet, \sr}_L$.
Since simple roots are linearly independent and generate all roots, this uniquely extends the lift to $\Phi^{\bullet}_L$ by additivity.
\end{proof}

The above choice of $\mm$ may be equivalently packaged into a cocharacter $\omega$ of the adjoint form of $L$, which we call a resonant cocharacter at $(x, \zeta)$.
\begin{construction}
Let $T^{\ad}_L$ be the maximal torus of the adjoint form $L^{\ad}$ of the resonant subgroup $L$.
We consider a cocharacter
\begin{equation*}
 \omega = \omega_{(x, \zeta)} \in X_{\bullet}(T^{\ad}_L)   
\end{equation*}
with the property that $\beta(\omega(\zeta)) = \beta(x)$, $\forall \beta \in \Phi^{\bullet}_L$ and call it a \emph{resonant cocharacter} at $(x, \zeta)$.
\end{construction}
\begin{proof}
Such a cocharacter always exists. Indeed, for each $\beta \in \Phi^{\bullet}_L$ we have $\beta(x) = \zeta^{m_{\beta}}$ for the additive function $\mm: \Phi^{\bullet}_L \to \Z$. We thus need to construct a cocharacter $\omega$ such that $\beta(\omega(x)) = \zeta^{m_{\beta}}$ for any $\beta \in \Phi^{\bullet}_L$. By additivity of $m_{\beta}$, it is enough to check this property only on simple roots $\alpha_i \in \Phi^{\bullet, \sr}_L$. These form a free basis of the root lattice $\Z\Phi^{\bullet}_L$, which is also equal to the character lattice of the adjoint form $X^{\bullet, \ad}_{L}$. Hence we have the dual basis of cocharacters $\omega_i \in X^{{\ad}}_{\bullet, L}$. Setting $\omega := \sum_i m_{\alpha_i} \cdot \omega_i$ gives the desired cocharacter.
\end{proof}

\begin{remark}\label{remark: non-uniqueness of the resonant part}
Given the choice of $\mm$, we have defined $\omega$ by the property $\langle \beta, \omega \rangle = m_{\beta}$, $\forall \beta \in \Phi^{\bullet}_L$. The ambiguity in choosing $\omega$ is the same as the ambiguity in choosing the additive function $\mm$.
\begin{enumerate}
    \item The evaluation map $\ev_{\zeta}: X^{\bullet, \ad}_{L} \to k$, $\lambda \mapsto \lambda(\zeta)$ is not injective unless $\zeta$ is generic. For generic $\zeta$, such $\omega$ is unique. For $\zeta$ root of unity of $\ord(\zeta) = \ell$, it is determined up to $\ell X_{\bullet}^{{\ad}}$.
    \item To define $\omega$, we have to pass to the adjoint form $L^{\ad}$. The constructed cocharacter $\omega$ of $T^{\ad}_{L}$ may not lift to a cocharacter of a maximal torus $T_L$ of $L$. Even if it does, such a lift of $\omega$ in $T_L$ is not unique.
\end{enumerate}
\end{remark}

\subsubsection{Parabolic and Levi subgroups of $L$.}\label{section: parabolic and levi subgrops of L}
We consider $(x, \zeta)$ as above and work inside the resonant subgroup $L=L_{[x, \zeta]}$.
Consider the subgroup $P_x \leq L$ spanned by the same maximal torus and root spaces 
\begin{equation}\label{equation: roots of Px}
 \Phi^{\bullet}_{P_x} := \{\beta \in \Phi^{\bullet}_L \mid \beta(x) = \zeta^m \ \text{for some} \ m \leq 0 \}
\end{equation}
For each root $\beta \in \Phi^{\bullet}_L$, at least one of $\pm \beta$ lies inside $\Phi^{\bullet}_{P_x}$. Applying this observation to simple roots of $L$, we conclude that $P_x$ contains a Borel subgroup of $L$. Consequently, $P_x \leq L$ is a parabolic subgroup.
By the definition of $\omega$, we further have $\beta(x) = \beta(\omega(\zeta)) = \zeta^{\langle \beta, \omega \rangle}$. Therefore, $\Phi^{\bullet}_{P_x} = \{\beta \in \Phi^{\bullet}_L \mid \langle \beta, \omega \rangle \leq 0 \}$. These are precisely the roots of the opposite parabolic subgroup $P_{\omega} \leq L$, hence $P_x = P_{\omega}$.

Moreover, for any $\beta \in \Phi^{\bullet}_L$, both $\pm \beta$ lie in $\Phi^{\bullet}_{P_x}$ if and only if $\beta(x) = 1$. This condition cuts out the root system $ \Phi^{\bullet}_{x} = \{\beta \in \Phi^{\bullet}_L \mid \beta(x) = 1 \}$ of the stabilizer subgroup $L_x$. We deduce that $L_x$ is the Levi factor of $P_x$. As in the above paragraph, this may be rewritten as $\Phi^{\bullet}_{x} = \{\beta \in \Phi^{\bullet}_L \mid \langle \beta, \omega \rangle = 0 \}$, which is the root subsystem of $L_{\omega} \leq L$.

Altogether, we have proved the following.
\begin{lemma}\label{lemma: parabolic and levi subgrops of L}
The stabilizer $L_x \leq L$ is a Levi subgroup of $L$. It is the Levi factors of the parabolic subgroup $P_x \leq L$. Moreover, $L_x = L_{\omega}$ and $P_x = P_{\omega}$ are the Levi and opposite parabolic associated to the resonant cocharacter $\omega \in X_{\bullet, L}^{{\ad}}$. 
\end{lemma}

\begin{remark}\label{remark: parabolic and levi subgrops of L}
Let $y = x \cdot \lambda(\zeta)$ for some $\lambda \in \Z\Phi^{\bullet}_G$ be an element in the $\zeta$-resonance class of $x$. Then the resonant cocharacter of $y$ is given by $\omega + \lambda^{\ad}$, where $\lambda^{\ad}$ is the projection of $\lambda$ to the cocharacter lattice of $L^{\ad}$ along $X_{\bullet, G} = X_{\bullet, L} \twoheadrightarrow X_{\bullet, L}^{\ad}$. Since $L = L_{[x, \zeta]} = L_{[y, \zeta]}$ depends only on the resonance class of $x$, Lemma \ref{lemma: parabolic and levi subgrops of L} follows for $y$.

Concretely, we deduce $L_y \leq L$ is a Levi subgroup of $L$. It is the Levi factor of the parabolic subgroup $P_{y}$. Moreover, $L_y = L_{\omega + \lambda^{\ad}}$ and $P_{\omega + \lambda^{\ad}}$ are the Levi and opposite parabolic associated to the resonant cocharacter $\omega + \lambda^{\ad} \in X_{\bullet, L}^{{\ad}}$.
\end{remark}

\subsubsection{Independence on Weyl group actions.}\label{subsubsection: independence on W-action}
Fix the point $\zeta \in \Gm^{\rot}$. Any closed point $x \in T$ projects to a closed point $\overline{x} = x \cdot W$ of $T\sslash W$.

The resonant subgroup $L_{[x, \zeta]}$ and its Weyl group $W_{[x, \zeta]}$ depend, up to canonical isomorphism given by conjugation, only on $\overline{x}$. 
Since they also depend only on the $\zeta$-resonance class of $x$ by definition, they are up to canonical isomorphism attached to the orbit of the $\zeta$-integral action of the affine Weyl group $W_{\aff} \acts T$, $(w, \lambda) \cdot x = \lambda(\zeta) \cdot w(x)$. Also see \S \ref{subsubsection: resonance} and \S \ref{section: actions I} below.

Similarly, the stabilizing Weyl group $W_x$, the stabilizing subgroup $L_x$ and the associated parabolic subgroup $P_x$ of $L_{[x, \zeta]}$ depend, up to canonical isomorphism given by conjugation, only on $\overline{x}$.

In particular, we have shown the following.
\begin{claim}
The pair $P_x \leq L_{[x, \zeta]}$ depends, up to isomorphism, only on $\overline{x} \in T \sslash W$ and $\zeta \in \Gm^{\rot}$.   
\end{claim}

\subsubsection{Vocabulary to Humphrey's notation.}
When $G$ is of adjoint type and $\zeta$ is generic, the above constructions appear in the study of the classical category $\O$, see \cite{Hum08, Jan79}. More precisely, \cite{Hum08} works (under Langlands dual notation) with the BGG category $\O$ of the semisimple Lie algebra $\widecheck{\g}$. The genericity of $\zeta$ corresponds to the classical version of the category $\O$ (as opposed to its quantum version at a root of unity). In this setup, the (co)character $\omega$ is unique (as we saw in Remark \ref{remark: non-uniqueness of the resonant part}). For convenience, we provide a short vocabulary to the notation from \cite{Hum08} in this special case -- compare:
    \begin{itemize}
        \item  the definition of $W_L = W_{[x, \zeta]}$ to $W_{[c]}$ of \cite[\S 3.4]{Hum08},
        \item the definition of $\omega = \omega_{[x, \zeta]}$ to the definition of $c^{\natural}$ of \cite[\S 7.4]{Hum08},
        \item the definition of $W_x$ to the definition of $W^{\circ}_c$ of \cite[\S 7.4]{Hum08}.
    \end{itemize}

\subsection{Centralizers in loop groups}\label{subsection: centralizers in loop groups}

\subsubsection{Centralizers in loop groups.}\label{section: centralizers}

Let $(x, \zeta) \in T \times \Gm^{\rot}$. Consider its centralizer $\Cent_{\lo G}(x, \zeta) \leq \lo G$ in the loop group $\lo G$. It might be not connected; we denote the identity component by
\begin{equation}
J := J_{(x, \zeta)} := \Cent^{\circ}_{\lo G}(x, \zeta) \leq \lo G.
\end{equation}
Furthermore, consider the positive part the connected centralizer $J$, given by
\begin{equation}
J^+ 
:= J \cap \lop G \leq \lop G   
\end{equation}
Equivalently, $J^+ = \Stab_{J}(t^{0})$ is the stabilizer inside $J$ of the base point $t^{0} \in \Gr_G$. 

More generally, for any cocharacter $\lambda \in X_{\bullet}$, we consider
\begin{equation}\label{equation: positive stabilizer}
    J^{+}_{\lambda} := J \cap t^{\lambda} \cdot \lop G \cdot t^{-\lambda},
\end{equation}
the stabilizer $\Stab_{J}(t^{\lambda})$ inside $J$ of the point $t^{\lambda} \in \Gr_G$.

\begin{remark}
It will become apparent later -- see Lemmata \ref{lemma: positive stabilizers as parabolics -- generic zeta} and \ref{lemma: zeta root of unity -- positive stabilizers} -- that all $J^{+}_{\lambda}$ are already connected. For example, this a posteriori means $J^{+}$ is given by the connected centralizer $J^+ = \Cent^{\circ}_{\lop G}(x, \zeta)$ of $(x, \zeta)$ inside the positive loop group.   
\end{remark}

\subsubsection{Root system of $J$.}
We can conveniently describe the root system of $J$ as follows.
\begin{lemma}\label{lemma: root spaces of J}
Let $(x, \zeta) \in T \times \Gm^{\rot}$ arbitrary and denote $\ell= \ord(\zeta) \in \Z_{\geq 0}$. Then the root subsystem $\Phi^{\bullet}_J \subseteq \Phi^{\bullet}_{G, \aff}$ of $J = J_{(x, \zeta)}$ has the form    
\begin{align*}
\Phi^{\bullet}_J 
&= \{ \beta + m\hbar \in \Phi^{\bullet}_{G, \aff} \mid \beta(x) \cdot \zeta^{m} = 1 \} \\
&= \{\beta - (m_{\beta} + \ell \Z)\hbar \in \Phi^{\bullet}_{G, \aff} \mid \beta \in \Phi^{\bullet}_{L}\}.
\end{align*}
for the additive function $\mm: \Phi^{\bullet}_L \to \Z$, $\beta \mapsto m_{\beta}$ from Notation \ref{notation: additive function m}.
\end{lemma}

\begin{proof} 
The centralizer $J$ is determined by the closed root subsystem $\Phi^{\bullet}_J$ of the affine real root system $\Phi^{\bullet}_{G, \aff}$ of $G$ by Lemma \ref{lemma: centralizers give closed subroot systems}. Since the adjoint action of $(x, \zeta)$ on a root space $U_{\beta + m\hbar}$ is given by $\beta(x)\cdot \zeta^m$, this closed root system is given by $\Phi^{\bullet}_J = \{ \beta + m\hbar \in \Phi^{\bullet}_{G, \aff} \mid \beta(x) \cdot \zeta^{m} = 1 \}$ as in Lemma \ref{proposition: affine centralizer as regular subgroup}. The first equality is thus clear.

For the second equality, consider the resonant subgroup $L=L_{[x, \zeta]} \leq G$ with root system $\Phi^{\bullet}_L \subseteq \Phi^{\bullet}_G$. Note that $\Phi^{\bullet}_{L}$ is given by the image of $\Phi^{\bullet}_J$ under the projection \eqref{equation: projection of root systems}. Hence any root $\beta +m\hbar \in \Phi^{\bullet}_J$ has $\beta\in \Phi^{\bullet}_L$.
By Construction \ref{notation: additive function m} of the additive function $\mm$, the subset $m_{\beta} + \ell \Z \subseteq \Z$ contains precisely those integers $m$ satisfying $\beta(x) = \zeta^m$. The second equality follows.
\end{proof}

\begin{notation}\label{notation: WJ}
We denote by $W_J := W_{\Phi^{\bullet}_J} \leq W_{G, \aff}$ the reflection subgroup generated by the affine reflections along roots from $\Phi^{\bullet}_J \leq \Phi^{\bullet}_{G, \aff}$.    
\end{notation}

\begin{remark}\label{remark: root system of J project on root system of L}
Note that the resonant root system $\Phi^{\bullet}_{[x, \zeta]} = \Phi^{\bullet}_L$ from \eqref{equation: resonant root system} is the finite root subsystem given by the projection of the affine root system $\Phi^{\bullet}_J$. In terms of \eqref{equation: projection of root systems}, this reads as
\begin{equation*}
   \Phi^{\bullet}_{[x, \zeta]} = \overline{\Phi}^{\bullet}_J.
\end{equation*}
\end{remark}

\begin{remark}
Since $J$ is connected, it is determined by the root subsystem $\Phi^{\bullet}_J$. Also note that $J \leq \lo L^{\sc}$. Indeed, it is connected and contains only root spaces from $\Phi^{\bullet}_{L, \aff}$.    
\end{remark}

\subsubsection{Slanting.}\label{section: slanting}
We can get rid of the integers $m_{\beta}$ by a suitable automorphism of $\lo L$, which we call \emph{slanting}. Geometrically, slanting is given by conjugation with $t^{\omega}$, where $\omega$ is the resonant cocharacter. 

\begin{discussion}\label{discussion: slanting}
The adjoint action of a reductive group on itself factors through its adjoint form. Since $L^{\ad} \acts L$ by conjugation, we have an induced action $\lo L^{\ad} \acts \lo L$. For any cocharacter $\lambda \in X^{{\ad}}_{L, \bullet}$, conjugation by $t^{\lambda}$ gives an automorphism of $\lo L$. We call this map \emph{slanting by $\lambda$} and denote it
\begin{equation}
    \slant_{\lambda}: \lo L \to \lo L, \qquad g \mapsto t^{\lambda}\cdot g \cdot t^{-\lambda}
\end{equation}
By the same argument, slanting restricts to an automorphism of the identity component $\slant_{\lambda}: \lo L^{\sc} \to \lo L^{\sc}$. On the level of its root system, slanting is induced by
\begin{equation}\label{equation: slanting on root data}
    \slant_{\lambda}: \Phi^{\bullet}_{L, \aff} \to \Phi^{\bullet}_{L, \aff}, 
    \qquad
    \beta + m\hbar \mapsto \beta + (m + \langle \beta, \lambda \rangle)\hbar.
\end{equation}
Clearly, $\slant_{\lambda}$ and $\slant_{-\lambda}$ are inverse to each other.
Using this notation, any subgroup $H \leq \lo L$ can be slanted to an isomorphic copy $\slant_{\lambda}(H) := t^{\lambda} \cdot H \cdot t^{-\lambda} \leq \lo L$.

Furthermore, slanting on root systems \eqref{equation: slanting on root data} induced a map
\begin{equation}\label{equation: slanting on affine weyl groups}
    \slant_{\lambda}: W_{L, \aff} \to W_{L, \aff}.
\end{equation}
Any subgroup of $W_{L, \aff}$ may be again slanted by the automorphism $\slant_{\omega}$. Under this automorphism, the natural action $W_{L, \aff} \acts X_{\bullet}$ identifies with the $\lambda$-shifted action $(\slant_{\lambda}(W_{L, \aff}, \bullet_{\lambda}) \acts X_{\bullet}$.
Indeed, this is clear from the defining formulas \S \ref{section: shifted actions}, \S \ref{section: shifted affine weyl group actions} and \eqref{equation: slanting on root data}.
\end{discussion}

In particular, we may slant by the resonant cocharacter $\omega$. 
Slanting $J \leq \lo L^{\sc}$, we obtain the embedding $\slant_{\omega}(J) = t^{\omega} \cdot J \cdot t^{-\omega} \leq \lo L^{\sc}$.
As we saw, this embedding can be defined purely on the level of associated root systems.
\begin{lemma}\label{discussion: slanting J in general}
The closed root subsystem of $\slant_{\omega}(J)$ is given by
\begin{equation*}
    \Phi^{\bullet}_{\slant_{\omega}(J)} = \{ \beta + \ell \Z \mid \beta \in \Phi^{\bullet}_L \}.
\end{equation*} 
\end{lemma}
\begin{proof}
Since conjugation by $t^{\omega}$ sends the root space $U_{\beta + m\hbar}$ to $U_{\beta + (m + \langle \omega, \beta \rangle \hbar )} = U_{\beta + (m + m_{\beta}) \hbar}$ by Lemma \ref{lemma: conjugation of affine root spaces}, the slanting map on root systems identifies
\begin{align*}
   \Phi^{\bullet}_J &\xrightarrow{=} \Phi^{\bullet}_{\slant_{\omega}(J)} \\
   \beta + m\hbar &\mapsto \beta + (m + m_{\beta})\hbar.
\end{align*}
The result then follows from the description of $\Phi^{\bullet}_J$ in Lemma \ref{lemma: root spaces of J}.
\end{proof}

The slanted subgroup $t^{\omega} \cdot J \cdot t^{-\omega}$ is simpler to manipulate -- as we will record in Lemmata \ref{lemma: centralizer as resonant subgroup -- generic zeta}, \ref{lemma: zeta root of unity -- computation of centralizer}, it is given by the reductive group $L$ resp. the identity component of its dilated loop group, embedded in the naive way. Vice versa, the original centralizer $J$ is then given by slanting these naive embeddings via $\slant_{-\omega}$.

\begin{remark}\label{remark: slanting on weyl groups general case}
The slanting map on root data $\slant_{\omega}: \Phi^{\bullet}_J \xrightarrow{=} \Phi^{\bullet}_{\slant_{\omega}(J)}$ induced a map on the corresponding reflection subgroups
\begin{equation*}
    \slant_{\omega}: W_J \xrightarrow{=} W_{\slant_{\omega}(J)}
\end{equation*}
of $W_{L, \aff}$. By Discussion \ref{discussion: slanting}, $\slant_{\omega}$ intertwines the naive action and the $\omega$-shifted action on $X_{\bullet}$.
\end{remark}

\subsection{Weyl group stabilizers for simply connected \texorpdfstring{$G$}{G}}\label{subsection: weyl group stabilizers}
Let $k$ be algebraically closed of characteristic zero.
Assuming $G = G^{\sc}$ is simply connected, we discuss stabilizers of finite and affine Weyl group actions on its maximal torus $T = T^{\sc}$. The simply connected assumption is needed for technical reasons -- even in the reductive situation, one has to consider stabilizers for $T^{\sc}$.

\subsubsection{Affine Weyl group action on extended torus.}\label{section: actions I}
We have the defining action of the finite Weyl group $W \acts T$. It lifts to a natural action of the affine Weyl group on the extended torus
\begin{align}\label{equation: affine weyl group action I}
\begin{split}
W_{\aff} &\acts \Gm^{\rot} \times T \\
(\lambda, w) \cdot (\zeta, x) &= (\zeta, \lambda(\zeta) \cdot w(x)).
\end{split}
\end{align}
It is useful to view this as an algebraic family over $\Gm^{\rot}$ of actions of $W_{\aff} = \Z\Phi^{\bullet} \rtimes \Gm^{\rot}$ on the torus $T$. Namely, for fixed $\zeta \in \Gm^{\rot}$, we have the \emph{$\zeta$-integral action}
\begin{equation}\label{equation: zeta integral action of affine Weyl group}
W_{\aff} \acts T, \qquad (\lambda, w) \cdot x = \lambda(\zeta) \cdot w(x).    
\end{equation}

\subsubsection{Weyl group stabilizers as reflection subgroups.}
Consider a closed point $x \in T=T^{\sc}$.
\begin{lemma}\label{lemma: stabilizer in finite weyl group}
The stabilizer of $x$ under the action of $W \acts T$ is given by
\begin{equation*}
    \Stab_W(x) = W_x.
\end{equation*}   
\end{lemma}
\begin{proof}
The stabilizer $\Stab_W(x) \leq W$ is a reflection subgroup by \S \ref{section: stabilizers in weyl groups as reflection subgroups}, Proposition \ref{proposition: steinberg fixed point theorem}.
Hence we only have to describe which reflections $s_{\beta}$, $\beta \in \Phi^{\bullet}_G$ lie inside it. Since $s_{\beta}(x) = x \cdot \widecheck{\beta}(\beta(x))^{-1}$, we see that $s_{\beta}$ fixes $x$ if and only if $\beta(x)=1$. Consequently, $\Stab_W(x)$ is the reflection subgroup of the root subsystem $\{ \beta \in \Phi^{\bullet}_G \mid \beta(x)=1 \}$. By definition, this is the root subsystem $\Phi^{\bullet}_x$ defining $W_x \leq W$. Also see Proposition \ref{proposition: pseudo-levi subgroups}, (iii).
\end{proof}

Taking a closed point $(x, \zeta) \in T \times \Gm^{\rot} = T^{\sc} \times \Gm^{\rot}$, we have the following affine analogue. To state it, recall from Notation \ref{notation: WJ} that $W_J = W_{\Phi^{\bullet}_J}$ stands for the reflection subgroup of $W_{\aff}$ associated to the root subsystem $\Phi^{\bullet}_J$ of $J$. This root subsystem was described in Lemma \ref{lemma: root spaces of J}.

\begin{lemma}\label{lemma: affine stabilizer in weyl group}
The stabilizer of $(x, \zeta)$ under the action $W_{\aff} \acts T \times \Gm^{\rot}$ is given by
\begin{equation*}
    \Stab_{W_{\aff}}(x, \zeta) = W_J.
\end{equation*}
In other words, the stabilizer of $x \in T$ under the $\zeta$-integral action of $W_{\aff}$ is given by $W_J$.
\end{lemma}
\begin{proof}
Fix $\zeta \in \Gm^{\rot}$. Then $\Stab_{W_{\aff}}(x, \zeta)$ is given by the stabilizer $\Stab_{W_{\aff}}(x)$ under the $\zeta$-integral action \eqref{equation: zeta integral action of affine Weyl group}. 
As in the finite case, the stabilizer $\Stab_{W_{\aff}}(x)$ is a reflection subgroup by Proposition \ref{proposition: steinberg fixed point theorem}. 
Hence it is enough to determine which affine reflections $s_{\beta + m\hbar} \in \Phi^{\bullet}_{G, \aff}$ lie inside it. Since $s_{\beta + m\hbar}(x) = x \cdot \widecheck{\beta}(\beta(x) \cdot \zeta^m)^{-1}$, the reflection $s_{\beta + m\hbar}$ fixes $x$ if and only if $\beta(x) \cdot \zeta^m =1$. Hence $\Stab_{W_{\aff}}(x)$ is the reflection subgroup associated to the root system $\{\beta + m\hbar \in \Phi^{\bullet}_{G, \aff} \mid \beta(x) \cdot \zeta^m =1 \}$. Since this is the root system $\Phi^{\bullet}_J$ of $J$, the result follows.
\end{proof}

As a corollary, we deduce that the orbit of $x \in T$ under the $\zeta$-integral action $W_{\aff} \acts T$ is given by $W_{\aff} / W_J$. These orbits will be discussed further in \S \ref{section: some useful notation} and prominently appear as the labeling set for connected componets of fixed points on $\Gr_G$.

\subsubsection{Image of stabilizer in finite Weyl group.}

We may easily describe the image of the stabilizer under the projection to the finite Weyl group.

\begin{lemma}
Let $(x, \zeta) \in T \times \Gm^{\rot}$.
Under the projection map $W_{\aff} \twoheadrightarrow W$, the image of $\Stab_{W_{\aff}}(x, \zeta)$ is given by the resonant Weyl group $W_{[x, \zeta]} \leq W$.
\end{lemma}
\begin{proof}
The projection in question is induced by the projection of root systems $\Phi^{\bullet}_{G, \aff} \twoheadrightarrow \Phi^{\bullet}_{G}$. By Lemma \ref{lemma: affine stabilizer in weyl group}, $\Stab_{W_{\aff}}(x, \zeta) = W_{J}$ is the reflection subgroup associated to the subsystem $\Phi^{\bullet}_J$, whose projection is precisely the root subsystem $\Phi^{\bullet}_{[x, \zeta]} $ of $W_{[x, \zeta]}$ by Remark \ref{remark: root system of J project on root system of L}. The result follows.
\end{proof}

If $\zeta \in \Gm - \mu_{\infty}$ is generic, the above stabilizer $\Stab_{W_{\aff}}(x, \zeta) = W_J$ thus fits in the following diagram whose rows are short exact sequences.
\begin{equation}\label{equation: stabilizer exact sequences for zeta generic}
    \begin{tikzcd}
        1 \arrow[r] & \Z\Phi_{\bullet} \arrow[r] & W_{\aff} \arrow[r] & W \arrow[r] & 1 \\
        1 \arrow[r] & 1 \arrow[r] \arrow[u, hook] & \Stab_{W_{\aff}}(x, \zeta) \arrow[r , "="] \arrow[u, hook] & W_{[x, \zeta]} \arrow[r] \arrow[u, hook] & 1
    \end{tikzcd}
\end{equation}

On the other hand, assume now $\zeta \in \mu_{\infty}$ is primitive root of unity of $\ord(\zeta) = \ell$. Then $\Stab_{W_{\aff}}(x, \zeta) = W_J$ fits in the diagram of horizontal short exact sequences
\begin{equation}\label{equation: stabilizer exact sequences for zeta root of unity}
    \begin{tikzcd}
        1 \arrow[r] & \Z\Phi_{\bullet} \arrow[r] & W_{\aff} \arrow[r] & W \arrow[r] & 1 \\
        1 \arrow[r]& \ell \Z\Phi_{\bullet} \arrow[r] \arrow[u, hook] & \Stab_{W_{\aff}}(x, \zeta) \arrow[r] \arrow[u, hook] & W_{[x, \zeta]} \arrow[r] \arrow[u, hook] & 1 \\
        1 \arrow[r]  & \ell \Z\Phi_{\bullet} \arrow[r] \arrow[u, hook] & \ell \Z\Phi_{\bullet} \rtimes W_{x} \arrow[r] \arrow[u, hook] & W_{x} \arrow[r] \arrow[u, hook] & 1 
    \end{tikzcd}
\end{equation}
Here, the bottom row is given by the separate stabilizers of $x$ inside $\Z\Phi_{\bullet}$ and $W$. Note that the actual stabilizer in the middle row is bigger that their semidirect product.

To wrap up, if $G$ is simply connected, the reflection subgroups \eqref{equation: resonant weyl group} are equivalently given by
\begin{equation}\label{eq: weyl group and resonant weyl group}
   W_x = \{w \in W \mid wx = x\} 
   \qquad \text{and} \qquad  
   W_{[x, \zeta]} = \{w \in W \mid wx \sim_{\zeta} x\}.
\end{equation}
In words, $W_x$ is the stabilizer of $x \in T$, while $W_{[x, \zeta]}$ is the stabilizer of its resonance class $[x, \zeta]$.

\subsection{Cocharacter graphs for simply connected \texorpdfstring{$G$}{G}}\label{section: cocharacter graphs}
The above combinatorics is conveniently encoded by the \textit{cocharacter graphs} $\overline{\Gamma} \subseteq T \times \Gm^{\rot} \times T \sslash W$ introduced in this section. 

To simplify the exposition, we will again assume $G =G^{\sc}$ to be simply connected throughout \S \ref{section: cocharacter graphs}. With extra bookkeeping, the same discussion works for any reductive group; the necessary modifications are described afterwards in \S \ref{section: cocharacter graphs for reductive groups}.

\subsubsection{Cocharacter graphs for simply connected groups.}\label{subsubsection: cocharacter graphs for simply connected groups}
Assume $G$ is simply connected. Then $\Z\Phi^{\bullet}_G = X_{\bullet}$ and $W_{\aff} = W_{\ext}$. Recall the action \eqref{equation: affine weyl group action I}. For any element $(\lambda, w) \in W_{\aff}$ with $\lambda \in X_{\bullet}$ and $w \in W$, consider the map
\begin{align*}
    (\lambda, w): T \times \Gm^{\rot} &\to T \\
    (x, \zeta) &\mapsto  \lambda(\zeta) \cdot w(x).
\end{align*}
and the closed subscheme $\Gamma_{\lambda, w} \hookrightarrow T \times \Gm^{\rot} \times T$ given by its graph
\begin{align*}
    \Gamma_{\lambda, w} 
    := &\{ (x, \zeta, \lambda(\zeta) \cdot w(x)) \in T \times \Gm^{\rot} \times T \} \\
    = &\{ (x, \zeta, y) \in T \times \Gm^{\rot} \times T \mid y = \lambda(\zeta) \cdot w(x) \}.
\end{align*}
Taking the union of these closed subschemes over varying $\lambda$, we obtain the closed ind-subscheme
\begin{equation*}
    \Gamma := \{ (x, \zeta, y) \in T \times \Gm^{\rot} \times T \mid \exists \lambda \in X_{\bullet}, \ y \cdot w(x)^{-1} = \lambda(\zeta) \} \hookrightarrow T \times \Gm^{\rot} \times T.
\end{equation*}
The ind-scheme $\Gamma$ is reduced. Its irreducible componets $\Gamma_{\lambda, w}$ are copies of $T \times \Gm^{\rot}$ labeled by $W_{\aff}$, intersecting in various ways. In particular, we have the distinguished component $\Gamma_{1, 1}$, parametrizing the points $\{ (x, \zeta, x) \in T \times \Gm^{\rot} \times T\}$.

\subsubsection{Action of affine Weyl group.}\label{section: actions}
Consider the action of $W_{\aff}$ on $T \times \Gm^{\rot} \times T$ given by the trivial action on the first copy of $T$ and the natural action \eqref{equation: affine weyl group action I} on $\Gm^{\rot} \times T$, namely
\begin{equation*}
   W_{\aff} \acts T \times \Gm^{\rot} \times T, \qquad (\lambda, w) \cdot (x, \zeta, y) = (x, \zeta, \lambda(\zeta) \cdot w(y)). 
\end{equation*}
This action restricts to an action
\begin{equation}\label{equation: affine weyl group action on cocharacter graphs}
W_{\aff} \acts \Gamma    
\end{equation}
on the closed sub-ind scheme $\Gamma \hookrightarrow T \times \Gm^{\rot} \times T$. It acts freely and transitively on the set of irreducible components $(\Gamma_{\lambda, w} \mid (\lambda, w) \in W_{\aff})$ of $\Gamma$.

\subsubsection{Variants.}\label{section: variants}
Given any pair of subgroups $W_1, W_2 \leq W$, consider their action on $T \times T$. The ind-subscheme $\Gamma \hookrightarrow T \times \Gm^{\rot} \times T$ is clearly $(W_1 \times W_2)$-equivariant. Taking the quotient, we get the sub-ind scheme
\begin{equation*}
    \Gamma^{T \sslash W_1,  T \sslash W_2} \hookrightarrow T\sslash W_1 \times \Gm^{\rot} \times T \sslash W_2.
\end{equation*}
We will be particularly interested in the case $W_1= 1$ and $W_2 = W$. We then write
\begin{equation*}
    \overline{\Gamma} := \Gamma^{T,  T \sslash W} \hookrightarrow T \times \Gm^{\rot} \times T \sslash W.
\end{equation*}

The action of $W_2=W$ on the right copy of $T$ permutes the irreducible components of $\Gamma$ according to the left translation on the labeling set $W \acts (X_{\bullet} \rtimes W)$. In other words, $\forall w \in W$, we have $w \cdot \Gamma_{\lambda, v} = \Gamma_{w(\lambda), wv}$.
The irreducible components of $\overline{\Gamma}$ are thus copies of $T \times \Gm^{\rot}$, labeled by the quotient set $X_{\bullet} = W \backslash (X_{\bullet} \rtimes W)$. They correspond to $W$-orbits of irreducible components of $\Gamma$. 
For each $\lambda \in X_{\bullet}$, we denote by $\overline{\Gamma}_{\lambda}$ the component corresponding to the $W$-orbit of $\Gamma_{\lambda, 1}$.

\subsubsection{Fibers.}\label{section: fibers}
Given $\Gamma$, $\overline{\Gamma}$ and the other variants, we may consider their fibers over points in the base $T \times \Gm^{\rot}$. 
Given any $(x, \zeta) \in T \times \Gm^{\rot}$, we denote these fibers as
\begin{equation*}
    \Gamma_{(x, \zeta)} \hookrightarrow T 
    \qquad
    \text{and}
    \qquad
    \overline{\Gamma}_{(x, \zeta)} \hookrightarrow T\sslash W.
\end{equation*}
These are closed-ind schemes. Their underlying reduced schemes are given by, possibly infinite, union of points.
Similarly, we denote the fibers over some $\zeta \in \Gm^{\rot}$ by
\begin{equation*}
    \Gamma_{\zeta} \hookrightarrow  T \times T 
    \qquad
    \text{and}
    \qquad
    \overline{\Gamma}_{\zeta} \hookrightarrow T \times T\sslash W.
\end{equation*}
These are again closed sub-ind-schemes. The irreducible components of the reduces scheme underlying $\Gamma_{\zeta}$ are given by copies of $T$.

\subsubsection{Fibers of $\Gamma$ as orbits on cocharacter lattice.}\label{section: specializations of the lattice}
For any $(x, \zeta) \in T \times \Gm^{\rot}$, the action $W_{\aff} \acts \Gamma$ from \eqref{equation: affine weyl group action on cocharacter graphs} restricts to a transitive action
\begin{equation}\label{equation: transitive action of Waff}
  W_{\aff} \acts |\Gamma_{(x, \zeta)}|
\end{equation}
on the discrete set of closed points of the fiber $\Gamma_{(x, \zeta)}$. Under the embedding to $|T|$, the set $|\Gamma_{(x, \zeta)}|$ is thus given by the $W_{\aff}$-orbit of $x$ under the $\zeta$-integral action \eqref{equation: zeta integral action of affine Weyl group}. 

\begin{lemma}\label{lemma: identification of fibers of Gamma}
We have a natural identification  
\begin{equation*}
 |\Gamma_{(x, \zeta)}| = (X_{\bullet} \rtimes W) /  W_J.
\end{equation*}
\end{lemma}
\begin{proof}
The stabilizer in $W_{\aff}$ of the base point $(x, \zeta, x) \in |\Gamma_{(x, \zeta)}| \subseteq |T \times \Gm^{\rot} \times T|$ is given as
\begin{equation*}
    \Stab_{W_{\aff}}(x, \zeta, x) = W_J.
\end{equation*}
by Lemma \ref{lemma: affine stabilizer in weyl group}. Considering the transitive action \eqref{equation: transitive action of Waff}, we conclude.
\end{proof}

\begin{remark}\label{remark: colisions of components of Gamma}
In other words: the irreducible components of $\Gamma$ are copies of $T \times \Gm^{\rot}$ labeled by $W_{\aff}$, whose collisions over the point $(x, \zeta)$ are described by the quotient map
\begin{equation*}
    W_{\aff} \twoheadrightarrow W_{\aff} / W_{J}.
\end{equation*}
Similarly, the irreducible components of the reduced ind-scheme underlying $\Gamma_{\zeta}$ are copies of $T$ labeled by $W_{\aff}/ \ell X_{\bullet}$. whose collisions over the point $x \in T$ are given by
$W_{\aff}/ \ell X_{\bullet} \twoheadrightarrow  (W_{\aff}/ \ell X_{\bullet})/ W_{L}$.
\end{remark}

We may also describe the fiber $|\overline{\Gamma}_{(x, \zeta)}| \hookrightarrow |T\sslash W|$ similarly to Lemma \ref{lemma: identification of fibers of Gamma}. The group $W_{\aff}$ acts on itself by left and right translations. In particular, we get the commuting actions
\begin{equation*}
    W \acts W_{\aff} \actsr W_J
\end{equation*}
Here, the left action is induced by the given embedding of $W$ into the semidirect product defining $W_{\aff}$. The quotient by this $W$ naturally identifies as $X_{\bullet} = W\backslash W_{\aff}$. This identification is compatible with the right actions of the stabilizer $W_J \leq W_{\aff}$.

\begin{lemma}\label{lemma: fiber of cocharacter graphs as lattice quotient}
We have a natural identification    
\begin{equation*}
 |\overline{\Gamma}_{(x, \zeta)}| = X_{\bullet} /  W_J.
\end{equation*}
\end{lemma}
\begin{proof}
With the above left $W$-action, the embedding
$W_{\aff} / \Stab_{W_{\aff}}((x, \zeta, x)) = |\Gamma_{(x, \zeta)}| \hookrightarrow T$
becomes $W$-equivariant. The claim follows from Lemma \ref{lemma: identification of fibers of Gamma} by taking quotient.    
\end{proof}

We make this quotient more explicit in Lemmata \ref{lemma: labeling set for components -- zeta generic}, \ref{lemma: labeling set for components -- zeta root of unity}, depending on the genericity of $\zeta$.

\subsubsection{Labeling set for components.}\label{section: some useful notation}
\begin{notation}\label{notation: Theta}
For notational convenience, we will denote the discrete set from Lemma \ref{lemma: fiber of cocharacter graphs as lattice quotient} by
\begin{equation*}
    \Theta := |\overline{\Gamma}_{(x, \zeta)}| = X_{\bullet} /  W_J.
\end{equation*}
Note that $\Theta$ implicitly depends on $(x, \zeta)$. An element $\vartheta \in \Theta$ may be equivalently regarded as:
\begin{enumerate}
    \item[(i)] a closed point of the fiber of $\overline{\Gamma}$ over $(x, \zeta)$ (by definition);
    \item[(ii)] a closed point of $|T\sslash W|$ which is $\zeta$-resonant with $x$ modulo $W$ (under $ |\overline{\Gamma}_{(x, \zeta)}| \subseteq |T \sslash W|$);
    \item[(iii)] an element in the quotient $X_{\bullet} /  W_J$ (under the orbit description of Lemma \ref{lemma: fiber of cocharacter graphs as lattice quotient})
    \item[(iv)] a closed point of the quasi-diagonal $Z_{\ell} \subseteq T \times T \sslash W$ (as in \S \ref{section: quasi-diagonals at roots of unity} by Lemma \ref{lemma: Z is Gamma} below).
\end{enumerate} 
\end{notation}

The set $\Theta$ will play an essential role in \S \ref{section: fixed points on affine grassmannians} as the labeling set for components of fixed points on the affine Grassmannian. To each point of $\Theta$, we can associate the stabilizing subgroups as follows.

\begin{notation}\label{notation: levi and paraolic subgroups labeled by elements of Theta}
Given $x\in T$, recall from \S \ref{subsubsection: independence on W-action} that the Weyl group $W_x$, the stabilizing subgroup $L_x$ and the associated parabolic subgroup $P_x$ of $L_{[x, \zeta]}$ depend, up to canonical isomorphism, only on its image $\overline{x}$ in $T \sslash W$.
Interpreting $\vartheta \in \Theta$ as a closed point of $T \sslash W$ with a lift $y \in T$, it is thus consistent to denote
\begin{equation*}
    W_{\vartheta} := W_{y} , \qquad
    L_{\vartheta} := L_{y} , \qquad
    P_{\vartheta} := P_{y} 
\end{equation*}
We have $W_{\vartheta} \leq W_{[x, \zeta]}$ and $L_{\vartheta} \leq P_{\vartheta}  \leq L_{[x, \zeta]}$. The isomorphism class of the subgroups $W_{\vartheta}$, $L_{\vartheta}$, $P_{\vartheta}$ varies as $\vartheta$ goes through $\Theta$.

We would like to stress that $W_{\vartheta}$, $L_{\vartheta}$, $P_{\vartheta}$ are \textbf{not} the subgroups associated to a cocharacter $\lambda \in X_{\bullet}$ lifting $\vartheta \in X_{\bullet}/W_J$. Rather, these subgroups are associated to the point $x \cdot \lambda(\zeta)$, or alternatively to the adjoint cocharacter $\omega + \lambda^{\ad} \in X^{\ad}_{\bullet, L}$. There is an extra shift by $\omega$.
\end{notation}

\subsection{Cocharacter graphs for reductive groups}\label{section: cocharacter graphs for reductive groups}

\subsubsection{Comments.}
Assume now $G$ to be reductive. For technical reasons, the above discussion needs to be modified. Firstly, we need to replace $W_{\aff}$ by $W_{\ext} = X_{\bullet} \rtimes W$. Furthermore, we need to replace the second copy of $T$ by $T^{\sc} \times \pi_1(G)$, the product of its simply connected cover and the fundamental group of $G$. At first in \S \ref{subsubsection: cocharacter graphs for reductive groups}, \S \ref{subsubsection: variants and stabilizers for reductive G}, we will assume $G$ is semisimple and also replace the first copy of $T$ by $T^{\sc}$; we show how to descend back to $T$ in \S \ref{section: action of the fundamental group}.

\subsubsection{Cocharacter graphs for reductive groups.}\label{subsubsection: cocharacter graphs for reductive groups}
We start by exteding the definitions from \S \ref{subsubsection: cocharacter graphs for simply connected groups}.
First, note that a character $\lambda: \Gm^{\rot} \to T$ lifts to a character $\Gm^{\rot} \to T^{\sc}$ if and only if its class in the fundamental group $\overline{\lambda} \in \pi_1(G)$ is trivial. As in \S \ref{section: borovoi fundamental group}, we denote $\chi(\overline{\lambda}) \in X_{\bullet}(T)$ a chosen representative of $\overline{\lambda}$ under the surjection $X_{\bullet}(T) \twoheadrightarrow \pi_1(G)$. Consequently, $\underline{\lambda} := \lambda \cdot \chi(\overline{\lambda})^{-1}$ lies in $X_{\bullet}(T^{\sc}) \leq X_{\bullet}(T)$.
For any $(\lambda, w) \in W_{\ext}$ with $\lambda \in X_{\bullet}(T)$ and $w \in W$, we thus have the map
\begin{align*}
    (\lambda, w): T^{\sc} \times \Gm^{\rot} &\to T^{\sc} \times \pi_1(G) \\
    (x, \zeta) &\mapsto (\underline{\lambda}(\zeta) \cdot w(x), \overline{\lambda}).
\end{align*}
In words, by taking $\underline{\lambda}$, we consider the cocharacter $\lambda$ as a map into the copy of $T^{\sc}$ labeled by $\overline{\lambda} \in \pi_1(G)$. We then consider the closed subscheme given by its graph
\begin{equation*}
    \Gamma_{\lambda, w} := \{ (x, \zeta, y, \sigma) \in T^{\sc} \times \Gm^{\rot} \times T^{\sc} \times \pi_1(G) \mid y = \underline{\lambda}(q) \cdot w(x), \ \sigma = \overline{\lambda} \} \hookrightarrow T^{\sc} \times \Gm^{\rot} \times T^{\sc} \times \pi_1(G).
\end{equation*}
Taking the union of these closed subschemes over all $\lambda \in X_{\bullet}$, $w \in W$, we obtain the closed ind-subscheme
\begin{equation*}
    \Gamma \! := \! \{ (x, \zeta, y, \sigma) \in T^{\sc} \times \Gm^{\rot} \times T^{\sc} \times \pi_1(G) \mid \exists \lambda \in X_{\bullet}, \ y = \underline{\lambda}(\zeta)\cdot w(x), \ \sigma = \overline{\lambda} \} \! \hookrightarrow \! T^{\sc} \times \Gm^{\rot} \times T^{\sc} \times \pi_1(G).
\end{equation*}
In words, the ind-subscheme $\Gamma$ now lives inside a disjoint union of copies of $T^{\sc} \times \Gm^{\rot} \times T^{\sc}$ labeled by the discrete set $\pi_1(G)$. For each $\sigma \in \pi_1(G)$, we denote $\Gamma^{\sigma} \hookrightarrow T^{\sc} \times \Gm^{\rot} \times T^{\sc} \times \{\sigma\}$ the corresponding component. While all these copies are abstractly isomorphic via the choice of lifts $\chi(\sigma)$, these identifications are not canonical.

The irreducible components $\Gamma_{\lambda, w}$ of $\Gamma$ are copies of $T^{\sc} \times \Gm^{\rot}$ labeled by $W_{\ext}$, intersecting in various ways. In particular, we have the distinguished component $\Gamma_{1, 1}$, parametrizing the quadruples $\{ (x, \zeta, x, 1) \in T^{\sc} \times \Gm^{\rot} \times T^{\sc} \times \pi_1(G)\}$.

There is again a natural action 
\begin{equation*}
W_{\ext} = X_{\bullet} \rtimes W \acts \Gamma, \qquad  (\lambda, w) \cdot (x, \zeta, y, \sigma) = (x, \zeta, \underline{\lambda}(\zeta) \cdot w(y), \overline{\lambda} \cdot \sigma ).   
\end{equation*}
In particular, it acts freely and transitively on the set of irreducible components of $\Gamma$.

\subsubsection{Variants and stabilizers.}\label{subsubsection: variants and stabilizers for reductive G}
We repeat the discussion \S \ref{section: actions}, \S \ref{section: variants}, \S \ref{section: fibers}: after taking the quotient, we get the variants $\Gamma^{T^{\sc} \sslash W_1, T^{\sc} \sslash W_2} \hookrightarrow T^{\sc} \sslash W_1 \times \Gm \times T^{\sc} \sslash W_2 \times \pi_1(G)$. We in particular denote $\overline{\Gamma} := \Gamma^{T^{\sc}, T^{\sc} \sslash W} \hookrightarrow T^{\sc} \times \Gm^{\rot} \times T^{\sc} \sslash W \times \pi_1(G)$. We can further take fibers, consider stabilizers, etc.

Replacing all occurrences of $W_{\aff}$ by $W_{\ext}$, the results of \S \ref{section: specializations of the lattice} go through with the following extra argument. We again consider the stabilizer of the base point $\Stab_{W_{\ext}}((x, 1, x, 1))$. The extended affine Weyl group $W_{\ext}$ acts on the components of $T^{\sc} \times \Gm^{\rot} \times T^{\sc} \times \pi_1(G)$ labeled by $\pi_1(G)$ through the surjection $W_{\ext} \twoheadrightarrow W_{\ext} / W_{\aff} = \pi_1(G)$. After fixing $\sigma = 1$, we are thus left with the previous action of $W_{\aff} = \Z\Phi^{\bullet}_G \rtimes W$ (which is the extended affine Weyl group of $G^{\sc}$).
In particular, Lemma \ref{lemma: fiber of cocharacter graphs as lattice quotient} holds; we again write
\begin{equation*}
 \Theta = |\overline{\Gamma}_{(x, \zeta)}| = X_{\bullet} /  W_J
\end{equation*}
and use the further notation from \S \ref{section: some useful notation}, with respect to the points of the (appropriate copy of) $T^{\sc}$.

\subsubsection{Action of the fundamental group.}\label{section: action of the fundamental group}
The fundamental group $\pi_1(G)$ acts on $T^{\sc}$ through the central deck transformations
\begin{equation*}
    \act: \pi_1(G) \to \pi_1(G^{\ad}) = Z(G^{\sc}) \leq T^{\sc} \acts T^{\sc}.
\end{equation*}
Consequently, $\pi_1(G) \acts T^{\sc} \times \Gm^{\rot} \times T^{\sc} \times \pi_1(G)$, using the diagonal deck transformation $\act \times \act$ on $T^{\sc} \times T^{\sc}$ and the trivial action on $\Gm^{\rot} \times \pi_1(G)$. 

This action preserves each of the closed subschemes $\Gamma_{\lambda, w}$ for $(\lambda, w) \in W_{\ext}$. Indeed, $\Gamma_{\lambda, w}$ parametrizes $(x, \zeta, y, \sigma)$ satisfying $y = \lambda(\zeta) \cdot w(x)$ and $\sigma = \overline{\lambda}$. An element $\tau \in \pi_1(G)$ acts by $\tau \cdot (x, \zeta, y, \sigma) = (\tau x, \zeta, \tau y, \sigma)$. The latter condition is preserved trivially; for the former one note that
\begin{equation*}
    y = \lambda(\zeta) \cdot w(x) 
    \iff \tau \cdot y = \tau \cdot \lambda(\zeta) \cdot w(x) 
    \iff  \tau \cdot y = \lambda(\zeta) \cdot w(\tau \cdot x)
\end{equation*}
since the action of $W$ is trivial on $Z(G^{\sc})$ and compatible with the product on $T^{\sc}$. 

Altogether, $\pi_1(G)$ acts on $\Gamma \hookrightarrow T^{\sc} \times T^{\sc} \times \Gm^{\rot} \times \pi_1(G)$. This action is clearly free. We obtain the quotient
\begin{equation}\label{equation: descended Gamma eq 1}
    \Gamma / \pi_1(G) \hookrightarrow \tfrac{T^{\sc} \times T^{\sc}}{\pi_1(G)} \times \Gm^{\rot} \times \pi_1(G). 
\end{equation}

\begin{remark}\label{remark: identification for Gamma in semisimple case}
We may further use the following identification. Consider the group isomorphism $T^{\sc} \times T^{\sc} \xrightarrow{\cong} T^{\sc} \times T^{\sc}$ given by $(x, y) \mapsto (x, x^{-1}y)$. Under this isomorphism, the diagonal action $\act \times \act$ of $\pi_1(G)$ identifies with its action on the first factor $\act \times \id$. Taking the quotient identifies $\tfrac{T^{\sc} \times T^{\sc}}{\pi_1(G)} = \tfrac{T^{\sc}}{\pi_1(G)} \times T^{\sc} = T^{\der} \times T^{\sc}$. We rewrite \eqref{equation: descended Gamma eq 1} as
\begin{equation}\label{equation: descended Gamma eq 2}
    \Gamma^{T^{\der}, T^{\sc}} \hookrightarrow T^{\der} \times \Gm^{\rot} \times T^{\sc} \times \pi_1(G). 
\end{equation}
where the sub-ind scheme $\Gamma^{T^{\der}, T^{\sc}}$ is given by transporting $\Gamma$ through the above identifications and taking the quotient by $\pi_1(G)$.
\end{remark}

\begin{remark}\label{remark: identification for Gamma in general reductive case}
For semisimple $G$, we have $T=T^{\der}$ and Remark \ref{remark: identification for Gamma in semisimple case} already describes the desired cocharacter graphs.

For reductive $G$, we consider \eqref{equation: descended Gamma eq 2} for $G^{\ad}$, giving $\Gamma^{T^{\ad}, T^{\sc}} \hookrightarrow T^{\ad} \times \Gm^{\rot} \times T^{\sc} \times \pi_1(G)$. By base change along the map $T \to T^{\ad}$ to the first factor we obtain
\begin{equation}\label{equation: descended Gamma eq 3}
    \Gamma^{T, T^{\sc}} \hookrightarrow T \times \Gm^{\rot} \times T^{\sc} \times \pi_1(G),
\end{equation}
the appropriate generalization of the cocharacter graphs for the reductive group $G$.
In particular, we obtain the variant $\overline{\Gamma}^{T, T^{\sc}\sslash W} \hookrightarrow T \times \Gm^{\rot} \times T^{\sc}\slash W \times \pi_1(G)$. 
\end{remark}

%% file: 4_fixed_points.tex
\section{Fixed points on affine Grassmannians}\label{section: fixed points on affine grassmannians}

Let $G$ be a reductive group and $\Gr_G = \lo G / \lop G$ its affine Grassmannian over $k$. We assume that $k$ is algebraically closed field. For ease of discussion, we further assume $k$ is of characteristic zero; this is not essential and can be largely avoided.

Consider the fixed-point family $\Fix_{\lo G \rtimes \Gm^{\rot}}(\Gr_G)$. Since $\Gr$ is an ind-scheme, so is $\Fix_{\lo G \rtimes \Gm^{\rot}}(\Gr)$.
Its natural derived structure is trivial by the criterion from \S \ref{subsubsection: reduced and derived variants} applied to the ind-presentation by affine Schubert varieties.

Our main interest lies in the classical reduced restriction of this family to the extended torus
\begin{equation*}
    \Fix^{\red}_{T \times \Gm^{\rot}}(\Gr_G).
\end{equation*}
In the rest of this section, we explain how to effectively describe the fibers of this family as the point $(x, \zeta) \in T \times \Gm^{\rot}$ varies, for both $\Gr_G$ as well as for affine Schubert varieties $\Gr_{\leq \mu}$ inside it.
We separate the section into a uniform discussion in \S \ref{section: general structure of the fixed points}, followed by specific descriptions in the case of $\zeta$ generic \S \ref{section: zeta generic} and $\zeta$ root of unity \ref{section: zeta root of unity}.

\subsection{General structure of the fixed points}\label{section: general structure of the fixed points}

\subsubsection{Centralizer action on fixed-points.}\label{section: the subgroup H^+}
With the subgroups $J \leq \lo G$ and $J^+ \leq \lop G$ from \S \ref{section: centralizers}, the following statements are obvious.
\begin{lemma}\label{lemma: H acts on fixed points general case}
The subgroup $J$ acts on $\Fix_{(x, \zeta)}(\Gr_G)$.
\end{lemma}
\begin{proof}
Since $J \leq \Cent_{\lo G \rtimes \Gm^{\rot}}(x, \zeta)$ is pointwise fixed under the adjoint action of $(x, \zeta)$ on $\lo G \rtimes \Gm^{\rot}$, it acts on the corresponding fiber of the fixed-point scheme by \S \ref{section: residual action on Fix}.
\end{proof}
\begin{lemma}\label{lemma: H^+ acts on fixed points general case}
The subgroup $J^{+}$ acts on $\Fix_{(x, \zeta)}(\Gr_{\leq \mu})$. 
\end{lemma}
\begin{proof} 
Since $\lop G\rtimes \Gm^{\rot}$ acts on $\Gr_{\leq \mu}$ and $J^+ = J \cap \lop G$, this follows by restricting the action from Lemma \ref{lemma: H acts on fixed points general case}.
\end{proof}

\begin{remark}
In the above statements, note that the whole centralizer $\Cent_{\lo G \rtimes \Gm^{\rot}}(x, \zeta)$ and its positive part $\Cent_{\lop G \rtimes \Gm^{\rot}}(x, \zeta)$ act on the respective fixed points. This amounts to an extra action by $\Gm^{\rot}$ as well as by the group of connected components. This is a natural piece of structure, but we do not need it for our purposes.
\end{remark}

\subsubsection{Fixed points on affine Grassmannian.}
We first describe the fixed points on $\Gr_G$ in terms of the connected centralizer $J$. 

To this end, recall the reflection subgroup $W_J \leq W_{\aff}$ from Notation \ref{notation: WJ}. We then have the labeling set $\Theta = X_{\bullet}/W_J$ from Notation \ref{notation: Theta}. 
For any $\vartheta \in \Theta$, we may choose a lift $\lambda \in X_{\bullet}$ and put $J^+_{\vartheta} := J^+_{\lambda}$ in the notation \eqref{equation: positive stabilizer}. This is independent of the choice of $\lambda$ up to isomorphism by Remark \ref{remark: independece of J vartheta on the lift} below.

\begin{lemma}\label{lemma: fixed points on affine Grassmannian}
 Let $(x, \zeta) \in T \times \Gm^{\rot}$. Then
 \begin{equation}\label{equation: fixed points on Gr}
     \Fix^{\red}_{(x, \zeta)}(\Gr_G) = \coprod_{\vartheta \in \Theta} J/J^{+}_{\vartheta}. 
 \end{equation}
\end{lemma}
\begin{proof}
Any point of $x \in \Gr_G$ can be written as a product
\begin{equation}\label{equation: root decomposition in affine grassmannian}
t^{\lambda} \cdot u_1 \cdots u_d
\end{equation}
with $\lambda \in X_{\bullet}$ and $u_i \in U_{\beta_i + m_i\hbar}$ nontrivial points in finitely many pairwise distinct affine root groups; see \cite[proof of Lemma 4.4]{RW22} for this well-known fact (\cite{RW22} discusses a more refined statement for a given Iwahori orbit -- our variant follows since the Iwahori orbits of $t^{\lambda}, \lambda \in X_{\bullet}$ are affine spaces paving the affine Grassmannian). 

The action of $(x, \zeta)$ fixes each $t^{\lambda}$. We now restrict attention to the fixed points in \eqref{equation: root decomposition in affine grassmannian} for one chosen $t^{\lambda}$. The expression \eqref{equation: root decomposition in affine grassmannian} is unique. On each root subgroup $U_{\beta + m\hbar}$, the element $(x, \zeta)$ acts via $\beta(x) \cdot \zeta^m$. Comparing to Lemma \ref{lemma: root spaces of J}, we deduce that the fixed points of this form are given by the orbit of $J$ through $t^{\lambda}$. Since $J^+_{\lambda}$ is the stabilizer of $t^{\lambda}$, this orbit is given by $J/J^+_{\lambda}$.

Ranging over all $\lambda \in X_{\bullet}$, we see that the fixed points decompose into a disjoint union of the orbits $J/J^+_{\lambda}$ of $t^{\lambda}$, $\lambda \in X_{\bullet}$. Moreover, two points $t^{\lambda_1}, t^{\lambda_2}$ lie in the same orbit of $J$ if and only if $\lambda_1, \lambda_2$ can be connected by affine reflections from $\Phi^{\bullet}_J$, in other words if and only if $\lambda_1, \lambda_2$ give the same point $\vartheta \in \Theta = X_{\bullet} / W_{J}$.
The lemma follows.
\end{proof}

\begin{remark}\label{remark: independece of J vartheta on the lift}
The final paragraph of the above proof shows that $t^{\lambda_1}$, $t^{\lambda_2}$ lie in the same orbit of $J$ if and only if $\lambda_1, \lambda_2$ give the same point $\vartheta \in \Theta = X_{\bullet} / W_{J}$. In particular, the stabilizer $J^{+}_{\lambda}$ depends, up to isomorphism given by conjugation inside $J$, only on the coset $\vartheta \in \Theta$ of $\lambda$. It is thus fair to denote it by $J^+_{\vartheta}$.
\end{remark}

\begin{notation}
We denote the connected components of \eqref{equation: fixed points on Gr} by
\begin{equation*}
    \Fix^{\red, \vartheta}_{(x, \zeta)}(\Gr_G) = J/J^{+}_{\vartheta}, \qquad \vartheta \in \Theta.
\end{equation*}
\end{notation}

\subsubsection{Fixed points on $\Gr_G$ and $\Gr_L$.}\label{section: fixed points on GrG and GrL}
We also have the following alternative approach for describing the fixed points $\Fix^{\red}_{(x, \zeta)}(\Gr_G)$.
Consider the defining inclusion of the resonant subgroup $L \leq G$. It induces a map on loop groups and on affine Grassmannians $\Gr_L \subseteq \Gr_G$. Using this embedding, we obtain the following identification of fixed points.

\begin{lemma}\label{lemma: replacing x by the resonant part}
We have identifications
\begin{equation*}
    \Fix^{\red}_{(x, \zeta)}(\Gr_G) = \Fix^{\red}_{(\omega(\zeta), \zeta)}(\Gr_L).
\end{equation*}
\end{lemma}
\begin{proof}
Consider the natural inclusion $\Gr_L \subseteq \Gr_G$ induced from $L \leq G$. Both $L$ and $G$ have the common maximal torus $T$. Both affine Grassmannians thus have the same coordinate fixed points $t^{\lambda}$, $\lambda \in X_{\bullet}$. Fixing one such $t^{\lambda}$ and using the decomposition \eqref{equation: root decomposition in affine grassmannian}, we are left to compare which root spaces are fixed under the adjoint action of $(x, \zeta)$. A root space $U_{\beta + m\hbar}$ of $\lo G$ for $\beta + m\hbar \in \Phi^{\bullet}_{G, \aff}$ is fixed if and only if $\beta(x)\cdot \zeta^m=1$. For this to be true, it is necessary that $\beta \in \Phi^{\bullet}_L$. For such $\beta \in \Phi^{\bullet}_L$, we have $\beta(x)=\beta(\omega(\zeta))$ by construction of $\omega$. Hence the above condition is satisfied if and only if $\beta \in \Phi^{\bullet}_L$ and $\beta(\omega(\zeta))\cdot \zeta^m=1$. This precisely gives the root spaces of $\lo L$ fixed under the adjoint action of $(\omega(\zeta), \zeta)$ as desired.
\end{proof}

Moreover, we can get rid of the element $\omega(\zeta)$ at the cost of a shift. 
To this end, take the adjoint form $L^{\ad}$ of $L$. Since any connected component of $\Gr_L$ is equivariantly isomorphic to the corresponding connected component of $\Gr_{L^{\ad}}$ via the map $\pi_1(L) \to \pi_1(L^{\ad})$ of their labeling sets, we do not lose anything by working with $L^{\ad}$. Here $\omega$ is a genuine cocharacter, so multiplication by $t^{\omega} \in \lo L^{\ad}$ induces a shift map on the corresponding affine Grassmannian $\Gr_{L^{\ad}}$.

\begin{lemma}\label{lemma: fixed points and the omega shift general case}
Translation by $t^{\omega}$ induces an eqivariant isomorphsim 
\begin{equation}
t^{\omega} \cdot \Fix_{(\omega(\zeta), \zeta)}(\Gr_{L^{\ad}}) = \Fix_{(1, \zeta)}(t^{\omega} \cdot \Gr_{L^{\ad}}).   
\end{equation}
\end{lemma}
\begin{proof}
Let $p \in \Gr_{L^{\ad}}$. We need to see that $(\omega(\zeta), \zeta)\cdot p = p$ holds if and only if $(1, \zeta) \cdot \omega(t) \cdot p = \omega(t) \cdot p$. By the general structure of fixed-point schemes \S \ref{section: residual action on Fix}, it is enough to check that inside $\lo L^{\ad} \rtimes \Gm^{\rot}$ we have
\begin{equation}
 (\omega(\zeta), \zeta) = \omega^{-1}(t) \cdot (1, \zeta) \cdot \omega(t).
\end{equation}
This is the case by direct computation within $\lo L^{\ad} \rtimes \Gm^{\rot}$:
\begin{align*}
(\omega(\zeta), \zeta)
&= (\omega(t^{-1}) \cdot 1 \cdot \omega(\zeta t), 1 \cdot \zeta \cdot 1)\\
&= \cdot (\omega(t^{-1}), 1) \cdot (1, \zeta) \cdot (\omega(t), 1)\\
&= \omega^{-1}(t) \cdot (1, \zeta) \cdot \omega(t).
\end{align*}
where the second equality comes from the semidirect product $\lo L^{\ad} \rtimes \Gm^{\rot}$ since conjugation by the loop rotation factor scales $t$ with weight one.
\end{proof}

Putting Lemmata \ref{lemma: replacing x by the resonant part}, \ref{lemma: fixed points and the omega shift general case} together, we have identified $\Fix^{\red}_{(x, \zeta)}(\Gr_G)$ as a disjoint union of fixed points of the loop rotation by $\zeta$ on a disjoint union of connected components $\Gr_{L^{\ad}}^{\sigma}$ of the affine Grassmannian of $L^{\ad}$

\begin{equation}\label{equation: fixed points on Gr_G and loop rotation on Gr_L}
    \Fix^{\red}_{(x, \zeta)}(\Gr_G) 
    \cong \coprod_{\pi_1(L)} \Fix^{\red}_{(1, \zeta)}(t^{\omega} \cdot \Gr^{{\sigma}}_{L^{\ad}})
    = \coprod_{\pi_1(L)} \Fix^{\red}_{(1, \zeta)}(\Gr^{{\underline{\omega} + \sigma}}_{L^{\ad}})
\end{equation}
The disjoint union runs over the set $\pi_1(L)$ of connected components of $\Gr_L$, whose elements determine $\sigma \in \pi_1(L^{\ad})$ thought the natural map $\pi_1(L) \to \pi_1(L^{\ad})$.

The fixed points of the loop rotation by $\zeta$ on affine Grassmannians -- such as $\Gr_{L^{\ad}}$ -- are well known. The identification \eqref{equation: fixed points on Gr_G and loop rotation on Gr_L} thus reduces the computation of $\Fix^{\red}_{(x, \zeta)}(\Gr_G)$ to this known special case.

\subsubsection{Fixed-points on affine Schubert varieties.}
We can further refine the computation from Lemma \ref{lemma: fixed points on affine Grassmannian} to describe fixed points on each affine Schubert variety $\Gr_{\leq \mu}$ in $\Gr_G$.
From \S \ref{section: the subgroup H^+} we get the following diagram
\begin{equation}\label{equation: fixed points on affine schubert varieties 1}
J^+ \acts
\begin{tikzcd}
    \Gr_{\leq \mu}  \arrow[r, hookrightarrow] & \Gr_G \arrow[r, equal] & \lo G / \lop G \\
    \Fix_{(x, \zeta)}(\Gr_{\leq \mu}) \arrow[u, hookrightarrow] \arrow[r, hookrightarrow] & \Fix_{(x, \zeta)}(\Gr_G) \arrow[u, hookrightarrow]  \arrow[r, equal] & \coprod_{\vartheta \in \Theta} J/J^{+}_{\vartheta}. \arrow[u, hookrightarrow]
\end{tikzcd}
\end{equation}
where all maps are equivariant closed immersions.

\begin{lemma}\label{lemma: upshot of the general case}
The scheme $\Fix^{\red}_{(x, \zeta)}(\Gr_{\leq \mu})$ identifies with a disjoint union of closed $J^+$-equivariant subschemes
\begin{equation*}\label{equation: B-equivariant closed immersion into classical flag varieties general case}
J^+ \acts \Fix^{\red}_{(x, \zeta)}(\Gr_{\leq \mu}) \hookrightarrow \coprod_{\vartheta \in \Theta} J/J^+_{\vartheta}.
\end{equation*}
\end{lemma}
\begin{proof}
Take the reduced bottom row of diagram \eqref{equation: fixed points on affine schubert varieties 1}. 
\end{proof}

Furthermore, it is useful to understand the embedding from Lemma \ref{lemma: upshot of the general case} in terms of the coordinate fixed points $t^{\lambda}$. Given any $\mu \in X^{+}_{\bullet}$, recall \eqref{notation: weights in mu} that 
$ \wt(\mu) = W \cdot \{ \lambda \in X^{+}_{\bullet} \mid \lambda \leq \mu \} \subseteq X_{\bullet}$
stands for the union of the Weyl orbits of the dominant weights $\lambda$ below $\mu$.
Under the geometric Satake equivalence, $\wt(\mu)$ labels the fixed points of $T \times \Gm^{\rot}$ on $\Gr_{\leq \mu}$.
Consequently, the fixed point scheme $\Fix_{(x, \zeta)}(\Gr_{\leq \mu})$ contains the discrete set $t^{\lambda}$, $\lambda \in \wt(\mu)$.

\begin{lemma}\label{lemma: connected components of fixed points n affine schubert variety}
The $T\times \Gm^{\rot}$-fixed-points on $\Fix_{(x, \zeta)}(\Gr_{\leq \mu})$ are given by $\{ t^{\lambda} \mid \lambda \in \wt(\mu) \}$. Two such fixed-points $t^{\lambda_1}$, $t^{\lambda_2}$ lie in the same connected component of $\Fix_{(x, \zeta)}(\Gr_{\leq \mu})$ if and only if $\lambda_1$, $\lambda_2$ lie in the same orbit of $W_J \acts X_{\bullet}$.   
\end{lemma}
\begin{proof}
We may work with reduced fixed point schemes.
The description of coordinate fixed points is clear from the discussion above. We are left to describe the connected components.

For one direction, assume that $\lambda_1, \lambda_2$ lie in the same orbit of $W_J$. Then they can be connected by a sequence of affine reflections $s_{\beta + m\hbar}$ for $\beta+m\hbar \in \Phi^{\bullet}_J \subseteq \Phi^{\bullet}_{G, \aff}$. To see this, first assume that $\lambda_2 = s_{\beta + m\hbar} \cdot \lambda_1$ for a single reflection in $W_J$. We claim that the affine root $\beta + m\hbar$ induces a $\P^1$ which connects the points $t^{\lambda_1}$, $t^{\lambda_2}$ inside $\Fix^{\red}_{(x, \zeta)}(\Gr_{\leq \mu})$.
Indeed, whenever two coweights $\lambda_1$, $\lambda_2$ are connected by an affine reflection $s_{\beta+m\hbar} \in W_{G, \aff}$, then the associated root space $U_{\beta + m\hbar}$ extends to a $\P^1$ connecting the corresponding points $t^{\lambda_1}$, $t^{\lambda_2}$ inside $\Gr_G$. Moreover, if $\lambda_1, \lambda_2 \in \wt(\mu)$, this $\P^1$ lies inside $\Gr_{\leq \mu}$. Since we also have ${\beta + m\hbar} \in \Phi^{\bullet}_J$, the root space $U_{\beta + m\hbar}$ lies inside $J$, hence the induced $\P^1$ actually lies inside $\Fix^{\red}_{(x, \zeta)}(\Gr_G)$. Altogether, it lies inside 
\begin{equation*}
    \Fix^{\red}_{(x, \zeta)}(\Gr_{\leq \mu}) = \Fix^{\red}_{(x, \zeta)}(\Gr_G) \cap \Gr_{\leq \mu},
\end{equation*}
proving the claim. 
In general, choosing a sequence of affine reflection of $W_J$ connecting $\lambda_1$, $\lambda_2$, we have thus constructed a chain of $\P^1$'s connecting $t^{\lambda_1}$ and $t^{\lambda_2}$ inside $\Fix^{\red}_{(x, \zeta)}(\Gr_{\leq \mu})$. Hence they lie in the same connected component.

For the converse, note that when $\lambda_1$ and $\lambda_2$ lie in different orbits of $W_J$, the fixed-points  $t^{\lambda_1}$, $t^{\lambda_2}$ end up in different terms of the disjoint union $\coprod_{\vartheta \in \Theta} J/J^+_{\vartheta}$, hence lie in different connected components of its subscheme $\Fix^{\red}_{(x, \zeta)}(\Gr_{\leq \mu})$.
\end{proof}

To conclude our computation of the fixed points, it remains to explicitly describe the players in Lemmata \ref{lemma: fixed points on affine Grassmannian}, \ref{lemma: upshot of the general case}. We do so in the next two subsections, distinguishing whether $\zeta$ is generic or a root of unity.

\subsection{\texorpdfstring{$\zeta$}{Zeta} generic}\label{section: zeta generic}
\subsubsection{Setup.}
Let $(x, \zeta) \in T \times \Gm^{\rot}$ be a closed point with $\zeta$ generic. The purpose of this subsection is to describe the fibers of the fixed-point scheme $\Fix^{\red}_{(x, \zeta)}(\Gr_G)$ as infinite disjoint unions of partial flag varieties for the resonant subgroup $L = L_{[x, \zeta]}$, as well as $\Fix^{\red}_{(x, \zeta)}(\Gr_{\leq \mu})$ as a union of finite Schubert varieties inside them.

\subsubsection{Centralizers over generic $\zeta$.}
When $\zeta$ is generic, the connected centralizer $J = J_{(x, \zeta)}$ is given by a slanted copy of the resonant subgroup $L = L_{[x, \zeta]}$.
\begin{lemma}\label{lemma: centralizer as resonant subgroup -- generic zeta}
Let $(x, \zeta) \in T \times \Gm^{\rot}$ with $\zeta$ generic. We have a preferred identification
\begin{equation*}
    \slant_{\omega}: J \xrightarrow{=} L,
\end{equation*}
describing the connected centralizer $J$ as a slanted copy of the resonant subgroup $L$. On root systems, this is induced by
$\slant_{\omega}: \Phi^{\bullet}_J \to \Phi^{\bullet}_L $, $\beta \mapsto \beta + m_{\beta}\hbar$.
\end{lemma}
\begin{proof}
We have described the root system of $J$ in Lemma \ref{lemma: root spaces of J}. Since $\ell = \ord(\zeta)=0$, it is given by $ \Phi^{\bullet}_J = \{ \beta + m_{\beta} \cdot \hbar \mid \beta \in \Phi^{\bullet}_L \}$ for the (uniquely determined) additive function $\mm: \beta \mapsto m_{\beta}$ from Notation \ref{notation: additive function m}.
Slanting by $\omega$ as in Lemma \ref{discussion: slanting J in general}, we get
$\Phi^{\bullet}_{\slant_{\omega}(J)} = \{ \beta \mid \beta \in \Phi^{\bullet}_L \}$, which is the root system of the connected reductive group $L$. The discussion in the same lemma implies the formula on root systems.
\end{proof}

\subsubsection{Labeling set for components.}
Under slanting on root systems $\slant_{\omega}: \Phi^{\bullet}_J \xrightarrow{=} \Phi^{\bullet}_L$ from Remark \ref{remark: slanting on weyl groups general case}, the group $W_J$ identifies as follows.
\begin{lemma}\label{lemma: resonant subgroup -- zeta generic}
Slanting identifies
\begin{equation*}
    \slant_{\omega}: W_J \xrightarrow{=} W_L.
\end{equation*}
\end{lemma}
\begin{proof}
Clear from Lemma \ref{lemma: centralizer as resonant subgroup -- generic zeta} since $W_J = W_{\Phi^{\bullet}_J}$ and $W_L = W_{\Phi^{\bullet}_L}$.   
\end{proof}

We can now describe the labeling set $\Theta$ from Notation \ref{notation: Theta} more explicitly.
To this end, recall the $\omega$-shifted action from \S \ref{section: shifted actions}.
\begin{lemma}\label{lemma: labeling set for components -- zeta generic}
For $(x, \zeta)$ with $\zeta$ generic, we have
\begin{equation*}
 \Theta = |\overline{\Gamma}_{(x, \zeta)}|  = X_{\bullet} / (W_{L}, \bullet_{\omega}).
\end{equation*}
\end{lemma}
\begin{proof}
 By Notation \ref{notation: Theta} and Lemma \ref{lemma: fiber of cocharacter graphs as lattice quotient}, we have $\Theta = |\overline{\Gamma}_{(x, \zeta)}|  = X_{\bullet}/W_J$ with respect to the natural action of $W_J$ on $X_{\bullet}$. Since the natural action of $W_J$ identifies with the $\omega$-shifted action of $\slant_{\omega}(W_J)$ by Remark \ref{remark: slanting on weyl groups general case}, we identify the quotient as
 \begin{equation*}
     X_{\bullet}/W_J = X_{\bullet}/(W_{\slant_{\omega}(J)}, \bullet_{\omega}).
 \end{equation*}
Since $W_{\slant_{\omega}(J)} = W_L$ is the resonant Weyl group by Lemma \ref{lemma: resonant subgroup -- zeta generic}, the result follows.
\end{proof}

\begin{remark}
Similar description works for other variants of cocharacter graphs -- for example,  
\begin{equation*}
   |{\Gamma}_{(x, \zeta)}|  = W_{\ext} / (W_{L}, \bullet_{\omega}). 
\end{equation*}
\end{remark}

\subsubsection{Positive stabilizers as parabolic subgroups.}
In terms of the slanted identification, the subgroup $J^+$ is given as $P_x$, the parabolic subgroup from \S \ref{section: parabolic and levi subgrops of L}. More generally, stabilizers of each $t^{\lambda} \in \Gr_G$ are described as follows.
\begin{lemma}\label{lemma: positive stabilizers as parabolics -- generic zeta}
For $(x, \zeta)$ with $\zeta$ generic and arbitrary $\lambda \in X^{\bullet}$, slanting identifies
\begin{equation*}
    \slant_{\omega}: J^{+}_{\lambda} \xrightarrow{=}  P_{x \cdot \lambda(\zeta)},
\end{equation*}   
where $P_{x \cdot \lambda(\zeta)} \leq L$ is a parabolic subgroup whose Levi factor $L_{x \cdot \lambda(\zeta)} \leq L$ is the connected stabilizer of $x \cdot \lambda(\zeta)$ inside $G$. 
\end{lemma}
\begin{proof}
Note that $U_{\beta + m\hbar} \subseteq t^{\lambda} \cdot \lop G \cdot t^{-\lambda}$ if and only if $ t^{-\lambda} \cdot U_{\beta + m\hbar} \cdot t^{\lambda} \subseteq \lop G$. Since $ t^{-\lambda} \cdot U_{\beta + m\hbar} \cdot t^{\lambda} = U_{\beta + (m - \langle \beta, \lambda \rangle)\hbar}$ by Lemma \ref{lemma: conjugation of affine root spaces}, the root spaces of $t^{\lambda} \cdot \lop G \cdot t^{-\lambda}$ are given by
$\{ \beta + m\hbar \mid m - \langle \beta, \lambda \rangle \geq 0 \}$.
To describe the roots spaces of $J^{+}_{\lambda}$, we intersect with the root system of $J$. This yields
\begin{equation*}
    \Phi^{\bullet}_{J^{+}_{\lambda}} = \{ \beta - m_{\beta}\hbar \mid \beta \in \Phi^{\bullet}_L, \  m_{\beta} + \langle \beta, \lambda \rangle \leq 0 \}
\end{equation*}
After slanting as in \S \ref{section: slanting}, this is the subgroup of $\slant_{\omega}(J^{+}_{\lambda}) \leq L$ spanned by the root spaces from
\begin{equation*}
    \Phi^{\bullet}_{\slant_{\omega}(J^{+}_{\lambda})} = \{ \beta \in \Phi^{\bullet}_L \mid  m_{\beta} + \langle \beta, \lambda \rangle \leq 0 \} = \{ \beta \mid \beta(x \cdot \lambda(\zeta) = \zeta^m \ \text{for some} \ m \leq 0 \}.
\end{equation*}
Here, the second equality follows since $\beta(x \cdot \lambda(\zeta)) = \beta(x) \cdot \beta(\lambda(\zeta)) = \zeta^{m_{\beta}} \cdot \zeta^{\langle \beta, \lambda \rangle} = \zeta^{m_{\beta} + \langle \beta, \lambda \rangle}$; this uniquely determines the exponent by genericity of $\zeta$.

These are the root spaces of $P_{x \cdot \lambda(\zeta)}$ from \eqref{equation: roots of Px}, a parabolic subgroup with Levi factor $L_{x \cdot \lambda(\zeta)}$ by Lemma \ref{lemma: parabolic and levi subgrops of L} and Remark \ref{remark: parabolic and levi subgrops of L}.
\end{proof}

\begin{remark}\label{remark: diferent labelings of centralizers -- generic zeta}
We may label the above subgroups of $L$ differently. Let $\vartheta \in \Theta = X_{\bullet}/(W_{\slant_{\omega}(J)}, \bullet_{\omega})$ be the orbit of $\lambda$. Under the interpretation, $\Theta = |\overline{\Gamma}_{(x, \zeta)}| \subseteq |T \sslash W|$, the element $\vartheta$ is given by the equivalence class of the point $x \cdot \lambda(\zeta)$. In Notation \ref{notation: levi and paraolic subgroups labeled by elements of Theta}, we get
\begin{equation*}
  P_{\vartheta} = P_{x \cdot \lambda(\zeta)} = P_{\omega+\lambda^{\ad}} 
  \qquad \text{and} \qquad
  L_{\vartheta} = L_{x \cdot \lambda(\zeta)} = L_{\omega+\lambda^{\ad}},
\end{equation*}
where the latter equalities realize our subgroups as the parabolic and Levi subgroup of $L$ attached to the cocharacter $\omega + \lambda^{\ad} \in X^{\ad}_{\bullet, L}$ by Remark \ref{remark: parabolic and levi subgrops of L}.
\end{remark}

\begin{lemma}\label{lemma: positive centralizer as parabolic -- gneric zeta}
Inside $\lop G \rtimes \Gm^{\rot}$,  
\begin{equation*}
\slant_{\omega}: J^{+} \xrightarrow{=} P_x 
\end{equation*}
where $P_x = P_{\omega} \leq L$ is the opposite parabolic subgroup with Levi factor $L_x = L_{\omega} \leq L$.
\end{lemma}
\begin{proof}
Special case of Lemma \ref{lemma: positive stabilizers as parabolics -- generic zeta} and Remark \ref{remark: diferent labelings of centralizers -- generic zeta} for $\lambda$ trivial.    
\end{proof}

\subsubsection{Fixed points on affine Grassmannian at \texorpdfstring{$\zeta$}{zeta} generic.}
We now describe the fixed points on the whole affine Grassmannian.
\begin{theorem}\label{lemma: fixed points of (c, zeta) on Gr_G -- zeta generic}
For $(x, \zeta) \in T \times \Gm^{\rot}$ with $\zeta$ generic, we have
\begin{equation}\label{equation: fixed points of (x, zeta) on Gr_G -- zeta generic}
    \Fix^{\red}_{(x, \zeta)}(\Gr_G) = \coprod_{\vartheta \in X_{\bullet}/(W_{L}, \bullet_{\omega})} L/P_{\vartheta}
\end{equation}
where $L = L_{[x, \zeta]}$ is the resonant subgroup, $P_{\vartheta}$ is the parabolic subgroup labeled by $\vartheta \equiv x \cdot \lambda(\zeta)$, and $W_L$ acts on $X_{\bullet}$ by the $\omega$-shifted action.
\end{theorem}
\begin{proof}
By Lemma \ref{lemma: fixed points on affine Grassmannian}, the fixed points are given by a disjoint union of the homogeneous spaces $J/J^+_{\vartheta}$ ranging over $\vartheta \in \Theta$.
In the situation at hand, $\Phi^{\bullet}_J \cong \Phi^{\bullet}_L$ and $J \cong L$ by the slanted identification from Lemma \ref{lemma: centralizer as resonant subgroup -- generic zeta}. The labeling set for components from Notation \ref{notation: Theta} becomes $\Theta \cong X_{\bullet} / (W_L, \bullet_{\omega})$ by Lemma \ref{lemma: labeling set for components -- zeta generic}. For given $\vartheta \in X_{\bullet} / (W_L, \bullet_{\omega})$, pick a lift $\lambda \in X_{\bullet}$. Slanting identifies $J^+_{\lambda} \cong P_{x \cdot \lambda(\zeta)}$ by Lemma \ref{lemma: positive stabilizers as parabolics -- generic zeta}; this parabolic is independent of the lift and denoted $P_{\vartheta}$ as in Notation \ref{notation: levi and paraolic subgroups labeled by elements of Theta} and Remark \ref{remark: diferent labelings of centralizers -- generic zeta}. The result follows.
\end{proof}

\begin{remark}
The above theorem can be also proved through the viewpoint of \S \ref{section: fixed points on GrG and GrL}, which we outline now. Passing to appropriate components if necessary, we only need to compute fixed points of the loop rotation by $\zeta$ on $\Gr_{L^{\ad}}$. These fixed-points are given by
\begin{equation}
    \Fix^{\red}_{(1, \zeta)}(\Gr_{L^{\ad}}) \cong \coprod_{\lambda \in X^{\ad}_{\bullet, L}/W_L} L^{\ad}/P_{\lambda}.
\end{equation}
where $\lambda$ runs over orbit representatives for the usual action of $W_L$ on $X^{\ad}_{\bullet, L}$ and $P_{\lambda} \leq L^{\ad}$ is the parabolic subgroup associated to $\lambda$ as in \S \ref{section: levi subgroups}.

Indeed, since $\omega$ is a cocharacter of $L^{\ad}$, shifting identifies $X^{\ad}_{\bullet, L}/(W_L, \bullet_{\omega}) \cong X^{\ad}_{\bullet, L}/W_L$. The orbits of this naive action $W_L \acts X^{\ad}_{\bullet, L}$ parametrize the dominant cocharacters $X^+_{\bullet}(L^{\ad})$ of $L^{\ad}$. 
The fixed points of $(1, \zeta)$ on any affine Grassmannian are well-known, e.g. \cite[above Example 2.1.12]{Zhu15}. Each affine Schubert cell is a geometric vector bundle over the corresponding partial flag variety, with nontrivial weights in the fibers. The loop rotation fixed points are then given by the zero section. These are the same as the fixed points of the single element $\zeta$ by the genericity assumption.
\end{remark}

\subsubsection{Fixed points on affine Schubert varieties as finite Schubert varieties.}
We can now prove the desired description of fixed-points on any affine Schubert variety.
\begin{theorem}\label{theorem: generic zeta -- fiber of fixed points and classical schubert varieties}
The reduced scheme $\Fix^{\red}_{(x, \zeta)}(\Gr_{\leq \mu})$ is isomorphic to a union of classical Schubert varieties for the connected reductive group $L = L_{[x, \zeta]}$.
\end{theorem}
\begin{proof}
Unraveling Lemma \ref{lemma: upshot of the general case} via the slanted descriptions from Lemmata \ref{lemma: positive centralizer as parabolic -- gneric zeta},  \ref{lemma: fixed points of (c, zeta) on Gr_G -- zeta generic} identifies the variety 
$F = \Fix^{\red}_{(x, \zeta)}(\Gr_{\leq \mu})$
of our interest as a closed $P_{x}$-equivariant subvariety
\begin{equation*}\label{equation: B-equivariant closed immersion into classical flag varieties -- zeta generic}
P_{x} \acts F \hookrightarrow \coprod_{\vartheta \in X_{\bullet}/(W_{L}, \bullet_{\omega})} L/P_{\vartheta}   
\end{equation*}
where $P_{x} \leq L$ acts by left translation on each of the partial flag varieties $L/P_{\vartheta} = L/P_{x \cdot \lambda(\zeta)}$

We may work on each connected components $L/P_{\vartheta}$ separately. The corresponding piece of the fixed point scheme $F^{\vartheta} = F \cap L/P_{\vartheta}$ cuts out a closed $P_{x}$-equivariant subvariety -- therefore, it is a downwards closed union of finitely many Schubert cells. Altogether, we get a (not necessarily disjoint) union of classical Schubert varieties in $L/P_{\vartheta}$.
\end{proof}

\subsubsection{Combinatorial description.}
Moreover, the classical Schubert varieties appearing in Theorem \ref{theorem: generic zeta -- fiber of fixed points and classical schubert varieties} can be described combinatorially. Indeed, to get such a description, all we need to do is to describe the set of relevant torus fixed-points. 
These fixed points on $\Gr_{\leq \mu}$ are labeled by the finite set $\wt(\mu)$ from \eqref{notation: weights in mu}.

\begin{lemma}\label{lemma: combinatorial description -- zeta generic}
The $T\times \Gm^{\rot}$-fixed-points on $\Fix^{\red}_{(x, \zeta)}(\Gr_{\leq \mu})$ are given by $\{ t^{\lambda} \mid \lambda \in \wt(\mu) \}$. Two such fixed-points $t^{\lambda_1}$, $t^{\lambda_2}$ lie in the same connected component of $\Fix^{\red}_{(x, \zeta)}(\Gr_{\leq \mu})$ if and only if $\lambda_1$, $\lambda_2$ lie in the same orbit of the $\omega$-shifted action $(W_L, \bullet_{\omega}) \acts X_{\bullet}$.    
\end{lemma}
\begin{proof}
Follows by unraveling Lemma \ref{lemma: connected components of fixed points n affine schubert variety}.
Under these identifications, wherever $\lambda_1$ and $\lambda_2$ are related by a reflection from $(W_L, \bullet_{\omega})$, the two points $t^{\lambda_1}$, $t^{\lambda_2}$ can be connected by a $\P^1$ inside $F$.
\end{proof}

\begin{remark}
For an alternative proof of Lemma \ref{lemma: combinatorial description -- zeta generic}, we only need to see that for each $\vartheta \in \Theta$, the subscheme $F^{\vartheta}= \Fix^{\red}_{(x, \zeta)}(\Gr_{\leq \mu}) \cap L/P_{\vartheta}$ is already connected. This is clear since $F^{\vartheta}$ is a closed union of Bruhat cells by Lemma \ref{theorem: generic zeta -- fiber of fixed points and classical schubert varieties} and the finite Bruhat order has a unique minimal element.
\end{remark}

\begin{remark}[Reducible unions]
By the above result, the connected components of the fixed points are given by $\Fix^{\red, \vartheta}_{(x, \zeta)}(\Gr_{\leq \mu}) \subseteq  L/P_{\vartheta}$, $\vartheta \in \Theta$; each of those is a union of finite Schubert varieties. However, it may not be irreducible. 

An explicit example already appears when $G= GL_4$, $\Gr_{\leq \mu}$ for $\mu = (4,4,0,0)$, generic $\zeta$, and $x = (\zeta^3, \zeta^2, \zeta^1, 1)$. Then $\omega = \rho$, so the $\omega$-shifted action corresponds to the dot action. In total there are $|\wt(\mu)| = 85$ fixed points. The partition into connected components of $\Fix^{\red}_{(x, \zeta)}(\Gr_{\leq \mu})$ is given by
\begin{multline*}
    85 
    = 14 + 8 + 4 + 2 + 1 + 4 
    + 6 + 3 + (4+3-2) + (3+2-1)\\
    + 3 + (3+2-1) +3 + (2+4-1)
    +(4+3-2) + 2 + 1 + 1
    + 3 + 2 + 2 + 1
    +2
\end{multline*}
Here, each bracketed term corresponds to a single connected component given by a reducible union of two finite Schubert varieties intersecting along a third one.
\end{remark}

\subsection{\texorpdfstring{$\zeta$}{Zeta} root of unity}\label{section: zeta root of unity}

\subsubsection{Setup.}
We now discuss the complementary case when $\zeta$ is a root of unity. Fix such point $(x, \zeta) \in T \times \Gm^{\rot}$ and denote by $\ell = \ord(\zeta) \in \N$ the primitive order of $\zeta$.
We describe the fiber $\Fix^{\red}_{(x, \zeta)}(\Gr_G)$ as a finite union of $\ell$-dilated affine flag varieties for the resonant subgroup $L = L_{[x, \zeta]}$. Inside them, $\Fix^{\red}_{(x, \zeta)}(\Gr_{\leq \mu})$ is given by a closed union of $\ell$-dilated Iwahori orbits.

\subsubsection{Centralizers over roots of unity.}
\begin{lemma}\label{lemma: zeta root of unity -- computation of centralizer}
Inside $\lo L$, slanting identifies
\begin{equation*}
    \slant_{\omega}: J \xrightarrow{=} \lol L^{\sc}.
\end{equation*}
\end{lemma}
\begin{proof}
By Lemma \ref{lemma: root spaces of J}, the root subsystem of $J$ is given by  
$\Phi^{\bullet}_J = \{\beta - (m_{\beta} + \ell \Z)\hbar \mid \beta \in \Phi^{\bullet}_{L}\}$.
Slanting by $\omega$ as in Discussion \ref{discussion: slanting J in general}, we obtain
\begin{equation*}
\Phi^{\bullet}_{\slant_{\omega}(J)} = \{\beta - \ell \Z\hbar \mid \beta \in \Phi^{\bullet}_{L}\},
\end{equation*}
which is precisely the root subsystem $\Phi^{\bullet}_{\lol L^{\sc}}$ of $\lol L^{\sc} = \lol^{\circ} L$, embedded in $\lo G$ naively. Since both $J$ and $\lol L^{\sc}$ are connected, the result follows.
\end{proof}

\begin{remark}
As a sanity check, note that an element $t^{\lambda}$ in $\lo T$ lies in the centralizer of $(x, \zeta)$ if and only if $\lambda(\zeta) = 1$, in other words if and only if $\lambda \in \ell X_{\bullet}$. Consequently, $\Cent_{\lo G}(x) \cap \lo T = \lol T$. The identity componet of the centralizer then intersects $\lo T$ as $J \cap \lo T = \lol T^{\sc} = \slant_{\omega}(\lol T^{\sc})$.
\end{remark}

\subsubsection{Labeling set for components.}
Slanting on root systems $\slant_{\omega}: \Phi^{\bullet}_J \xrightarrow{=} \Phi^{\bullet}_{\lol L^{\sc}}$ from Remark \ref{remark: slanting on weyl groups general case} identifies the corresponding reflection subgroups.
\begin{lemma}\label{lemma: resonant subgroup -- zeta root of unity}
Slanting identifies
\begin{equation*}
    \slant_{\omega}: W_J \xrightarrow{=} \ell \Z\Phi^{\bullet}_L \rtimes W_L.
\end{equation*}
as the $\ell$-dilated affine Weyl group of $L$.
\end{lemma}
\begin{proof}
Clear from Lemma \ref{lemma: zeta root of unity -- computation of centralizer} since $W_J = W_{\Phi^{\bullet}_J}$ and $\ell \Z\Phi^{\bullet}_L \rtimes W_L = W_{\Phi^{\bullet}_{\lol L^{\sc}}}$.   
\end{proof}

We can now describe the labeling set $\Theta$ from Notation \ref{notation: Theta} as follows.
\begin{lemma}\label{lemma: labeling set for components -- zeta root of unity}
For $(x, \zeta)$ with $\zeta$ generic, we have
\begin{equation*}
 \Theta = |\overline{\Gamma}_{(x, \zeta)}| = X_{\bullet} / (W^{\aff}_{L}, \bullet^{\ell}_{\omega}).
\end{equation*}
\end{lemma}
\begin{proof}
 By Notation \ref{notation: Theta} and Lemma \ref{lemma: fiber of cocharacter graphs as lattice quotient}, we have $\Theta = |\overline{\Gamma}_{(x, \zeta)}|  = X_{\bullet}/W_J$ with respect to the natural action of $W_J$ on $X_{\bullet}$. 
Since the natural action of $W_J$ identifies with the $\omega$-shifted action of $\slant_{\omega}(W_J)$ by Remark \ref{remark: slanting on weyl groups general case}, we obtain the first identification in
 \begin{equation*}
     X_{\bullet}/W_J 
     = X_{\bullet}/(W_{\slant_{\omega}(J)}, \bullet_{\omega}) 
     = X_{\bullet}/(W_{L, \aff}, \bullet^{\ell}_{\omega}). 
 \end{equation*}
 The second identification follows since $W_{\slant_{\omega}(J)} = \Z\ell\Phi^{\bullet}_L \rtimes W_L$ is the $\ell$-dilated affine Weyl group of $L$ by Lemma \ref{lemma: resonant subgroup -- zeta root of unity}, whose orbits on $X_{\bullet}$ identify with the orbits of $W_{L, \aff}$ under the $\ell$-dilated action.
\end{proof}

\begin{remark}
Similar description works for other variants of cocharacter graphs -- for example,  
\begin{equation*}
   |{\Gamma}_{(x, \zeta)}|  = W_{\ext} / (W_{L, \aff}, \bullet^{\ell}_{\omega}). 
\end{equation*}
\end{remark}

\subsubsection{Positive centralizers as parahoric subgroups.}
Under slanting, we identify the positive centralizer $J^+$ as the $\ell$-dilated parahoric subgroup $\eP^{x}_{\ell} = \eP^{\omega}_{\ell}$ of $\lol L^{\sc}$.
More generally, we describe the stabilizers $J^{+}_{\lambda}$, $\lambda \in X_{\bullet}$.

\begin{lemma}\label{lemma: zeta root of unity -- positive stabilizers}
Inside $\lop G$, slanting identifies
\begin{equation*}
\slant_{\omega}: J^{+}_{\lambda} \xrightarrow{=} \eP^{x \cdot \lambda(\zeta)}_{\ell},
\end{equation*}    
where $\eP^{x \cdot \lambda(\zeta)}_{\ell} = \eP_{\ell}^{\omega+\lambda^{\ad}}$ is the standard opposite parahoric subgroup of $\lol L^{\sc}$ labeled by the point $x \cdot \lambda(\zeta) \in T$ resp. the cocharacter $\omega+\lambda^{\ad} \in X^{\ad}_{\bullet, L}$.
\end{lemma}
\begin{proof}
Note that $U_{\beta + m\hbar} \subseteq t^{\lambda} \cdot \lop G \cdot t^{-\lambda}$ if and only if $ t^{-\lambda} \cdot U_{\beta + m\hbar} \cdot t^{\lambda} \subseteq \lop G$. Since $ t^{-\lambda} \cdot U_{\beta + m\hbar} \cdot t^{\lambda} = U_{\beta + (m - \langle \beta, \lambda \rangle)\hbar}$ by Lemma \ref{lemma: conjugation of affine root spaces}, the root spaces of $t^{\lambda} \cdot \lop G \cdot t^{-\lambda}$ are given by
$\{ \beta + m\hbar \mid m - \langle \beta, \lambda \rangle \geq 0 \}$.
To describe the roots of $J^{+}_{\lambda}$, we intersect with the root system of $J$. This yields
\begin{equation*}
    \Phi^{\bullet}_{J^{+}_{\lambda}} = \{ \beta - (m_{\beta} + c\ell)\hbar \mid \beta \in \Phi^{\bullet}_L, \ c \in \Z, \ m_{\beta} + c\ell + \langle \beta, \lambda \rangle \leq 0 \}
\end{equation*}
Slanting by $\omega$ as in \S \ref{section: slanting} and rewriting $m_{\beta} + \langle \beta, \lambda \rangle = \langle \beta, \omega \rangle + \langle \beta, \lambda^{\ad} \rangle = \langle \beta, \omega + \lambda^{\ad} \rangle$, this becomes
\begin{equation*}
    \Phi^{\bullet}_{\slant_{\omega}(J^{+}_{\lambda})} = \{ \beta - c\ell\hbar \mid \beta \in \Phi^{\bullet}_L, \ c \in \Z, \ c\ell + \langle \beta, \omega+\lambda \rangle \leq 0 \}
\end{equation*}
where $\lambda^{\ad} \in X^{\ad}_{\bullet, L}$ is the projection of $\lambda \in X_{\bullet}$.

These are precisely the root spaces of the standard opposite parahoric subgroup $\eP^{\omega+\lambda^{\ad}}_{\ell}$ of the $\ell$-dilated loop group $\lol L^{\sc}$ associated to the cocharacter $\omega+\lambda^{\ad} \in X_{\bullet, L}^{\ad}$ as in \S \ref{subsubsection: parahoric subgroups}, \S \ref{subsubsection: dilated variants}. 
\end{proof}

\begin{remark}\label{remark: positive stabilizers -- zeta root of unity}
By Remark \ref{remark: independece of J vartheta on the lift}, the stabilizer $J^{+}_{\lambda}$ depends, up to isomorphism given by conjugation, only on the class $\vartheta \in \Theta = X_{\bullet} / (W^{\aff}_{L}, \bullet^{\ell}_{\omega})$. Denote $\vartheta$ the image of $\lambda$ under the projection to $\Theta$. In terms of the identification $\Theta = |\overline{\Gamma}_{(x, \zeta)}| \subseteq |T \sslash W|$, the element $\vartheta$ is given by the equivalence class of the point $x \cdot \lambda(\zeta)$. We denote the parahoric subgroup from Lemma \ref{lemma: zeta root of unity -- positive stabilizers} by
\begin{equation*}
 \eP^{\vartheta}_{\ell} 
 = \eP^{x \cdot \lambda(\zeta)}_{\ell} 
 = \eP_{\ell}^{\omega+\lambda^{\ad}}.  
\end{equation*}    
\end{remark}

In particular, we deduce the description of the positive centralizer.
\begin{lemma}\label{lemma: zeta root of unity -- positive centralizer}
\label{lemma: H^+ conjugates to P via alpha general case roots of unity}
Inside $\lop L$, slanting identifies
\begin{equation*}
\slant_{\omega}: J^{+} \xrightarrow{=} \eP_{\ell}^{x}.
\end{equation*}
where $\eP_{\ell}^{x} = \eP_{\ell, L}^{\omega} \leq \lo_{\ell} L$ is the standard opposite parahoric subgroup corresponding to $\omega$.
\end{lemma}
\begin{proof} 
Special case of Lemma \ref{lemma: zeta root of unity -- positive stabilizers} for trivial $\lambda$.
\end{proof}

\subsubsection{Fixed points on affine Grassmannian at \texorpdfstring{$\zeta$}{zeta} root of unity.}

We can now make explicit the fixed points on the whole affine Grassmannian.
\begin{theorem}\label{lemma: fixed points of (c, zeta) on Gr_G -- zeta root of unity}
Let $(x, \zeta) \in T \times \Gm^{\rot}$, where $\zeta$ is a root of unity of $\ord(\zeta) = \ell$. Then
\begin{equation*}
    \Fix^{\red}_{(x, \zeta)}(\Gr_G) = \coprod_{\vartheta \in X_{\bullet}/(W_{L, \aff}, \bullet^{\ell}_{\omega})} \Fl^{\vartheta}_{\ell, L^{\sc}}.
\end{equation*}
where $L=L_{[x, \zeta]}$ is the resonant subgroup, $\Fl^{\vartheta}_{\ell, L^{\sc}} = \lol L^{\sc}/ \eP^{\vartheta}_{\ell}$ is the $\ell$-dilated affine flag variety labeled by $\vartheta \equiv x \cdot \lambda(\zeta)$, and $W_{L, \aff}$ acts on $X_{\bullet}$ by the $\omega$-shifted $\ell$-dilated action.
\end{theorem}
\begin{proof}
Unravel Lemma \ref{lemma: fixed points on affine Grassmannian} in terms of the explicit identification above. The fixed points are a disjoint union of homogeneous spaces $J/J^+_{\vartheta}$, $\vartheta \in \Theta$. We identify $J = \lol L^{\sc}$ by Lemma \ref{lemma: zeta root of unity -- computation of centralizer}. The labeling set for components from Notation \ref{notation: Theta} becomes $\Theta = X_{\bullet} / \Theta = X_{\bullet}/(W_{\aff}, \bullet^{\ell}_{\omega} ) $ by Lemma \ref{lemma: labeling set for components -- zeta root of unity}. 
For given $\vartheta$, pick a lift $\lambda \in X_{\bullet}$. Slanting identifies $J^+_{\vartheta} = \eP^{\vartheta}_{\ell} = \eP^{x \cdot \lambda(\zeta)}_{\ell}$ by Lemma \ref{lemma: zeta root of unity -- positive stabilizers} and Remark \ref{remark: positive stabilizers -- zeta root of unity}. The result follows.
\end{proof}

\begin{remark}
For an alternative approach to the above theorem, we may again adapt the viewpoint of \S \ref{section: fixed points on GrG and GrL}, reducing the problem to the computation of (appropriate components) of the loop rotation by $\zeta$ on $\Gr_{L^{\ad}}$.

Since $\zeta$ is a primitive root of unity, these are given by
\begin{equation}\label{equation: alternative approach roots of unity}
    \Fix^{\red}_{(1, \zeta)}(\Gr_{L^{\ad}}) = \coprod_{\vartheta \in X^{\ad}_{\bullet, L}/ (W_{L, \aff}, \bullet^{\ell}_0)} \lol L^{\sc}/\eP^{\lambda}_{\ell}
\end{equation}
where $\lambda$ runs over the orbits of $(W_{\aff}, \bullet^{\ell}_0) \acts X_{\bullet}$ and $\eP^{\lambda}_{\ell} \leq \lol L^{\sc}$ is the associated standard opposite parahoric subgroup of the $\ell$-dilated loop group.

Here, we have identified $ X_{\bullet}/ (W_{\aff}, \bullet^{\ell}_{\omega}) \cong X_{\bullet}/ (W_{\aff}, \bullet^{\ell}_{0})$ by shifting with $\omega$, which is a cocharacter of $L^{\ad}$. 
Since $\ord(\zeta) = \ell$, the group $\mu_{\ell}$ is a finite cyclic group of order $\ell$ on the generator $\zeta$, so \eqref{equation: alternative approach roots of unity} follows by \cite[\S 1.6, \S 4.6]{RW22} or \cite[\S 2.1.4]{BBSV22}.
\end{remark}

\subsubsection{Fixed points on affine Schubert varieties as parahoric orbits.}
We now intersect the above description with any affine Schubert variety $\Gr_{\leq \mu}$ of $\Gr_G$.
\begin{theorem}\label{theorem: fiber of fixed points and classical schubert varieties general case at root of unity}
The scheme $\Fix^{\red}_{(x, \zeta)}(\Gr_{\leq \mu})$ is isomorphic to a closed union of $\eP_{\ell, L^{\sc}}^{\omega}$-orbits -- in particular $\Iw_{\ell, L^{\sc}}$-orbits -- on affine flag varieties $\Fl_{\ell, L}^{\vartheta}$ for the dilated loop group $\lo_{\ell} L^{\sc}$ of the simply connected form of the resonant subgroup $L = L_{[x, \zeta]}$.
\end{theorem}
\begin{proof}
Unraveling Lemma \ref{lemma: upshot of the general case} via the slanted descriptions from Lemmata \ref{lemma: fixed points of (c, zeta) on Gr_G -- zeta root of unity} and \ref{lemma: zeta root of unity -- positive centralizer} identifies the reduced scheme
$F = \Fix^{\red}_{(x, \zeta)}(\Gr_{\leq \mu})$
of our interest as a $\eP_{\ell}^{\omega}$-equivariant subvariety
\begin{equation*}\label{equation: B-equivariant closed immersion into classical flag varieties -- zeta rot of unity}
\eP_{\ell}^{\omega} \acts F \hookrightarrow \coprod_{\vartheta \in X_{\bullet}/(W_{L, \aff}, \bullet^{\ell}_{\omega})} \lo_{\ell} L^{\sc}/\eP_{\ell}^{\vartheta}   
\end{equation*}
where $\eP_{\ell}^{\omega} \leq \lo_{\ell} L^{\sc}$ acts by left translation on each of the affine flag varieties $\Fl_{\ell, L^{\sc}}^{\vartheta} = \lol L^{\sc}/\eP_{\ell}^{\vartheta}$.

Restricting attention to a single component $\Fl_{\ell, L^{\sc}}^{\vartheta}$ of the right-hand side, we get a closed union of $\eP_{\ell}^{\omega}$-orbits as desired. Since $\Iw_{\ell, L^{\sc}} \leq \eP_{\ell}^{\omega}$, it is in particular closed union of $\ell$-dilated Iwahori orbits.
\end{proof}

\subsubsection{Combinatorial description.}
Again, the dilated Iwahori orbits appearing in Theorem \ref{theorem: fiber of fixed points and classical schubert varieties general case at root of unity} may be described via the relevant torus fixed-points $t^{\lambda}$, labeled by the finite set $\wt(\mu)$ from \eqref{notation: weights in mu}.

\begin{lemma}\label{lemma: combinatorial description -- zeta root of unity}
The $T\times \Gm^{\rot}$-fixed-points on $\Fix^{\red}_{(x, \zeta)}(\Gr_{\leq \mu})$ are given by $\{ t^{\lambda} \mid \lambda \in \wt(\mu) \}$. Two such fixed-points $t^{\lambda_1}$, $t^{\lambda_2}$ lie in the same connected component of $\Fix^{\red}_{(x, \zeta)}(\Gr_{\leq \mu})$ if and only if $\lambda_1$, $\lambda_2$ lie in the same orbit of the $\omega$-shifted $\ell$-dilated action $(W_{L, \aff}, \bullet^{\ell}_{\omega}) \acts X_{\bullet}$.    
\end{lemma}
\begin{proof}
Follows by unraveling Lemma \ref{lemma: connected components of fixed points n affine schubert variety} in terms of Lemma \ref{lemma: zeta root of unity -- positive centralizer} Theorem \ref{theorem: fiber of fixed points and classical schubert varieties general case at root of unity}.
\end{proof}

As in the finite case, Lemma \ref{lemma: combinatorial description -- zeta root of unity} shows that each $F^{\vartheta} = \Fix^{\red}_{(x, \zeta)}(\Gr_{\leq \mu}) \cap \lol L^{\sc}/\eP_{\ell}^{\vartheta}$ is already connected.

%% file: 5_bezrukavnikov_finkelberg_description_of_equivariant_K_theory.tex
\section{Equivariant \texorpdfstring{$K$}{K}-theory ring of the affine Grassmannian}\label{section: equivariant K-theory ring of the affine grassmannian}

The aim of this section is to establish an algebro-geometric description of the ring $K^{\top, 0}_{T \times \Gm^{\rot}}(\Gr_{G}; \C)$, the main focus of this paper. Our Theorem \ref{theorem: equivariant K-theory as deformation to normal cone -- G sssc} presents this ring in terms of completed global functions on affine schemes of representation-theoretical origin. These are given by gluing together infinitely many deformations to normal cones at quasi-diagonals $Z_{\ell}$ in two copies of the torus $T$. We will see that this description can be made very explicit.

Here, let us point out some related literature. The historically first description \cite{KK90} gives $K$-theory of the affine Grassmannian as a dual to its compactly supported variant. To further describe the ring structure, one can appeal to the GKM description \cite{HHH05}. However, this only realizes $K^{\top, 0}_{T \times \Gm^{\rot}}(\Gr_G; \C)$ as a subring of an infinite product of Laurent polynomial rings, cut out by an infinite set of equalizer conditions -- to get a concrete understanding of the ring, one needs to solve this nontrivial system of congruences. Our result may be interpreted as indirectly doing so.

For most of the discussion, we assume $G=G^{\sc}$ to be simply connected or $GL_n$. (While the definitions work for any reductive $G$ as stated, they have to be carefully modified to give the correct answer for $K$-theory.)

We first recall the case of cohomology in \S \ref{section: recollaction of cohomology} and then the GKM description of $K$-theory in \S \ref{section: Goresky--Kottwitz--McPherson description of equivariant $K$-theory}.
We then move on to define the quasi-diagonals inside $T \times \Gm^{\rot} \times T \sslash W$ in \S \ref{section: quasi-diagonals at roots of unity} and compare them to specializations of cocharacter graphs. We blow up these quasi-diagonals in \S \ref{section: deformations of quasi-diagonals to normal cones}, obtaining the deformation to their normal cones. We then explain in \S \ref{section: map to equivariant K-theory: combinatorics}, in purely combinatorial terms, how the completed ring of functions on this deformation maps to the GKM ring of the affine Grassmannian. We then recast this argument in terms of classes of vector bundles in \S \ref{section: map to equivariant K-theory: topology}. In fact, one may view these arguments independently. Moreover, both constructions work integrally. We then prove our main result, Theorem \ref{theorem: equivariant K-theory as deformation to normal cone -- G sssc}, with $\C$-coefficients for simply connected group $G$ in \S \ref{section: bezrukavnikov--finkelberg description of equivariant K-theory}.
We discuss its generalizations in \S \ref{section: variants and generalizations} -- this includes arbitrary reductive groups, arbitrary affine flag varieties, and change of equivariance.

\subsection{Recollection of cohomology}\label{section: recollaction of cohomology}
Let $k=\C$. Before we embark to describe the $K$-theory, we recall the well-known case of cohomology of $\Gr_G$, following \cite{Gin95, BF07, Yun10}. This argument has two parts: the purely topological computation of non-equivariant cohomology of $\Gr_G$, and the subsequent description of its equivariant version as a deformation to normal cone.

\subsubsection{Cohomology of finite flag varieties.}
As a warmup, recall the equivariant cohomology of finite flag varieties, which will be useful later. Let $\k$ be a coefficient ring.
\begin{lemma}\label{lemma: equivariant cohomology of finite flag varieties}
Let $G$ be a reductive group and $T \leq G$ be a maximal torus over $k$. Let $P_J \leq G$ be a parabolic subgroup and $W_2 \leq W$ corresponding Weyl group. Assume $\k$ is a field of $\chara \k = 0$. Then we have a graded ring isomorphism
\begin{equation*}
    \O(\t \times_{\t \sslash W} \t \sslash W_2) \xrightarrow{=} H^{\bullet}_{T}(G/P_J; \k),
\end{equation*}
where the grading on the left-hand side comes from the weight two action on both copies of $\t$.
More generally, take a pair of Levi subgroups $L_1, L_2 \leq G$ with Weyl groups $W_1, W_2 \leq W$. Denote $P_J$ the standard parabolic associated to $L_J := L_2$. Then we have a graded ring isomorphism
\begin{equation*}
    \O(\t \sslash W_1 \times_{\t \sslash W} \t \sslash W_2) \xrightarrow{=} H^{\bullet}_{L_1}(G/P_J; \k).
\end{equation*}
\end{lemma}
\begin{proof}
The $G$-equivariant cohomology is given by 
\begin{equation*}
 H^{\bullet}_G(G/P_J; \k) = H^{\bullet}(G \backslash G / P_J; \k) = H^{\bullet}(\pt / P_J; \k) = H^{\bullet}(\pt / L_J; \k) = H^{\bullet}_{L_J}(\pt; \k),   
\end{equation*}
the equivariant base for $L_J$. Consequently, we have a natural identification
\begin{equation*}
 \O(\t \sslash W_J) \xrightarrow{=} H^{\bullet}_G(G/P_J; \k)   
\end{equation*}
compatibly with the structure map to the equivariant base $\O(\t \sslash W) \xrightarrow{=} H^{\bullet}_G(\pt; \k)$. By equivariant formality, the $T$-equivariant cohomology is given by base change along $H^{\bullet}_G(\pt; \k) \to H^{\bullet}_T(\pt; \k)$, so we get the desired
\begin{equation*}
\O(\t \times_{\t \sslash W} \t \sslash W_J) \xrightarrow{=} H^{\bullet}_{T}(G/P_J; \k).    
\end{equation*}
Similarly, $L_1$-equivariant cohomology is given by base change along $H^{\bullet}_G(\pt; \k) \to H^{\bullet}_{L_1}(\pt; \k)$, proving the more general statement.
\end{proof}

\begin{remark}\label{remark: finite flag varieties with trivial loop rotation}
Consequently, adding the equivariance with respect to the trivial loop rotation, we have
\begin{equation*}
   \O(\Ga^{\rot} \times \t \sslash W_1 \times_{\t \sslash W} \t \sslash W_2) \xrightarrow{=} H^{\bullet}_{L_1 \times \Gm^{\rot}}(G/P_J; \k). 
\end{equation*}
In particular, this statement holds after completion $(-)^{\wedge}_{(0, 0)}$ over the origin $(0,0) \in \t \times \Ga^{\rot}$.
\end{remark}

\begin{remark}\label{remark: GKM for cohomology of finite flag varieties}
In terms of their GKM presentations, this description works as follows. Take the diagonal $\t\sslash W \hookrightarrow \t \sslash W \times \t \sslash W$. Consider its base change $\Delta^J := \t \times_{\t\sslash W} \t\sslash W_J \hookrightarrow \t \times \t\sslash W_J$. This defines a closed subscheme parametrizing $\{ (x, y) \in \t \times \t \sslash W_J \mid x = y \mod W  \}$. It has irreducible components isomorphic to copies of $\t$ labeled by the finite set $W/W_J$. The GKM localization to fixed-points corresponds to restriction to these components. 
\end{remark}

\begin{remark}\label{remark: cohomology of finite flag varietes as coinvariant algebra}
Taking fiber over the origin in the equivariant base, we get the usual description of non-equivariant cohomology, namely
\begin{equation*}
    H^{\bullet}(G/P_J; \k) = \frac{\k[t_1, \dots, t_n]}{\k[t_1, \dots, t_n]^W}. 
\end{equation*}
\end{remark}

\begin{remark}
Upon taking the correct equivariant base, the discussion in this section works for any equivariant multiplicative generalized cohomology theory $E_G$. In particular, it works for $K$-theory: for any coefficient ring $\k$, for example $\k=\Z$, we have
\begin{equation*}
    \O(T \times_{T \sslash W} T \sslash W_J) = K^{\top, 0}_{T \times \Gm^{\rot}}(G/P_J; \k).
\end{equation*}
All the above remarks are valid in this setting.
\end{remark}

\subsubsection{Based loop space description of ordinary cohomology.}\label{section: based loop space description}
Assume $G$ is a simply connected group over $k =\C$. The non-equivariant cohomology of the affine Grassmannian $\Gr_G$ can be computed by the following classical argument, going back to \cite[\S 1.7]{Gin95}. 

Firstly, denote $K$ a connected compact real Lie group with complexification $G(\C)$. Then $K \leq G(\C)$ is a maximal compact subgroup. Denote $\Omega K$ its based loop space. It is then a nontrivial fact that we can identify the homotopy type of the affine Grassmannian as $\Gr_G(\C) \simeq \Omega K$, see for example \cite[Introduction]{YZ09}, \cite[\S 1.2]{Gin95}, \cite[Remark 5.1]{HHH05}, \cite[\S 1.6]{Zhu15}.

We are thus left to understand the graded ring $H^{\bullet}(\Omega K)$. Firstly, the cohomology $H^{\bullet}(K)$ is understood by \cite{Bor53}. In particular, with coefficients in a field $\k$ of $\chara \k = 0$, the cohomology ring $H^{\bullet}(K; \k) = \Lambda_{\k}[x_1, \dots, x_n]$ is the exterior algebra on finitely many generators $x_i$, $i=1, \dots, d$ by \cite[Proposition 7.2]{Bor53}. Furthermore, $d = \rank G$ and $\deg x_i = a_i -1$ where $a_1, \dots a_d$ are the exponents of $G$.

Consequently, the rational homotopy type of $K$ is that of a product of odd dimensional spheres $S^{a_i-1}$. By the compatibility of the based loop space $\Omega(-)$ with Eilenberg-MacLane spaces, the rational homotopy type of $\Omega K$ is given by a product of even dimensional spheres $S^{a_i - 2}$, and we deduce
\begin{equation}\label{equation: cohomology of loop groups}
    H^{\bullet}(\Omega K; \k) = \k[y_1, \dots, y_n], \qquad \deg y_i = a_i - 2
\end{equation}
In a more hands-on way, the passage from $K$ to $\Omega K$ can be also done directly through the path space fibration $\Omega K \to P K \to K$ and Serre's spectral sequence; the elements $x_i$ are transgressions of $y_i$. See \cite[Example 1.16]{Hat04}.

\subsubsection{Bezrukavnikov--Finkelberg description of equivariant cohomology.}
We recall the description of the equivariant cohomology of affine flag varieties over $k=\C$ via certain deformations to normal cones, following \cite{BF07, Yun10}. This crucially uses the topological input \S \ref{section: based loop space description}. We assume the coefficient ring $\k$ is a field of $\chara \k = 0$.

Let us sketch this argument.
The equivariant cohomology of $\Fl^{J} = \lop G / \eP_{J}$ is a module over the equivariant base via
\begin{equation}\label{equation: equivariant cohomology as modules over equivariant base}
    H^{\bullet}_{G \times \Gm^{\rot} \times L_{J}}(\pt; \k) \to H^{\bullet}_{\lop G \times \Gm^{\rot}}(\Fl^{J}; \k).
\end{equation}
The map \eqref{equation: equivariant cohomology as modules over equivariant base} is easily seen to be injective. The source is given by space $\O(\t \sslash W \times \Ga^{\rot} \times \t \sslash W_{L_J})$. In fact, both sides are abstractly isomorphic polynomial rings.

However, \eqref{equation: equivariant cohomology as modules over equivariant base} is not surjective. To correct for the missing functions, one needs to deform to the normal cone of the image of $\t$  under the diagonal projection $\t \to \t \sslash W \times \t \sslash W_{L_J}$, in other words to the normal cone of the closed subscheme $\t \sslash W \times_{\t \sslash W} \t \sslash W_{L_J}$.
One directly checks that \eqref{equation: equivariant cohomology as modules over equivariant base} factors injectively through this deformation.
Finally, to prove surjectivity, one notes that \eqref{equation: equivariant cohomology as modules over equivariant base} is a graded map with respect to the natural gradings on both sides. The graded pieces are finite-dimensional, so it is enough to compare their Hilbert series. These are easily computable from the deformation description for the left-hand side, and \eqref{equation: cohomology of loop groups} together with equivariant formality for the right-hand side.

To state the result, consider Levi subgroups $L_1, L_2$ of $G$, with associated Weyl groups $W_1, W_2 \leq W$. 
Let $\eP_J$ be the standard parahoric sugroup associated to $L_J :=L_2$. Conisder the corresponding affine flag variety $\Fl^J = \lo G / \eP_J$, with the natural left action $\lop L_1 \rtimes \Gm^{\rot} \acts \Fl^J$.
\begin{proposition}\label{proposition: cohomology of affine flag varieties as deformation to the normal cone}
 The $(\lop L_1 \times \Gm^{\rot})$-equivariant cohomology ring of the affine flag variety $\Fl^{J}$ is given, together with its structure map, by  
 \begin{equation}
    \begin{tikzcd}
        \O(N( \t \sslash W_1 \times_{\t \sslash W} \t \sslash W_2, \t \sslash W_1 \times \t \sslash W_2)) \arrow[r, equal]& H_{\lop L_1 \times \Gm^{\rot}}(\Fl^{J}; \k) \\
         \O(\t \sslash W_1 \times \Ga^{\rot} \times \t \sslash W_2) \arrow[u, hookrightarrow] \arrow[r, equal] & H_{L_1 \times \Gm^{\rot} \times L_{2}}(\pt; \k) \arrow[u, hookrightarrow] 
    \end{tikzcd}
 \end{equation}
\end{proposition}
\begin{proof}
The case of $\lop G \rtimes \Gm^{\rot} \acts \Gr_G$ is covered in \cite[Theorem 1]{BF07}. From there, the general case follows by the argument \cite[Theorem 4.3]{Yun10}. Strictly speaking, \cite{Yun10} treats only the full flag variety with the Iwahori action, but the argument works in general.
\end{proof}

\begin{remark}[Informal GKM localization]\label{remark: informal GKM localization}
Working equivariantly with respect to $T \times \Gm^{\rot}$, the above identification can be unraveled in GKM terms. Let us give an informal account in the complex analytic setting now; we will give a detailed discussion in the $K$-theoretical context below.

On the complex analytic level, consider $\t(\C) \sslash W_{\aff} \hookrightarrow \t(\C) \sslash W_{\aff} \times \Ga \times \t(\C) \sslash W_{\aff}$. By pullback, we obtain the subspace $\overline{\mathbf{\Gamma}}(\C) \hookrightarrow \t(\C) \times \Ga(\C) \times \t(\C) \sslash W_{\aff}$. Its irreducible components are copies of $\t(\C)$, labeled by the cocharacter lattice $X_{\bullet}$. Take the strict transform $\overline{\mathbf{\Upsilon}}(\C)$ of $\overline{\mathbf{\Gamma}}(\C)$ in the blowup $N( \t(\C) \sslash W_1 \times_{\t(\C) \sslash W} \t(\C) \sslash W_2, \t(\C) \sslash W_1 \times \t(\C) \sslash W_2)$. Its irreducible components are still copies of $\t(\C)$ labeled by $X_{\bullet}$, and the localization corresponds to restriction of functions to them. Also see \cite[\S 3.2]{BF07}.
\end{remark}

\subsection{Goresky--Kottwitz--McPherson description of equivariant \texorpdfstring{$K$}{K}-theory} \label{section: Goresky--Kottwitz--McPherson description of equivariant $K$-theory}
We record some standard formality results on equivariant topological $K$-theory of $\Gr_G$, together with its well-known GKM description. In this subsection, $G$ is any reductive group over $k=\C$.

\subsubsection{Limit along affine Schubert varieties.} Recalling \S \ref{subsubsection: K-theory of colimits}, consider the affine Grassmannian $\Gr_G$ with its $T\times \Gm^{\rot}$-equivariant Iwahori paving by affine cells. Consider the induced skeleta $\sk_i = \sk_i \Gr$, $i \in \N_0$. By induction, $K^{\top, \bullet}_{T \times \Gm^{\rot}}(\sk_i)$ is a free module over $K^{\top, \bullet}_{T \times \Gm^{\rot}}(\pt)$; it is $2$-periodic and concentrated in even degrees; the affine cells give a basis of the degree zero part. The restriction maps induced by the skeletal inclusions $\sk_i \hookrightarrow \sk_{i'}$ are thus surjective; the inverse system in particular satisfies the Mittag--Leffler condition. This paving restricts to each affine Schubert variety $\Gr_{\leq \mu}$, where the same results hold. We deduce the following.
\begin{proposition}\label{proposition: completion of K-theory}
We have an identification of rings
 \begin{equation*}
  K^{\top, *}_{T \times \Gm^{\rot}}(\Gr_G) = \lim_{\mu \in X^+_{\bullet}} K^{\top, *}_{T \times \Gm^{\rot}}(\Gr_{\leq \mu}).   
\end{equation*} 
All terms are concentrated in the even part.
\end{proposition}
Since we do not loose anything by focusing on the degree zero piece $K^{\top, 0}_{T \times \Gm^{\rot}}(\Gr_G)$ of the $K$-theory ring, we do so from now on.

\subsubsection{GKM ring of affine Grassmannian.}\label{subsubsection: GKM ring of affine Grassmannian}
The action of the extended torus $T \times \Gm^{\rot}$ on $\Gr_G$ satisfies the GKM assumptions for integral topological $K$-theory; it thus may be described as an explicit subring of $K$-theory of the fixed points \cite{HHH05} and Proposition \ref{proposition: GKM of HHH}.
\begin{notation}\label{notation: GKM of Gr}
Let $\k$ be a base ring.
The GKM ring of the affine Grassmannian $\GKM^{0}_{T \times \Gm^{\rot}}(\Gr_G; \k) \subseteq \prod_{\lambda \in X_{\bullet}} \O(T \times \Gm^{\rot})$ is given by the subring
\begin{equation*}
    \GKM^{0}_{T \times \Gm^{\rot}}(\Gr_G; \k) 
    = \bigl\{ (f_{\lambda})_{\lambda} \in \prod_{\lambda \in X_{\bullet}} \O(T \times \Gm^{\rot}) \quad \big| \quad f_{\lambda} \equiv f_{s_{\beta + m\hbar} \cdot \lambda} \mod (1 - {\beta}q^m), \quad \forall \lambda, \forall s_{\beta + m\hbar} \bigr\},
\end{equation*}
where the congruence conditions run over all cocharacters $\lambda \in X_{\bullet}$ and all affine reflections $s_{\beta +m\hbar}$ associated to affine roots $\beta + m\hbar \in \Phi^{\bullet} \oplus \Z\hbar$.
\end{notation}

\begin{proposition}\label{proposition: GKM description of K-theory of Gr}
Over $\k=\Z$, we have
\begin{equation*}
K^{\top, 0}_{T \times \Gm^{\rot}}(\Gr_G; \k) = \GKM^{0}_{T \times \Gm^{\rot}}(\Gr_G; \k). 
\end{equation*}
\end{proposition}
\begin{proof}
The action of $T \times \Gm^{\rot} \acts \Gr_G$ satisfies the GKM assumptions for integral topological $K$-theory; this is the case for all Kac--Moody partial flag varieties by with respect to the natural torus \cite[\S 5]{HHH05}; let us sketch some details.

For $\Gr_G$, the fixed points of $T\times \Gm^{\rot}$ are given by $(t^{\lambda} \mid \lambda \in X_{\bullet})$ and we claim
\begin{equation*}
\begin{tikzcd}[row sep=5]
\GKM^{0}_{T \times \Gm^{\rot}}(\Gr_G; \k) \arrow[r, hookrightarrow] & \prod_{\lambda \in X_{\bullet}} \O(T \times \Gm^{\rot})\\
K^{\top, 0}_{T \times \Gm^{\rot}}(\Gr_G; \k) \arrow[u, equal] \arrow[r, hookrightarrow] & \prod_{\lambda \in X_{\bullet}} K^{\top, 0}_{T\times \Gm^{\rot}}(t^{\lambda}; \k) \arrow[u, equal].
\end{tikzcd}
\end{equation*}
Here, the Iwahori orbits on $\Gr_G$ provide a $T\times \Gm^{\rot}$-equivariant affine paving. Via the construction from Remark \ref{remark: GKM in the BB case}, the Assumptions \ref{assumption: HHH assumptions} are satisfied by \cite[\S 5, p.211]{HHH05}. The ring $K^{\top, 0}_{T \times \Gm^{\rot}}(\Gr_G; \k)$ thus has a GKM description by Proposition \ref{proposition: GKM of HHH}. Unraveling this gives precisely the ring from Notation \ref{notation: GKM of Gr} as in \cite{HHH05}. 
\end{proof}

\begin{remark}
Also see \cite{HHH05}, \cite[\S 2.2.3]{BBSV22}, \cite{Par17}, \cite[Appendix A]{Sit24}, \cite{Shi14} for the parallel overview of complexified cohomology of $\Gr_G$ in our situation.
\end{remark}

\begin{remark}
Given an affine Schubert variety $\Gr_{\leq \mu}$, the corresponding GKM ring is given by passing to the subset $\wt(\mu) \subseteq X_{\bullet}$ as
\begin{equation*}
    \GKM^{0}_{T \times \Gm^{\rot}}(\Gr_{\leq \mu}; \k) 
    = \bigl\{ (f_{\lambda}) \in \prod_{\lambda \in \wt(\mu)} \O(T \times \Gm^{\rot}) \quad \big| \quad f_{\lambda} \equiv f_{\lambda'} \mod (1 - {\beta}q^m) \bigr\},
\end{equation*}
where the conditions run over all pairs $\lambda, \lambda' \in \wt(\mu)$ such that $\lambda' = s_{\beta + m\hbar} \cdot \lambda$ for a simple reflection associated to some affine root $\beta + m\hbar \in \Phi^{\bullet} \oplus \Z \hbar$. We again have
\begin{equation*}
K^{\top, 0}_{T \times \Gm^{\rot}}(\Gr_{\leq \mu}; \k) = \GKM^{0}_{T \times \Gm^{\rot}}(\Gr_{\leq \mu}; \k). 
\end{equation*}
The restriction maps from $K^{\top, 0}_{T \times \Gm^{\rot}}(\Gr_G; \k)$ on the left correspond to restriction along $\wt(\mu) \subseteq X_{\bullet}$ on the right. Note that $\Gr_{\leq \mu}$ are GKM spaces in the more standard sense -- the number of fixed points is finite.
\end{remark}

Our aim is to describe this $K$-theory ring, in other words explicitly solve the system of congruences defining $\GKM^0_{T \times \Gm^{\rot}}(\Gr; \k)$ when $\k=\C$. However, it is nontrivial to do so directly, and we apply other topological tricks. In any case, our Theorem \ref{theorem: equivariant K-theory as deformation to normal cone -- G sssc} below may be regarded as a solution to this purely combinatorial problem.

\subsection{Quasi-diagonals at roots of unity}\label{section: quasi-diagonals at roots of unity}
Let $\k$ be any coefficient ring and $G$ a split reductive group over $\k$ with maximal torus $T$ and Weyl group $W$.
In \S \ref{subsubsection: relationship to fibers of Gamma}, \S \ref{subsubsection: generatirs of the ideal sheaf of Zl} we assume $G=G^{\sc}$ to be simply connected. 

\subsubsection{Power $\ell$ map.}
Consider $T \times T$; we denote respectively $t_1, \dots, t_n$ and $s_1, \dots, s_n$ coordinates on the first and second copy. 
Denote
\begin{equation*}
    (-)^{\ell} : T \to T, \qquad t \mapsto t^{\ell}
\end{equation*}
the $\ell$-th power map, coming from the group structure on $T$. In coordinates, it sends $t_i \mapsto t_i^{\ell}$. Since the $W$-action on $T$ is by group automorphisms, the $\ell$-th power map is $W$-equivariant. It thus induces a map
\begin{equation*}
    (-)^{\ell} : T \sslash W \to T \sslash W.
\end{equation*}
The corresponding map $\O(T \sslash W) \to \O(T \sslash W)$ on rings of functions is induced by the substitution $f(t_1, \dots, t_n) \mapsto f(t^{\ell}_1, \dots, t^{\ell}_n)$. Similarly for the quotient by any subgroup of $W$.

\subsubsection{Quasi-diagonals at roots of unity.}
For any $\ell \in \N$, we define the corresponding \textit{quasi-diagonal} by the pullback
\begin{equation*}
\begin{tikzcd}
Z^{T \sslash W, T\sslash W}_{\ell} \arrow[d] \arrow[r]  & T \sslash W \arrow[d, "(-)^{\ell}"]  \\
T \sslash W \arrow[r, "(-)^{\ell}"] & T \sslash W
\end{tikzcd}
\end{equation*}
It has trivial derived structure by flatness of $(-)^{\ell}$.
It is a closed subscheme of $T \sslash W \times T \sslash W$ via the pullback square
\begin{equation*}
    \begin{tikzcd}
        Z^{T \sslash W, T\sslash W}_{\ell} \arrow[r, hook] \arrow[d] & T\sslash W \times T\sslash W \arrow[d, "(-)^{\ell} \times (-)^{\ell}"] \\
        T\sslash W \arrow[r, hook, "\Delta"] & T\sslash W \times T\sslash W
    \end{tikzcd}
\end{equation*}
On its functor of points, it parametrizes
\begin{equation*}
Z^{T \sslash W, T\sslash W}_{\ell}: R \mapsto \{ (x, y) \in (T\sslash W \times T\sslash W)(R) \mid x^{\ell} = y^{\ell} \}.
\end{equation*}

\subsubsection{Variants.}
More generally, for any two subgroups of $ W_1, W_2 \leq W$, we define the variant $Z^{W_1, W_2}_{\ell} = Z_{\ell}^{T \sslash W_1, T\sslash W_2}$ by either of the pullback squares
\begin{equation*}
\begin{tikzcd}
Z_{\ell}^{T \sslash W_1, T\sslash W_2} \arrow[d] \arrow[r]  & T \sslash W_2 \arrow[d, "(-)^{\ell}"]  \\
T \sslash W_1 \arrow[r, "(-)^{\ell}"] & T \sslash W
\end{tikzcd}
\qquad
\begin{tikzcd}
Z^{T \sslash W_1, T\sslash W_2}_{\ell} \arrow[r, hook] \arrow[d] & T\sslash W_1 \times T\sslash W_2 \arrow[d, "(-)^{\ell} \times (-)^{\ell}"] \\
T\sslash W \arrow[r, hook, "\Delta"] & T\sslash W \times T\sslash W.
\end{tikzcd}
\end{equation*}
It is a closed subscheme of $T\sslash W_1 \times T\sslash W_2$, obtained by the base change of $Z^{T \sslash W, T \sslash W}_{\ell}$ along the faithfully flat map $T\sslash W_1 \times T\sslash W_2 \to T\sslash W \times T\sslash W$.
On its functor of points, it parametrizes
\begin{equation*}
Z^{T\sslash W_1 \times T\sslash W_2}_{\ell}(R) = \{ (x, y) \in (T\sslash W_1 \times T\sslash W_2)(R) \mid x^{\ell} W = y^{\ell} W \}.    
\end{equation*}

\begin{notation}\label{notation: quasi-diagonals}
In particular, we will use the following shorter notation for two special cases:
\begin{align*}
Z'_{\ell} & := Z^{T, T}_{\ell} \hookrightarrow T \times T,   \\
Z_{\ell} & := Z^{T, T\sslash W}_{\ell} \hookrightarrow T \times T \sslash W.
\end{align*}
\end{notation}

\begin{notation}
Furthermore, given $\zeta \in \mu_{\infty}$ of primitive order $\ell$, we will also use the notation $Z_{\zeta} = Z_{\ell}\hookrightarrow T \times T \sslash W$. Given any point $x \in T$ in the first copy of the torus, we denote its fiber over $x$ by $Z_{(x, \zeta)} = Z_{(x, \ell)} \hookrightarrow T \sslash W$.    
\end{notation}

\subsubsection{Relationship to fibers of $\Gamma$.}\label{subsubsection: relationship to fibers of Gamma}
Assume now $\k$ is an algebraically closed field of characteristic zero. Assume $G$ is simply connected.
Given $\zeta \in \mu_{\infty}$ of primitive order $\ell$, we have just defined the subscheme $Z_{\zeta}$ and its fibers $Z_{(x, \zeta)}$.
These are fully described by specializations of the sub-ind scheme $\overline{\Gamma} \hookrightarrow T \times \Gm^{\rot} \times T \sslash W$ from \S \ref{section: cocharacter graphs}.

\begin{lemma}\label{lemma: Z is Gamma}
Let $(x, \zeta) \in T \times \Gm^{\rot}$. Suppose $\zeta$ is a root of unity of primitive order $\ell$. On the underlying points, we have
\begin{equation*}
  |Z'_{\ell}| = |{\Gamma}_{\zeta}|,  \qquad \qquad |Z'_{(x, \ell)}| = |{\Gamma}_{(x, \zeta)}|, \qquad \qquad
  |Z_{\ell}| = |\overline{\Gamma}_{\zeta}|,  \qquad \qquad |Z_{(x, \ell)}| = |\overline{\Gamma}_{(x, \zeta)}|.
\end{equation*}
\end{lemma}
\begin{proof}
   It is enough to treat the $Z'_{\ell}$-variants, since the $Z_{\ell}$-variant follows by taking quatient by $W$. Moreover, it is enough to focus on the case of $Z'_{(x, \ell)}$ (which is a fiberwise refinement of $Z'_{\ell}$). Over any $(x, \zeta)$, both sets have a transitive action of $W_{\aff}$ and base point with the same stabilizer.
\end{proof}

\begin{remark}
    More explicitly, we may just compute inside $T \times T$ as follows
    \begin{align*}
       |Z'_{\ell}|  
       &=  \{ (x, y) \mid \exists w \in W: x^{\ell} = w(y^{\ell}) \} \\ 
       &= \{ (x, y) \mid \exists w \in W: (x \cdot w(y)^{-1})^{\ell} = 1 \} \\
       &= \{ (x, y) \mid \exists \lambda \in X_{\bullet}, \exists w \in W: x \cdot w(y)^{-1} = \lambda(\zeta) \} \\
       &= |\Gamma_{\zeta}| 
    \end{align*}
    The case of $Z'_{(x, \ell)}$ is the same computation with fixed $x$.
\end{remark}

\begin{remark}
The statement of Lemma \ref{lemma: Z is Gamma} holds on the level of reduced schemes: it identifies $Z_{\ell} = \overline{\Gamma}_{\zeta}^{\red}$ and $Z_{(x, \ell)} = \overline{\Gamma}_{(x, \zeta)}^{\red}$. 

However, while $Z_{\ell}$ is reduced by definition, the fiber $\overline{\Gamma}_{\zeta}$ will acquire some ind-unreduced structure from the $q$-direction. In fact, one may heuristically think about the normal cone to $Z_{\ell}$ as a counterpart of this ind-unreduced structure.
\end{remark}

\subsubsection{Generators of the ideal sheaf of $Z_{\ell}$.}\label{subsubsection: generatirs of the ideal sheaf of Zl}
Let again $\k$ be a general coefficient ring.
Consider the ring $\O(T \sslash W)$ over $\k$. It naturally identifies with the representation ring $R(G; \k)$. Denote $e_1, \dots, e_n \in \O(T \sslash W)$ the elements corresponding to the fundamental representations of $G$. Since $G$ is simply connected or $GL_n$, these freely generate $\O(T \sslash W)$ by \S \ref{subsubsection: coordinates and invariant polynomials}.

\begin{lemma}\label{lemma: explicit generators for the ideal sheaf of Zl}
The ideal sheaf $\eI_{Z_{\ell}}$ of the closed subscheme $Z^{T\sslash W \times T \sslash W}_{\ell} \hookrightarrow T\sslash W \times T \sslash W$ is generated by
\begin{equation*}
 f_i = e_i(t_1^{\ell}, \dots, t_n^{\ell}) - e_i(s_1^{\ell}, \dots, s_n^{\ell}), \qquad i = 1, \dots, n    
\end{equation*}
The same works for the ideal sheaf $\eI_{Z_{\ell}}^{T \sslash W_1, T \sslash W_2}$ of $Z_{\ell}^{T \sslash W_1, T \sslash W_2} \hookrightarrow T \sslash W_1 \times T \sslash W_2$ for any $W_1, W_2 \leq W$.
\end{lemma}
\begin{proof}  

By definition, $Z_{\ell}$ fits into the pullback square
\begin{equation}\label{diagram: fiber square for Zl, variant II}
    \begin{tikzcd}
        Z_{\ell} \arrow[r, hook] \arrow[d] & T\sslash W \times T\sslash W \arrow[d, "(-)^{\ell} \times (-)^{\ell}"] \\
        T\sslash W \arrow[r, hook, "\Delta"] & T\sslash W \times T\sslash W
    \end{tikzcd}
\end{equation}
Note that $T\sslash W$ is smooth (in fact an affine space) and the map $(-)^{\ell}: T\sslash W \to T \sslash W$ is finite (which is obvious before taking the quotient). The same is true for $(-)^{\ell} \times (-)^{\ell}: T\sslash W \times T\sslash W \to T\sslash W \times T\sslash W$, which is thus flat by miracle flatness \cite[Tag 00R4]{Sta}. Since it is surjective, it is in fact faithfully flat.

The closed immersions given by the horizontal arrows in \eqref{diagram: fiber square for Zl, variant II} induce short exact sequence
\begin{equation*}
\begin{tikzcd}
0 \arrow[r] & \eI_{Z_{\ell}} \arrow[r] & \O_{T\sslash W \times T\sslash W} \arrow[r] & \O_{Z_{\ell}} \arrow[r] & 0 \\
0 \arrow[r] & \eI_{T\sslash W} \arrow[r] \arrow[u] & \O_{T\sslash W \times T\sslash W} \arrow[r] \arrow[u] & \O_{T\sslash W} \arrow[u] \arrow[r] & 0 
\end{tikzcd}
\end{equation*}
The upper sequence is given by tensoring the bottom one along $(-)^{\ell} \times (-)^{\ell}: \O_{T\sslash W \times T\sslash W} \hookrightarrow \O_{T\sslash W \times T\sslash W}$. Indeed, this is true for the right-hand term by \eqref{diagram: fiber square for Zl, variant II} and tautological for the middle term, hence true for the left-hand term by flatness. In coordinates of $T \times T$, this map is given by $t_i \mapsto t_i^{\ell}$ and $s_i \mapsto s_i^{\ell}$ for any $i = 1, \dots, n$. 

Now, the ideal sheaf $\eI_{T\sslash W}$ of the diagonal in $T\sslash W \times T\sslash W$ is generated by the elements $f_j = e_j(t_1, \dots, t_n) - e_j(s_1, \dots, s_n)$ for $j = 1, \dots, n$ (as these are affine spaces with coordinates $e_j$.) Therefore, the ideal sheaf $\eI_{Z_{\ell}}$ is generated by their images
\begin{equation*}
    \eI_{Z_{\ell}} = \langle e_j(t^{\ell}_1, \dots, t^{\ell}_n) - e_j(s^{\ell}_1, \dots, s^{\ell}_n) \mid j = 1, \dots, n \rangle.
\end{equation*}
as an ideal inside $\O_{T\sslash W \times T \sslash W}$.

The general case follows, since $Z_{\ell}^{T \sslash W_1, T \sslash W_2} \hookrightarrow T \sslash W_1 \times T \sslash W_2$ is given by the base change of $Z_{\ell}^{T \sslash W, T \sslash W} \hookrightarrow T \sslash W \times T \sslash W$ along the faithfully flat map $T \sslash W_1 \times T \sslash W_2 \to T \sslash W \times T \sslash W$.
\end{proof}

\subsection{Deformations of quasi-diagonals to normal cones}\label{section: deformations of quasi-diagonals to normal cones}

\subsubsection{Deforming to the normal cone at roots of unity.}
Assume $G=G^{\sc}$ is simply connected or $GL_n$.
Recall the general notation for deformations to normal cones and affine blowups from \S \ref{section: deformations to normal cones and affine blowups}. To simplify the manipulation of roots of unity, assume $\k$ is an algebraically closed field of $\chara \k = 0$.

\begin{definition}\label{definition: deformation to normal cone at roots of unity}
Take the scheme $T \times T \sslash W \times \Gm^{\rot}$, viewed as a family over $\Gm^{\rot} \hookrightarrow \A^1$. Consider the family of subschemes $\eZ = (Z_{\zeta} \mid \zeta \in \mu_{\infty})$ of its fibers over varying $\zeta \in \mu_{\infty} \hookrightarrow \Gm^{\rot}$. We consider the affine blowup of this family
\begin{equation*}
    \eN_{\mu_{\infty}}(\eZ, T \times T \sslash W) := \Bl^{\circ}_{\eZ}(T \times T \sslash W \times \Gm^{\rot}).
\end{equation*}
\end{definition}

In terms of deformations to normal cones, we can restate the above as follows. Start with the scheme $T \times T\sslash W$. Then $\eN_{\mu_{\infty}}(\eZ, T \times T \sslash W)$ is the family over $\Gm^{\rot}$ which is the constant family with fiber $T \times T\sslash W$ away from roots of unity and the deformation to the normal cone of $Z_{\ell} \hookrightarrow T \times T\sslash W$ on each (formal) neighborhood of each root of unity $\zeta$ of primitive order $\ell \in \N$. In other words, we have pullback squares
\begin{equation*}
    \begin{tikzcd}
      C(Z_{\zeta}, T \times T \sslash W) \arrow[d] \arrow[r]&  \eN_{\mu_{\infty}}(\eZ, T \times T \sslash W) \arrow[d] & \arrow[l] \arrow[d] T \times T \sslash W \times (\Gm^{\rot} - \mu_{\infty}) \\
      \zeta \arrow[r] & \Gm^{\rot} & \arrow[l] \Gm^{\rot} - \mu_{\infty}. 
    \end{tikzcd}
\end{equation*}
Our family $\eN_{\mu_{\infty}}(Z, T \times T \sslash W)$ is given the usual Rees constructions for the filtered algebras $(\O_{T \times T \sslash W}, \eI^{\bullet}_{Z_{\ell}})$ at each $\zeta$ glued together along $\Gm^{\rot} - \mu_{\infty}$, where all of them naturally trivialize.

\subsubsection{Explicit description of the deformation.}
Since $T \times T \sslash W$ is affine, the scheme $\eN_{\mu_{\infty}}(\eZ, T \times T \sslash W)$ is affine with ring of functions
\begin{equation*}
\O(\eN_{\mu_{\infty}}(\eZ, T \times T \sslash W))
=
\O(T \times (T \sslash W))[q^{\pm 1}][\tfrac{f}{q - \zeta} \mid (\zeta, f) \in \Xi]
\end{equation*} 
where $\Xi$ is the set of all pairs $(\zeta, f) \in \k^{\times} \times \O(T  \times (T \sslash W))$ with $\zeta^{\ell} = 1$ and $f \in \eI_{Z_{\ell}}$ for a common $\ell \in \N$.
Note that this algebra is \textit{not} finitely generated over $\k$.    
Alternatively, grouping all $\ell$-torsion roots of unity together, we can rewrite it as follows.
\begin{lemma}\label{claim: putting ell-th roots of unity together} We have
\begin{equation*}
\O(\eN_{\mu_{\infty}}(\eZ, T \times T \sslash W))
=
\O(T \times (T \sslash W))[q^{\pm 1}][\tfrac{f}{1 - q^{\ell}} \mid \ell \in \N, \ f \in \eI_{Z_{\ell}}].
\end{equation*}
\end{lemma}
\begin{proof}
Recall $\k = \overline{\k}$ of $\chara \k = 0$.
Write $(q^{\ell}-1) = (q-\zeta_0)\cdots (q-\zeta_{\ell-1})$, $\zeta_i \in \mu_{\ell}$. Clearly, each
$\tfrac{f}{\zeta -q} = \tfrac{f}{1 - q^{\ell}} \cdot \prod_{i'\neq i} (q-\zeta_{i'})$ is generated by $\tfrac{f}{1 - q^{\ell}}$.
On the other hand, all the monomials $(q-\zeta_i)$ are pairwise coprime in $\k[q^{\pm 1}]$. By the Chinese remainder theorem, we may express
$1 = \sum_{i=0}^{\ell-1} h_i \cdot \prod_{i'\neq i} (q-\zeta_{i'})$ for some $h_i \in \k[q^{\pm 1}]$, 
hence
$1 = \sum_{i=0}^{\ell-1} h_i \cdot (q^{\ell} - 1) \tfrac{1}{q - \zeta_i}.$
Consequently, $\tfrac{f}{1 - q^{\ell}} = \sum_{i=0}^{\ell-1} h_i (q^{\ell} - 1) \tfrac{f}{\zeta - q}$ is generated by the elements $\tfrac{f}{\zeta - q}$.
\end{proof}

\begin{remark}
For any $\ell \in \N$, we denote the corresponding cyclotomic polynomial by
$\sigma_{\ell}(q) \in \Z[q^{\pm 1}]$.    
Since we work over an algebraically closed field $\k$ of $\chara \k = 0$, we also have
\begin{equation*}
  \O(T \times (T \sslash W))[q^{\pm 1}][\tfrac{f}{\sigma_{\ell}(q)} \mid \ell \in \N, \ f \in \eI_{Z_{\ell}}]  
\end{equation*}
by the same token as the lemma above.
\end{remark}

\subsubsection{Fundamental Bezrukavnikov classes.}
To generate the functions on the deformation, the following distinguished elements suffice. Recall that we are assuming $G$ simply connected or $GL_n$.
\begin{definition}
We define the \textit{fundamental Bezrukavnikov classes} as
\begin{equation*}
    b_{j, \ell} := \tfrac{f_j}{1-q^{\ell}} := \tfrac{e_j(t_1^{\ell}, \dots, t_n^{\ell}) - e_j(s_1^{\ell}, \dots, s_n^{\ell})}{1 - q^{\ell}}, \qquad \forall j \in \{ 1, \dots, n\}, \forall \ell \in \N. 
\end{equation*}
\end{definition}

\begin{lemma}\label{lemma: fundamental bezrukavnikov classes}
To generate $\O(\eN_{\mu_{\infty}}(\eZ, T\sslash W \times T \sslash W))$ as an algebra over $\O(T\sslash W \times \Gm^{\rot} \times T \sslash W)$, it suffices to add the infinite set of ring generators
\begin{equation*}
    \left( b_{j, \ell} \mid \forall j \in \{ 1, \dots, n\}, \forall \ell \in \N \right). 
\end{equation*}
\end{lemma}
\begin{proof}
By definition, for each $\zeta \in \mu_{\infty}$ and every $f \in \eI_{Z_{\zeta}}$, we need to add the generator $\tfrac{f}{q-\zeta}$. For fixed $\zeta$, it is enough to only consider $b_{j, \zeta} := \tfrac{f_j}{q-\zeta}$ for $f_j = e_j(t_1^{\ell}, \dots, t_n^{\ell}) - e_j(s_1^{\ell}, \dots, s_n^{\ell})$, because $f_1, \dots, f_j$ generate $\eI_{Z_{\ell}}$ by Lemma \ref{lemma: explicit generators for the ideal sheaf of Zl}.

By the same coprimality argument as in Lemma \ref{claim: putting ell-th roots of unity together}, we may group all the $\ell$-torsion roots of unity together -- the elements $b_{j, \zeta}$ and $b_{j, \ell}$ generate the same ring.
\end{proof}

\begin{remark}
Consider the base changed variant $Z^{T\sslash W_1, T \sslash W_2}_{\ell} \hookrightarrow T\sslash W_1 \times T \sslash W_2$. The same set $(b_{j, \ell})$ then generates $\O(N_{q=\zeta}(Z^{T\sslash W_1, T \sslash W_2}_{\ell}, T\sslash W_1 \times T \sslash W_2))$ over $\O(T\sslash W_1 \times \Gm^{\rot} \times T \sslash W_2)$: this follows by base change along the flat map $T\sslash W_1 \times T \sslash W_2 \to T\sslash W \times T \sslash W$.  
\end{remark}

\subsubsection{Completions of the deformation over a root of unity.}

The following lemma is the algebraic analogue of a local comparison between equivariant $K$-theory and cohomology of fixed points.

\begin{lemma}\label{lemma: completions of the deformation}
Let $(x, \zeta) \in T \times \Gm^{\rot}$ with $\zeta \in \mu_{\infty}$. Then
\begin{equation*}
    N_{q=\zeta}(Z_{\ell}, T \times T\sslash W))^{\wedge}_{(x, \zeta)} = \coprod_{\vartheta \in |Z_{(x, \zeta)}|} N_{h=0}(\t \times_{\t \sslash W_{[x, \zeta]}} \t \sslash W_{\vartheta}, \t \times \t \sslash W_{\vartheta})^{\wedge}_{(0, 0)}
\end{equation*}
\end{lemma}
\begin{proof}
We first consider the case of $Z'_{\zeta} \hookrightarrow T \times T$. Recall we have already identified $|Z'_{\zeta}|=|\overline{\Gamma}_{\zeta}|$ and $|Z'_{(x, \zeta)}| = |\Gamma_{(\zeta, x)}| $  in Lemma \ref{lemma: Z is Gamma}.

Decompose $Z_{(x,\zeta)}'$ into its connected components ${Z'}_{(x, \zeta)}^{\vartheta}$ labeled by $\vartheta \in |Z_{(x, \ell)}'|$. We make the following identifications:
\begin{align*}
  N_{q=\zeta}(Z'_{\zeta}, T \times T))^{\wedge}_{(x, \zeta)}
  &= \coprod_{\vartheta \in |Z'_{(x, \zeta)}|} N_{q=\zeta}({Z'}_{(x, \zeta)}^{\vartheta}, T \times T)^{\wedge}_{(x, \zeta)} \\
  &= \coprod_{\vartheta \in |Z'_{(x, \zeta)}|} N_{q=1}(T \times_{T \sslash W_{[x, \zeta]}} T, T \times T)^{\wedge}_{(1, 1)} \\
  &= \coprod_{\vartheta \in |Z'_{(x, \zeta)}|} N_{h=0}(\t \times_{\t \sslash W_{[x, \zeta]}} \t, \t \times \t)^{\wedge}_{(0, 0)}.
\end{align*}

The first line holds since the deformation to the normal cone is Zariski local and compatible with connected components, so we only need to study each of ${Z'}^{\vartheta}_{(x, \zeta)} \hookrightarrow T \times T$ separately. 

Now, the irreducible components of $|Z'_{\ell}|$ are labeled by $W_{\aff} / \ell X_{\bullet}$. Each of these connected components is a copy of $T$. On the other hand, the irreducible components of its fiber $|Z'_{(x, \zeta)}|$ are labeled by the set $W_{\aff} / \Stab_{W_{\aff}}{(x, \zeta, x)}$. By \eqref{equation: stabilizer exact sequences for zeta root of unity}, Remark \ref{remark: colisions of components of Gamma}, the collisions of the connected components of $Z'_{\ell}$ over $x$ are described by the $W_{[x, \zeta]}$-torsor of finite sets
\begin{equation*}
  W_{\aff} / \ell X_{\bullet} \twoheadrightarrow W_{\aff} / (\ell X_{\bullet} \rtimes W_{[x, \zeta]}, \bullet_{\omega}  )
\end{equation*}
where $W_{[x, \zeta]}$ acts by the $\omega$-shifted right action.
The colliding copies of $T$ at $\vartheta$ are labeled by $W_{[x, \zeta]}$, giving the second identification.

The third line is the standard identification of formal neighborhood of the identity in a reductive group with its Lie algebra (also see \cite[(2.17)]{BBSV22}). This concludes the discussion for $Z'_{\ell}$.

Now, consider the action of $W$ on the left hand side, induced from the tautological action on the second copy of $T$. It translates to the right-hand side as follows. Firstly, it permutes the connected components: given a component labeled by $\vartheta \in |Z'_{(x, \zeta)}|$, its orbit under $W \acts |Z'_{(x, \zeta)}|$ identifies with $W/W_{\vartheta}$. This cuts down the labeling set to $|Z_{(x, \zeta)}|$ as in Lemma \ref{lemma: fiber of cocharacter graphs as lattice quotient}. Furthermore, fixing any such $\vartheta$, we are left with the natural action of $W_{\vartheta}$ on the closed immersion ${Z'}_{(x, \zeta)}^{\vartheta} \hookrightarrow T \times T$.

Altogether, taking the quotient by $W$ on both sides gives the desired equality
\begin{equation*}
    N_{q=\zeta}(Z_{\ell}, T \times T\sslash W))^{\wedge}_{(x, \zeta)} = \coprod_{\vartheta \in |Z_{(x, \zeta)}|} N_{h=0}(\t \times_{\t \sslash W_{[x, \zeta]}} \t \sslash W_{\vartheta}, \t \times \t \sslash W_{\vartheta})^{\wedge}_{(0, 0)}.
\end{equation*}
\end{proof}

\begin{remark}
More generally, the final step of the above prove shows that for $(x, \zeta) \in T \times \Gm^{\rot}$ with $\zeta \in \mu_{\infty}$ and any pair of subgroups $W_1, W_2 \leq W$ we have
\begin{equation*}
    N_{q=\zeta}(Z^{W_1, W_2}_{\ell}, T\sslash W_1 \times T\sslash W_2))^{\wedge}_{(x, \zeta)} = \coprod_{\vartheta \in |Z^{W_1, W_2}_{(x, \zeta)}|} N_{h=0}(\t \sslash W_1 \times_{\t \sslash W_{[x, \zeta]}} \t \sslash W_{2, \vartheta}, \t \sslash W_1 \times \t \sslash W_{2, \vartheta})^{\wedge}_{(0, 0)}.
\end{equation*}
\end{remark}

\subsubsection{Strict transforms of cocharacter graphs.}
Finally, it is quite useful to keep track of how the cocharacter graphs transform in the affine blowup.

\begin{notation}
Denote 
\begin{equation*}
 \overline{\Upsilon} := \overline{\Gamma}^{\st} \hookrightarrow \eN_{\mu_{\infty}}(\eZ, T \times T\sslash W) = \Bl^{\circ}_{\eZ}(T \times T\sslash W \times \Gm^{\rot})   
\end{equation*}
the strict transform of $\overline{\Gamma} \hookrightarrow T \times \Gm^{\rot} \times T \sslash W$.
This is a closed sub-ind-schemes by \S \ref{section: embedding of strict transform into affine blowup}.
The same construction works with other variants -- in particular, we denote ${\Upsilon} := {\Gamma}^{\st} \hookrightarrow \eN_{\mu_{\infty}}(\eZ', T \times T) = \Bl^{\circ}_{\eZ'}(T \times T \times \Gm^{\rot})$ the strict transform of $\Gamma \hookrightarrow T \times \Gm^{\rot} \times T$.
\end{notation}
We have the following situation
\begin{equation}\label{equation: Gamma and Upsilon}
\begin{tikzcd}
\Bl^{\circ}_{\eZ}(T \times \Gm^{\rot} \times T \sslash W) \arrow[d] \arrow[r, hookleftarrow] & \overline{\Upsilon}  \arrow[r, hookleftarrow] \arrow[d] & \overline{\Upsilon}_{\wt(\mu)} \arrow[d] \\
T \times \Gm^{\rot} \times T \sslash W \arrow[r, hookleftarrow] & \overline{\Gamma} \arrow[r, hookleftarrow] & \overline{\Gamma}_{\wt(\mu)}.   
\end{tikzcd}
\end{equation}
The horizontal maps are closed immersions of ind-schemes, the outer terms are actual schemes.

By \S \ref{section: strict transforms}, each fiber $\overline{\Upsilon}$ is given by the blowup of $\overline{\Gamma}$ at the ind-subscheme $\eZ$. In particular, restricting to a formal neighborhood of $(x, \zeta)$, we see that
\begin{equation*}
    \overline{\Upsilon}^{\wedge}_{(x, \zeta)} = \Bl_{Z^{\wedge}_{\zeta}} \overline{\Gamma}^{\wedge}_{(x, \zeta)}.
\end{equation*}

\begin{remark}\label{remark: Upsilon and GKM}
One can trace where $\overline{\Upsilon}$ and $\Upsilon$ go under the identification of Lemma \ref{lemma: completions of the deformation}. Firstly, it decomposes into connected components $\overline{\Upsilon}^{\wedge}_{(x, \zeta)} = \coprod_{\vartheta \in |Z_{(x, \zeta)}|} \overline{\Upsilon}^{\vartheta, \wedge}_{(x, \zeta)}$. Each of the closed immersions
\begin{equation*}
\overline{\Upsilon}^{\vartheta, \wedge}_{(x, \zeta)} \hookrightarrow N_{h=0}(\t \times_{\t \sslash W_{[x, \zeta]}} \t \sslash W_{\vartheta}, \t \times \t \sslash W_{\vartheta})^{\wedge}_{(x, \zeta)}  
\end{equation*}
then identifies as a hyperplane arrangement of copies of $(\t\times \Ga)^{\wedge}_{(0, 0)}$ inside $(\t \times \Ga \times \t \sslash W_{\vartheta})^{\wedge}_{(0, 0)}$ intersecting in the origin, labeled by those elements $\lambda \in X_{\bullet}$ which lie in the corresponding coset $\vartheta \in X_{\bullet}/\Stab_{W_{\aff}}{(x, \zeta, x)}$, see also Lemma \ref{lemma: fiber of cocharacter graphs as lattice quotient}. Depending on the genericity of $\zeta$, the precise arrangement is that of Remarks \ref{remark: GKM for cohomology of finite flag varieties} and \ref{remark: informal GKM localization}, applied to the ambient reductive group $L_{[x, \zeta]}$ with the parabolic subgroup $P_{\vartheta}$.
\end{remark}

\subsubsection{Modifications in the reductive case.}\label{subsubsection: modifications in the reductive case}
While the subschemes $Z_{\zeta} = Z_{\ell} \hookrightarrow T \times T\sslash W$ from \S \ref{section: quasi-diagonals at roots of unity} make sense for any reductive $G$, they have to be modified for our purposes. For start, Lemma \ref{lemma: explicit generators for the ideal sheaf of Zl} is specific to $T=T^{\sc}$. Related modification is needed to generalize Lemma \ref{lemma: Z is Gamma} to the setup from \S \ref{section: cocharacter graphs for reductive groups}.

In view of this, we extend Definition \ref{definition: deformation to normal cone at roots of unity} of $\eZ$ and $\Bl^{\circ}_{\eZ}(T \times \Gm^{\rot} \times T\sslash W)$ to the case of reductive $G$ as follows. 
We consider the family $\widetilde{\eZ}$ of subschemes $(T \times_{T^{\ad}} (Z^{\sc}_{\zeta}/\pi_1(G^{\ad}))) \times \pi_1(G)) \subseteq T \times T^{\sc} \sslash W \times \pi_1(G)$. In words, we take one copy of $\eZ$ for $G^{\sc}$ in each connected component, descend it along the quotient by $\pi_1(G^{\ad})$, and finally pull back to the given torus $T$ as in \S \ref{section: action of the fundamental group}, Remarks \ref{remark: identification for Gamma in semisimple case} and \ref{remark: identification for Gamma in general reductive case}.

With this definition, the affine blowup $\Bl^{\circ}_{\widetilde{\eZ}}(T \times \Gm^{\rot} \times T^{\sc} \sslash W \times \pi_1(G))$ is centered at the family of reduced subschemes underlying $\Gamma^{T, T^{\sc}}_{\zeta}$ for $\zeta \in \mu_{\infty}$ from \S \ref{section: action of the fundamental group}. Similarly for actions of other $W_1, W_2 \leq W$ on the two copies of $T$.

\subsection{A combinatorial map}\label{section: map to equivariant K-theory: combinatorics}
We construct a comparison map from the completed ring of functions on the above deformation to the GKM ring of the affine Grassmannian. In this section, we work purely combinatorially. The parallel topological meaning is explained in \S \ref{section: map to equivariant K-theory: topology}. 

Let $\k$ be a coefficient ring. To simplify the manipulation of roots of unity in our deformation, we assume $\k$ is an algebraically closed field of $\chara \k = 0$. While the discussion makes sense for split reductive $G$, we focus on the relevant case of $G=G^{\sc}$ simply connected or $GL_n$.

\subsubsection{Components of the localization map.}
We define the localization map $$\loc: \O(T \times \Gm^{\rot} \times T \sslash W) \xrightarrow{\loc} \prod_{\lambda \in X_{\bullet}} \O(T \times \Gm^{\rot})$$ componentwise as
\begin{align}\label{equation: combinatorial localization map}
   \loc_{\lambda}: \O(T \times \Gm^{\rot} \times T \sslash W) &\xrightarrow{\loc} \prod_{\lambda \in X_{\bullet}} \O(T \times \Gm^{\rot}) \xrightarrow{\pr_{\lambda}} \O(T \times \Gm^{\rot}), \\   
   \gamma_{\lambda}: f(x, \zeta, y) &\mapsto f(x, \zeta, \lambda(\zeta) \cdot x).
\end{align}
Consider the closed embedding $\overline{\Gamma}_{\lambda} \hookrightarrow T \times \Gm^{\rot} \times (T \sslash W)$ and identify $\overline{\Gamma}_{\lambda} = T \times \Gm^{\rot}$ via the first two factors. The components $\loc_{\lambda}$ are then geometrically given by restriction to the irreducible components $\overline{\Gamma}_{\lambda}$ of $\overline{\Gamma}$, namely
\begin{equation}
   \loc_{\lambda}: \O(T \times \Gm^{\rot} \times T \sslash W) \twoheadrightarrow \O(\overline{\Gamma}_{\lambda}).
\end{equation}

We can now combinatorially check that the image of $\gamma$ satisfies the GKM conditions.
\begin{lemma}\label{lemma: GKM localization map combinatorial}
The image of the map $\loc$ from \eqref{equation: combinatorial localization map} lands in the subring $\GKM^{0}_{T \times \Gm^{\rot}}(\Gr_G; \k)$.     
\end{lemma}
\begin{proof}
We directly check the GKM conditions. To simplify notation, observe that each function from the base $\O(T \times \Gm^{\rot})$ restrict the same to each $\overline{\Gamma}_{\lambda}$, hence satisfies the GKM conditions trivially. It is thus enough to check on functions $f \in \O(T\sslash W)$ from the second copy. We identify $f$ with the corresponding $W$-invariant function on $T$ via the embedding $\O(T\sslash W) \hookrightarrow \O(T)$; the restriction to $\overline{\Gamma}_{\lambda} \hookrightarrow T\sslash W$ now corresponds to restriction to $\Gamma_{\lambda, 1} \hookrightarrow T$.

Consider any affine simple reflection $s_{\beta + m\hbar}$. 
Let $\lambda, \lambda' \in X_{\bullet}$ be such that $\lambda' = s_{\beta+ m\hbar} (\lambda)$, or equivalently
\begin{equation}\label{equation: combinatorial verification eq 0}
\lambda' = s_{\beta}(\lambda) \cdot (\widecheck{\beta})^m.   
\end{equation}
Assume that $(x, \zeta) \in T \times \Gm^{\rot}$ is in the kernel of the character $\beta+m\hbar$. In other words, this means 
\begin{equation}\label{equation: combinatorial verification eq 1}
\beta(x) \cdot \zeta^m = 1,  
\end{equation}
or equivalently $\beta(x) = \zeta^{-m}$.
Consequently, we have 
\begin{equation}\label{equation: combinatorial verification eq 2}
\widecheck{\beta}(\zeta^m) = \widecheck{\beta}(\beta(x^{-1})) = s_{\beta}(x) \cdot x^{-1}. 
\end{equation}

To check the GKM conditions for $f$, we need to see that if \eqref{equation: combinatorial verification eq 0}, \eqref{equation: combinatorial verification eq 1} hold, then $\loc_{\lambda}(f) - \loc_{\lambda'}(f)$ vanishes at $(x, \zeta)$. We compute
\begin{align*}
\gamma_{\lambda}(f(y)) - \gamma_{\lambda'}(f(y)) 
&= f(\lambda(\zeta) \cdot x) - f(s_{\beta}(\lambda(\zeta)) \cdot \widecheck{\beta}(\zeta^m) \cdot x) \\ &= f(\lambda(\zeta) \cdot x) - f(s_{\beta}(\lambda(\zeta)) \cdot s_{\beta}(x) \cdot x^{-1} \cdot x) \\ 
&= f(\lambda(\zeta) \cdot x) - f(s_{\beta}(\lambda(\zeta) \cdot x)) \\
&= f(\lambda(\zeta) \cdot x) - f(\lambda(\zeta) \cdot x) \\
&= 0
\end{align*}
Here, the first line is the definition of $\gamma$ and \eqref{equation: combinatorial verification eq 0}, the second follows by \eqref{equation: combinatorial verification eq 2}, the third is clear, and the fourth holds since $f$ is $W$-invariant. 
\end{proof}

\subsubsection{Factorization through the deformation.}
We claim that the localization map $\loc$ factors through our deformation.
\begin{lemma}
There is a factorization
\begin{equation*}
\begin{tikzcd}
    \O(T \times \Gm^{\rot} \times T \sslash W) \arrow[r, "\loc"] \arrow[d, hookrightarrow]& \GKM^{0}_{T \times \Gm^{\rot}}(\Gr_G; \k) \arrow[r, hookrightarrow] &  \prod_{\lambda \in X_{\bullet}} \O(T \times \Gm^{\rot})\\
    \O(\eN_{\mu_{\infty}}(\eZ, T \times T \sslash W)) \arrow[ru, swap, dashed, "\bloc"]
\end{tikzcd}
\end{equation*}
\end{lemma}
\begin{proof}
We first check that the whole horizontal composition factors through $\O(\eN_{\mu_{\infty}}(\eZ, T \times T \sslash W))$.
To this end, we need to check that for each root of unity $\zeta$ of $\ord(\zeta) = \ell$ and for every $f \in \O(T \times \Gm^{\rot} \times T \sslash W)$ that vanishes on $Z_{\ell}$, the image $\loc(f)$ inside $\prod_{\lambda \in X_{\bullet}} \O(T \times \Gm^{\rot})$ is divisible by $(q -\zeta)$.
This can be checked componentwise -- we need to see that each $\loc_{\lambda}(f) \in \O(T \times \Gm^{\rot})$ is divisible by $(q-\zeta)$, or equivalently that $\loc_{\lambda}(f)$ vanishes after substituing $q=\zeta$.

To check this, we may identify $f$ with its image under $\O(T \times T\sslash W) \hookrightarrow \O(T \times T)$. For any $(x, \zeta)$, the image of $f$ under the localization map is computed as
\begin{equation*}
\loc_{\lambda}(f) = f(x, \lambda(\zeta) \cdot x).  
\end{equation*}
Specializing to our root of unity $\zeta$, this is the value of $f: T \times T \to \A^1$ on the $T$-valued point $(x, \lambda(\zeta) \cdot x): T \to T \times \Gm^{\rot} \times T$. Since $\zeta$ is an $\ell$-th root of unity, the point $\lambda(\zeta) \in T$ is $\ell$-torsion: we have $\lambda(\zeta)^{\ell} = 1$.
Hence
\begin{equation*}
    x^{\ell} = (\lambda(\zeta)\cdot x)^{\ell}.
\end{equation*}
In other words, $(x, \lambda(\zeta) \cdot x)$ is a $T$-valued point of the closed subscheme $Z'_{\ell} \hookrightarrow T \times T$. Since $f \in \eI_{Z'_{\ell}}$ by assumption, we deduce the desired vanishing $f(x, \lambda(\zeta) \cdot x) = 0$.

To conclude, we need to see that the image of $\O(\eN_{\mu_{\infty}}(\eZ, T \times T \sslash W))$ inside $\prod_{\lambda \in X_{\bullet}} \O(T \times \Gm^{\rot})$ satisfies the GKM conditions. Given any $f \in \eI_{Z_{\ell}'}$ and $\lambda, \lambda' \in X_{\bullet}$ with $ \lambda' = s_{\beta, m}\cdot \lambda$, we thus need to check that the element
\begin{equation*}
\loc_{\lambda}(\tfrac{f}{q - \zeta}) - \loc_{\lambda'}(\tfrac{f}{q - \zeta})
= \tfrac{\loc_{\lambda}(f) - \loc_{\lambda'}(f)}{q - \zeta} \in \O(T \times \Gm^{\rot})
\end{equation*}
is divisible by $(1 - \beta q^m)$.
We already know this is the case for any element in $\O(T \times T \sslash W)$ by Lemma \ref{lemma: GKM localization map combinatorial}. Since $(q - \zeta) \lhd \O(T \times \Gm^{\rot})$ is a prime ideal which does not contain $(1-\beta q^m)$, we deduce that the above ratio is also divisible by $(1 - \beta q^m)$ inside $\O(T \times \Gm^{\rot})$ as desired.
\end{proof}

\subsubsection{Injectivity.}\label{subsubsection: injectivity}
The injectivity of the above constructed map is almost obvious.
\begin{lemma}\label{lemma: injectivity}
The following ring homomorphisms introduced above are injective
\begin{equation*}
    \O(T \times \Gm^{\rot} \times T \sslash W) \to \O(\eN_{\mu_{\infty}}(\eZ, T \times T \sslash W))  \to \GKM^{0}_{T \times \Gm^{\rot}}(\Gr_G; \k) \to \prod_{\lambda \in X_{\bullet}} \O(T \times \Gm^{\rot}).
\end{equation*}
\end{lemma}
\begin{proof}
All four rings in question are clearly reduced (the leftmost one are functions on a reduced scheme; the second one is given by adding certain elements from its fraction field; the third one is a subring of the rightmost infinite product of polynomial rings). Since they are moreover algebras over $\O(T \times \Gm^{\rot})$, they all embed into the localizations at the generic point $\eta \in T \times \Gm^{\rot}$ and it is enough to check the injectivity there. But over $\eta$, the first map becomes an isomorphism (by construction); the second map is injective (it is given by restricting functions on a localization of the affine space $\A^n_{\eta} = \Spec(k(\eta)(T\sslash W))$ over the fraction field $k(\eta)$ to an infinite discrete subset $X_{\bullet} \subseteq \A^n_{\eta}$ that spans a full rank sublattice and hence is Zariski dense); the third map becomes an isomorphism (by construction).
\end{proof}

\subsubsection{Schubert completion.}
Consider the closed immersions \eqref{equation: Gamma and Upsilon}. Taking global functions, we get the maps $\O(\eN_{\mu_{\infty}}(\eZ, T \times T \sslash W)) \to \O(\overline{\Upsilon}) \to \O(\overline{\Upsilon}_{\wt(\mu)})$. The composites $\O(\eN_{\mu_{\infty}}(\eZ, T \times T \sslash W)) \twoheadrightarrow \O(\overline{\Upsilon}_{\wt(\mu)})$ are surjective and compatible as $\mu$ varies. Their kernels induce an exhaustive filtration by ideals $(\eI_{\leq \mu})_{\mu \in X_{\bullet}}$ on $\O(\eN_{\mu_{\infty}}(\eZ, T \times T \sslash W))$. For convenience, we also denote the quotients by $\O(\eN_{\mu_{\infty}}(\eZ, T \times T \sslash W))_{\leq \mu} := \O(\eN_{\mu_{\infty}}(\eZ, T \times T \sslash W)) / \eI_{\leq \mu} $. We obtain the diagram
\begin{equation}\label{equation: schubert completion square}
\begin{tikzcd}
\O(\eN_{\mu_{\infty}}(\eZ, T \times T \sslash W)) \arrow[r, hook] \arrow[d, two heads] & \O(\overline{\Upsilon})  \arrow[d, two heads] \\
\O(\eN_{\mu_{\infty}}(\eZ, T \times T \sslash W))_{\leq \mu} \arrow[r, equal] & \O(\overline{\Upsilon}_{\wt(\mu)}).
\end{tikzcd}
\end{equation}
Here, the upper horizontal map is injective: restricting to irreducible components of $\overline{\Upsilon}$, the ring $\O(\overline{\Upsilon})$ embeds into $\prod_{\lambda \in X_{\bullet}} \O(T \times \Gm^{\rot})$; the induced map from $\O(\eN_{\mu_{\infty}}(\eZ, T \times T \sslash W))$ is precisely the injection from Lemma \ref{lemma: injectivity}. The filtration $(\eI_{\leq \mu})_{\mu \in X_{\bullet}}$ is thus separated.

\begin{definition}\label{definition: Schubert completion}
    The \textit{Schubert completion} of $\O(\eN_{\mu_{\infty}}(\eZ, T \times T \sslash W))$ is the completion with respect to the above separated exhaustive filtration. It is given by
    \begin{equation*}
        \widehat{\O}(\eN_{\mu_{\infty}}(\eZ, T \times T \sslash W)) := \lim_{\mu} {\O}(\eN_{\mu_{\infty}}(\eZ, T \times T \sslash W))_{\leq \mu}.
    \end{equation*}
    Since the filtration is separated, we have a natural injection
    \begin{equation*}
     \O(\eN_{\mu_{\infty}}(\eZ, T \times T \sslash W)) \hookrightarrow  \widehat{\O}(\eN_{\mu_{\infty}}(\eZ, T \times T \sslash W)).
    \end{equation*}
\end{definition}

With the above definitions, the following statement is clear.
\begin{lemma}\label{lemma: completed deformation as functions on Upsilon}
We have a ring isomorphism
    \begin{equation*}
        \widehat{\O}(\eN_{\mu_{\infty}}(\eZ, T \times T \sslash W)) = \O(\overline{\Upsilon}).
    \end{equation*}
\end{lemma}
\begin{proof}
Apply $\lim_{\mu}(-)$ to the bottom row in \eqref{equation: schubert completion square} varying in $\mu$. By definition of the Schubert completion, $\widehat{\O}(\eN_{\mu_{\infty}}(\eZ, T \times T \sslash W)) = \lim_{\mu} \O(\eN_{\mu_{\infty}}(\eZ, T \times T \sslash W))_{\leq \mu}$. Computing functions from an ind-presentation of $\overline{\Upsilon}$, also $\O(\overline{\Upsilon}) = \lim_{\mu} \O(\overline{\Upsilon}_{\wt(\mu)})$.
\end{proof}

We also easily deduce the following.
\begin{lemma}\label{lemma: completed comparison}
The above induces an injective ring homomorphism
\begin{equation}\label{equation: completed comparison}
    \widehat{\O}(\eN_{\mu_{\infty}}(\eZ, T \times T \sslash W)) \hookrightarrow \GKM^{0}_{T \times \Gm^{\rot}}(\Gr_G; \k).
\end{equation}
\end{lemma}
\begin{proof}
To construct the map, observe that $\GKM^{0}_{T \times \Gm^{\rot}}(\Gr_G; \k)$ is complete with respect to the filtration given by kernels of $\GKM^{0}_{T \times \Gm^{\rot}}(\Gr_G; \k) \to \GKM^{0}_{T \times \Gm^{\rot}}(\Gr_{\leq \mu}; \k)$. Under the injective homomorphism ${\O}(\eN_{\mu_{\infty}}(\eZ, T \times T \sslash W)) \hookrightarrow \GKM^{0}_{T \times \Gm^{\rot}}(\Gr_G; \k)$ from Lemma \ref{lemma: injectivity}, this matches the Schubert filtration on the source. Consequently, it induces a map \eqref{equation: completed comparison}. This map is injective as both sided embed compatibly into $\prod_{\lambda \in X_{\bullet}} \O(T \times \Gm^{\rot})$.
\end{proof}

\subsection{The topological meaning}\label{section: map to equivariant K-theory: topology}
We now explain the topological meaning of the map from \S \ref{section: map to equivariant K-theory: combinatorics}. This ties to the previous section by localization to fixed points. Ignoring this relationship, the topological arguments are completely independent.
In particular, we interpret the map $\loc$ and fundamental Bezrukavnikov classes in purely topological terms. This interpretation makes sense over any coefficient ring $\k$, for example over $\Z$.

\subsubsection{The map from functions on tori.}\label{section: the map from functions on tori}
We work with a reductive group $G$.
To allows for topological arguments, the geometry of $\Gr_G$ takes place over $k=\C$. However, the coefficient ring $\k$ may be taken arbitrary.

There is a natural map
\begin{equation*}
 K^{\top}_{T \times \Gm^{\rot} \times G}(\pt) \xrightarrow{} K^{\top}_{T \times \Gm^{\rot}}(\Gr_G)  
\end{equation*}
defined as the equivariant structure map under the vertical identifications
\begin{equation*}
    \begin{tikzcd}[row sep=5]
      K^{\top}_{T \times \Gm^{\rot} \times G}(\pt) \arrow[r] & K^{\top}_{T \times \Gm^{\rot}}(\Gr_G) \\
      & K^{\top}((T \times \Gm^{\rot}) \backslash \lo G / \lop G) \arrow[u, equal] \\
      K^{\top}((T \times \Gm^{\rot}) \backslash \pt / G) \arrow[uu, equal] \arrow[r] & K^{\top}((T \times \Gm^{\rot}) \backslash \lo G / G) \arrow[u, equal].
    \end{tikzcd}
\end{equation*}

\begin{remark}
More generally, given two parahoric subgroups $\eP$ and $\eP'$ of $\lo G$ with associated Levi quotients $L$ and $L'$, we get the map
\begin{equation*}
 K^{\top}_{L \times \Gm^{\rot} \times L'}(\pt) \to K^{\top}_{\eP \rtimes \Gm^{\rot}}(\lo G/ \eP').  
\end{equation*}    
\end{remark}

\subsubsection{Localization map.}
We can now identify
\begin{equation}\label{equation: localization map}
   \begin{tikzcd}[row sep=5]
        K^{\top, 0}_{T \times \Gm^{\rot} \times G}(\pt; \k) \arrow[r] & K^{\top, 0}_{T \times \Gm^{\rot}}(\Gr_G; \k) \arrow[r, hookrightarrow] & K^{\top, 0}_{T \times \Gm^{\rot}}(\Gr^{T \times \Gm^{\rot}}; \k) \\
        \O(T \times \Gm^{\rot} \times T \sslash W) \arrow[r] \arrow[u, equal] & \GKM^{0}_{T \times \Gm^{\rot}}(\Gr_G; \k) \arrow[u, equal] \arrow[r, hookrightarrow] & \prod_{\lambda \in X_{\bullet}} \O(T \times \Gm^{\rot}) \arrow[u, equal].
    \end{tikzcd}
\end{equation} 
In particular, this gives a canonical map
\begin{equation}\label{equation: map from two copies of tori}
 \O(T \times \Gm^{\rot} \times T \sslash W) \to K^{\top, 0}_{T \times \Gm^{\rot} \times G}(\Gr_G; \k)
\end{equation}
Under the identification from \eqref{equation: localization map}, the left-right composition is the localization map $\gamma$ from \eqref{equation: combinatorial localization map}.

\subsubsection{Tautological bundles on $\Gr$.}
Let us first take $G=GL_n$. Let $\Lambda_0 = k\llb t \rrb^{\oplus n} \subseteq k\llp t \rrp^{\oplus n}$ be the standard lattice. Fix some $\mu \in X_{\bullet}$ and let $\Gr_{\leq \mu}$ be the corresponding affine Schubert variety. It parametrizes certain sublattices $\Lambda \subseteq k\llp t \rrp^{\oplus n}$. Depending on $\mu$, we choose a sufficiently big positive integer $m \geq 0$ such that for each $\Lambda \in \Gr_{\leq \mu}$ it holds that $t^m \Lambda_0 \subseteq \Lambda \subseteq t^{-m} \Lambda_0$.

After fixing this normalization, we have the tautological vector bundle $\eE_{\leq \mu}$ on $\Gr_{\leq \mu}$, associating to an $R$-valued points $\Lambda$ of $\Gr_{\leq \mu}$ the algebraic vector bundle $t^{-m}\Lambda_0 / \Lambda$ on $R$:
\begin{equation*}
    \eE_{\leq \mu}: \Lambda \mapsto (\coker \Lambda \to t^{-m}\Lambda_0).
\end{equation*}
This is indeed a vector bundle, see \cite[proof of Theorem 1.1.3]{Zhu15}.
Similarly, we have the trivial rank $mn$ vector bundle $\eG_{\leq \mu}$ given by 
$\coker (\Lambda_0 \to t^{-m}\Lambda_0)$.
Both of these are naturally equivariant with respect to the groups $T \times \Gm^{\rot} \leq \lop GL_n \rtimes \Gm^{\rot}$.

\subsubsection{Bezrukavnikov classes for $GL_n$ from tautological bundles.}\label{section: bezrukavnikov classes for GL_n from tautological bundles}
Let $G=GL_n$. Consider the equivariant $K$-theory class given by the difference of the above vector bundles
\begin{equation}\label{equation: normalized tautological class 1}
[\eE_{\leq \mu}] - [\eG_{\leq \mu}] = [t^{-m}\Lambda_0 / \Lambda] - [t^{-m}\Lambda_0 / \Lambda_0] \in K^{\top, 0}_{T \times \Gm^{\rot}}(\Gr_{\leq \mu}; \k)    
\end{equation}
This can be rewritten in several ways. Firstly, it is given by the perfect complex
\begin{equation}\label{equation: normalized tautological class 2}
[\Lambda \to t^{-m}\Lambda_0] + [\Lambda_0 \to t^{-m}\Lambda_0][1]
\end{equation}
on $\Gr_{\leq \mu}$; here the $t^{-m}\Lambda_0$ sits in degree zero and $[1]$ is a homological shift. 
Alternatively, it can be also rewritten as a difference of classes of equivariant coherent sheaves which lie in $\Perf(\Gr_{\leq \mu})$, namely
\begin{equation}\label{equation: normalized tautological class 3}
    [\Lambda_0 / \Lambda_0 \cap \Lambda] - [\Lambda / \Lambda_0 \cap \Lambda].
\end{equation}
Informally, it may be thought of as the difference $[\Lambda_0] - [\Lambda]$, but this does not quite make sense since neither of these projective modules is of finite rank.

In any case, the class \eqref{equation: normalized tautological class 1} is clearly independent of the choice of $m$, and compatible with varying $\mu$. (This is obvious from \eqref{equation: normalized tautological class 3}; it alternatively follows by $[t^{-m}\Lambda_0 / \Lambda] - [t^{-m}\Lambda_0 / \Lambda_0] = [t^{-m}\Lambda_0 / \Lambda] + [t^{-m'}\Lambda_0 / t^{-m}\Lambda_0 ] - [t^{-m'}\Lambda_0 / t^{-m}\Lambda_0 ] - [t^{-m}\Lambda_0 / \Lambda_0] = [t^{-m'}\Lambda_0 / \Lambda] - [t^{-m'}\Lambda_0 / \Lambda_0]$ for any $m' \geq m$.)  It thus gives a well defined equivariant $K$-theory class on the whole of $\Gr_G$. We then have the following.

\begin{lemma}
Inside $K^{\top, 0}_{T \times \Gm^{\rot}}(\Gr_G; \k)$, we have $b_{1,1} = [t^{-m}\Lambda_0 / \Lambda] - [t^{-m}\Lambda_0 / \Lambda_0]$.
\end{lemma}
\begin{proof}
Given any $\lambda \in X_{\bullet}$, it is easy to compute the localizations of both sides to the corresponding fixed-point. Denote $\lambda = (t_1^{d_1}, \dots, t_n^{d_n})$ for $d_1, \dots, d_n \in \Z$.
%
For any $d \in \Z$, note that
\begin{equation}
[d]_q := \frac{1-q^d}{1-q} =   
\begin{cases}\label{equation: cases for geometric series}
1 + q + \dots + q^{d-1} & \text{if $d > 0$,} \\
0 & \text{if $d = 0$,} \\
-q^{-1}( 1 + q^{-1} + \dots + q^{d+1}) & \text{if $d < 0$.}
\end{cases}
\end{equation}
We now compute
\begin{align*}
\loc_{\lambda}([t^{-m}\Lambda_0 / \Lambda] - [t^{-m}\Lambda_0 / \Lambda_0])
&= \loc_{\lambda}([\Lambda_0 / \Lambda_0 \cap \Lambda] - [\Lambda / \Lambda_0 \cap \Lambda]) \\
&= \sum_{\{i \mid d_i > 0\}} (1+q+ \dots + q^{d_i-1})t_i - \sum_{\{i \mid d_i < 0\}} (q^{-1}+ \dots + q^{d_i+1})t_i \\
&= \sum_{i=1}^n \frac{1-q^{d_1}}{1-q} \cdot t_i \\
&= \frac{\sum_{i=1}^n t_i - \sum_{i=1}^n q^{d_i} \cdot t_i}{1-q} \\
& =\gamma_{\lambda} \left( \frac{e_1(t_1, \dots, t_n) - e_1(s_1, \dots, s_n)}{1-q} \right)\\
&=\gamma_{\lambda}(b_{1,1}) 
\end{align*}
Here, the first line is \eqref{equation: normalized tautological class 3}. The second line is by concrete knowledge of the fibers of $t^{-m}\Lambda_0 / \Lambda$ and $t^{-m}\Lambda_0 / \Lambda$ at the fixed point $\lambda$ as $T \times \Gm^{\rot}$ representations. The third line is by \eqref{equation: cases for geometric series}. The fourth line is clear, the fifth line follows by the description of $\gamma_{\lambda}$ from \eqref{equation: combinatorial localization map}, and the sixth line is by definition of $b_{1,1}$.

Since the above holds for any $\lambda \in X_{\bullet}$, the lemma follows.
\end{proof}

\subsubsection{Bezrukavnikov classes from vector bundles for general $G$.}
Let $G$ be any reductive group.
The affine Grassmannian is functorial in $G$, as is our description of its equivariant $K$-theory. Consequently, we can define the Bezrukavnikov classes from the standard one on $\Gr_{GL_n}$.

Namely, consider any representation $\rho: G \to GL_n$. This induces a map $\Gr_G \to \Gr_{GL_n}$, equivariant with respect to $\lo G \rtimes \Gm^{\rot} \to \lo GL_n \rtimes \Gm^{\rot}$ and compatible with the Schubert stratifications.
We thus have a pullback map $\rho^*: K^{\top, 0}_{T_{GL_n} \times \Gm^{\rot}}(\Gr_{GL_n}; \k) \to K_{T_{G} \times \Gm^{\rot}}^{\top, 0}(\Gr_G; \k)$. Using the double-quotient description \S \ref{section: local quantum hecke stacks} and the construction \S \ref{section: the map from functions on tori}, it follows that $\rho^*$ induces the obvious map $\rho^*: T_{GL_n} \times \Gm^{\rot} \times T_{GL_n} \to T_{G} \times \Gm^{\rot} \times T_{G}$, and hence on the deformation constructions.

In particular, for any $f \in \eI_{Z_{\ell}} \subseteq \O(T_{GL_n} \times T_{GL_n} \sslash W)$, this map sends $\tfrac{f}{1-q^{\ell}} \mapsto \tfrac{\rho^*(f)}{1-q^{\ell}}$. 
Specializing further to the case when $G=G^{\sc}$ is simply connected, $\rho = \omega_j$ is the $j$-th fundamental representation and $f=e_1(t_1, \dots, t_n) - e_1(s_1, \dots, s_n)$, we deduce it sends $b^{GL_n}_{1,1} \mapsto b^G_{j, 1}$. In other words, $b^G_{j, 1}$ is the equivariant $K$-theory class $\omega_j^*([\Lambda_0 / \Lambda_0 \cap \Lambda] - [\Lambda / \Lambda_0 \cap \Lambda] )$ on $\Gr_{G}$. We have proved the following.

\begin{lemma}\label{lemma: fundamental bezrukavnikov classes as pullbacks} We have
 \begin{equation*}
    b^G_{j,1}= \omega_j^{*} (b^{GL_n}_{1,1}). 
 \end{equation*}   
\end{lemma}

\subsubsection{Adams operations.}
The fundamental Bezrukavnikov classes 
$b_{j, \ell} = \tfrac{e_j(t^{\ell}_1, \dots, t_n^{\ell}) - e_j(s^{\ell}_1, \dots, s_n^{\ell})}{1-q^{\ell}}$
are related to each other through Adams operations in equivariant $K$-theory. 
\begin{lemma}\label{lemma: adams operations on bezrukavnikov classes}
Let $1 \leq j \leq n$. For each $\ell, \ell' \in \N$, we have
\begin{equation}\label{equation: adams operations on bezrukavnikov classes}
    \psi^{\ell}(b_{j, \ell'}) = b_{j, \ell\ell'}.
\end{equation}    
\end{lemma}
\begin{proof}
The Adams operation $\psi^{\ell}$ is a functorial ring endomorphism of $K^{\top, 0}_{T \times \Gm^{\rot}}(-)$, which exponentiate classes of line bundles to the $\ell$-th power.
By injectivity of localization to fixed-points (valid in our setup), it is thus enough to check the equality \eqref{equation: adams operations on bezrukavnikov classes} after restriction to each fixed point, in other words after applying $\gamma_{\lambda}$. 

After this restriction, each of the generators $t_i, s_i, q, 1$ is a $K$-theory class of an equivariant line bundle (a character of $T \times \Gm^{\rot}$). Therefore, $\psi^{\ell}$ acts by $(-)^{\ell}$ on each of them. Since $\psi^{\ell}$ is a ring homomorphism, the desired equality follows.  
\end{proof}

\subsubsection{Topological argument.}
With the above topological meaning at hand, we have the following clear construction, which we again state for $G$ simply connected or $GL_n$.
\begin{lemma}
    We have ring homomorphisms
    \begin{equation*}
        \O(T \times \Gm^{\rot} \times T \sslash W) \to \O(\eN_{\mu_{\infty}}(\eZ, T \times T \sslash W)) \to \widehat{\O}(\eN_{\mu_{\infty}}(\eZ, T \times T \sslash W)) \to K^{\top, 0}_{T \times \Gm^{\rot}}(\Gr_G; \k).
    \end{equation*}
\end{lemma}
\begin{proof}
The map from the leftmost term is given by the equivariant structure maps \S \ref{section: the map from functions on tori}. To see that it factors through the second term, we need to prove that the fundamental Bezrukavnikov classes $b_{j, \ell}$ give well-defined classes in $K$-theory. However, we have already interpreted them as Adams operations on the fundamental vector bundles in Lemmata \ref{lemma: fundamental bezrukavnikov classes as pullbacks}, \ref{lemma: adams operations on bezrukavnikov classes}.

Finally, $K^{\top, 0}_{T \times \Gm^{\rot}}(\Gr_G; \k)$ is complete with respect to the filtration given by kernels of restriction maps to affine Schubert varieties by Proposition \ref{proposition: completion of K-theory}. Equipping $\O(\eN_{\mu_{\infty}}(\eZ, T \times T \sslash W))$ with the Schubert filtration of Definition \ref{definition: Schubert completion} makes the comparison map filtered. Hence it factors through the Schubert completion $\widehat{\O}(\eN_{\mu_{\infty}}(\eZ, T \times T \sslash W))$ as desired.
\end{proof}

\subsection{Bezrukavnikov--Finkelberg description of equivariant \texorpdfstring{$K$}{K}-theory}\label{section: bezrukavnikov--finkelberg description of equivariant K-theory}
We now conclude the proof of our main result for simply connected $G$ with $\C$-coefficients. To this end, we will assume $k = \C = \k$. We further record the topological Adams generation. We find this case particularly illustrative; technical generalizations are then discusses in \S \ref{section: variants and generalizations}. 

\subsubsection{The isomorphism.}
We now prove the surjectivity by a topological argument.
\begin{theorem}\label{theorem: equivariant K-theory as deformation to normal cone -- G sssc}
Assume $G$ is simply-connected group over $\C$. Then the above constructed map induces a ring isomorphism
\begin{equation*}
   \widehat{\O}(\eN_{\mu_{\infty}}(\eZ, T \times T \sslash W)) \xrightarrow{\cong} K^{\top, 0}_{T \times \Gm^{\rot}}(\Gr_G; \C).
\end{equation*}
\end{theorem}
\begin{proof}
Regard both sides as quasi-coherent (so in particular fpqc) sheaves on $T \times \Gm^{\rot}$. To see that the discussed map is an isomorphism, it is thus enough to check it is an isomorphism after formal completion at each closed point $(x, \zeta) \in T \times \Gm^{\rot}$. (Alternatively, it is enough to check the isomorphisms at actual fibers by Remark \ref{remark: replacing formally completed fibers by actual fibers}.) We have thus reduced to checking that for each $(x, \zeta)$, the bottom horizontal map in the following diagram is an isomorphism.
\begin{equation*}
    \begin{tikzcd}
         \widehat{\O}(\eN_{q=\zeta}(\eZ, T \times T \sslash W)) \arrow[r] \arrow[d] & K^{\top, 0}_{T \times \Gm^{\rot}}(\Gr_G; \C) \arrow[d] \\
         \widehat{\O}(\eN_{q=\zeta}(\eZ, T \times T \sslash W))^{\wedge}_{(x, \zeta)} \arrow[r] & K^{\top, 0}_{T \times \Gm^{\rot}}(\Gr_G; \C)^{\wedge}_{(x, \zeta)}
    \end{tikzcd}
\end{equation*}
To this end, fix a closed point $(x, \zeta) \in T \times \Gm^{\rot}$. To check the isomorphism between the completions, we distinguish cases depending on the genericity of $\zeta$.

\bigskip
Assume first that $\zeta$ is a primitive root of unity of order $\ell$. Then consider the following diagram.
\begin{equation*}
    \begin{tikzcd}[row sep = 5, column sep =15]
      \widehat{\O}({\eN}_{\mu_{\infty}}(\eZ, T \times T \sslash W))^{\wedge}_{(x, \zeta)} \arrow[d, equal] \arrow[r] & K^{\top, 0}_{T \times \Gm^{\rot}}(\Gr_{G}; \C)^{\wedge}_{(x, \zeta)} \arrow[d, equal] \\
      \widehat{\O}({N}_{q=\zeta}(Z_{\ell}, T \times T \sslash W))^{\wedge}_{(x, \zeta)} \arrow[d, equal]  & \widehat{H}^{\bullet}_{T \times \Gm^{\rot}}(\Fix_{(x, \zeta)}(\Gr_G); \C)^{\wedge}_{(0, 0)} \arrow[d, equal] \\
      \prod_{\vartheta \in |\overline{\Gamma}_{(x, \zeta)}|} \widehat{\O}(N_{q=\zeta}({Z_{\ell}^{\vartheta}}, T \times T\sslash W))^{\wedge}_{(x, \zeta)} \arrow[d, equal] & \widehat{H}^{\bullet}_{T \times \Gm^{\rot}}(\coprod_{\vartheta \in |\overline{\Gamma}_{(x, \zeta)}|} \Fl^{\vartheta}_{\ell, L^{\sc}_{[x, \zeta]}}; \C)^{\wedge}_{(0, 0)} \arrow[d, equal] \\
      \prod_{\vartheta \in |\overline{\Gamma}_{(x, \zeta)}|} \widehat{\O}(N_{h=0}(\t \times_{\t \sslash W_{[x, \zeta]}} \t \sslash W_{\vartheta},  \t \times \t \sslash W_{\vartheta} ))^{\wedge}_{(0,0)} \arrow[r, equal] & \prod_{\vartheta \in |\overline{\Gamma}_{(x, \zeta)}|} \widehat{H}^{\bullet}_{T \times \Gm^{\rot}}(\Fl^{\vartheta}_{\ell, L^{\sc}_{[x, \zeta]}}; \C)^{\wedge}_{(0, 0)}
    \end{tikzcd}
\end{equation*}
The identifications in the left-hand column work as follows. First, since we are completing at $(x, \zeta)$, only the deformation to the normal cone of $Z_{\ell}$ centered at $\zeta$ matters. The next two identifications are covered by Lemma \ref{lemma: completions of the deformation}.

The identifications in the right-hand column go a follows. First use the localization from Corollary \ref{corollary: K-theory in the flat situation}, Theorem \ref{theorem: fibers of topological K-theory, HP and fixed-points} and Remark \ref{remark: the completed variant of equivariant localization} to pass from the completion of equivariant $K$-theory over $(x, \zeta)$ to the completion of equivariant cohomology of the fixed-points of $(x, \zeta)$ at the origin (alternatively see Remark \ref{remark: replacing formally completed fibers by actual fibers}). Then use the computation of the reduced fixed-points on the affine Grassmannian from Theorem \ref{lemma: fixed points of (c, zeta) on Gr_G -- zeta root of unity}. Finally, we may commute the disjoint union through cohomology.

The bottom identification is by taking the Schubert completion of the isomorphism from Proposition \ref{proposition: cohomology of affine flag varieties as deformation to the normal cone}. (Note that the completion of equivariant cohomology is given by the Schubert completion.) This finishes the discussion of the case when $\zeta$ is a root of unity.

\bigskip
Now, assume on the other hand that $\zeta \in \Gm^{\rot} - \mu_{\infty}$. We then make the analogous identifications.
\begin{equation*}
    \begin{tikzcd}[row sep = 5, column sep =15]
      \widehat{\O}({\eN}_{\mu_{\infty}}(\eZ, T \times T \sslash W))^{\wedge}_{(x, \zeta)} \arrow[d, equal] \arrow[r] & K^{\top, 0}_{T \times \Gm^{\rot}}(\Gr_{G}; \C)^{\wedge}_{(x, \zeta)} \arrow[d, equal] \\
      \widehat{\O}(T \times \Gm^{\rot} \times T \sslash W)^{\wedge}_{(x, \zeta)} \arrow[d, equal]  & \widehat{H}^{\bullet}_{T \times \Gm^{\rot}}(\Fix_{(x, \zeta)}(\Gr_G); \C)^{\wedge}_{(0, 0)} \arrow[d, equal] \\
      \prod_{\vartheta \in |\overline{\Gamma}_{(x, \zeta)}|} {\O}(\Gm^{\rot} \times T \times_{T \sslash W_{[x, \zeta]}} T \sslash W_{\vartheta} )^{\wedge}_{(x, \zeta)} \arrow[d, equal] & {H}^{\bullet}_{T \times \Gm^{\rot}}(\coprod_{\vartheta \in |\overline{\Gamma}_{(x, \zeta)}|} L_{[x, \zeta]}/P_{\vartheta}; \C)^{\wedge}_{(0, 0)} \arrow[d, equal] \\
      \prod_{\vartheta \in |\overline{\Gamma}_{(x, \zeta)}|} {\O}(\Ga \times\t \times_{\t \sslash W_{[x, \zeta]}} \t \sslash W_{\vartheta})^{\wedge}_{(0,0)} \arrow[r, equal] & \prod_{\vartheta \in |\overline{\Gamma}_{(x, \zeta)}|} {H}^{\bullet}_{T \times \Gm^{\rot}}(L_{[x, \zeta]}/P_{\vartheta}; \C)^{\wedge}_{(0, 0)}
    \end{tikzcd}
\end{equation*}

The identifications in the left-hand column are as follows. Firstly, since $\zeta \notin \mu_{\infty}$, after the completion at $(x, \zeta)$ the deformations to the normal cones do not play a role. For the second identification, note that the Schubert completion corresponds to a limit over restrictions to finite subsets of the component set $|\Gamma_{(x, \zeta)}|$. Each of the connected components $\Gamma^{\vartheta}_{(x, \zeta)}$, $\vartheta \in |\Gamma_{(x, \zeta)}|$ is given by $W_{[x, \zeta]}$ copies of $T$ intersection in the natural way; the stabilizer of this subscheme with respect to the $W$-action on the right copy of $T$ is precisely $W_{\vartheta}$. The third identification is the usual comparison of the completion of a semisimple, simply-connected group with its Lie algebra.

For the right hand column, the first identification is the localization of equivariant $K$-theory as in Corollary \ref{corollary: K-theory in the flat situation}, Theorem \ref{theorem: fibers of topological K-theory, HP and fixed-points} and Remark \ref{remark: the completed variant of equivariant localization}. The second isomorphism is the computation of reduced fixed points from Theorem \ref{lemma: fixed points of (c, zeta) on Gr_G -- zeta generic} (no completion is necessary, since each of the terms has bounded cohomology). The third identification is clear.

The bottom horizontal identification is the well-known description of $T$-equivariant cohomology of partial flag varieties of reductive groups from Lemma \ref{lemma: equivariant cohomology of finite flag varieties} and Remark \ref{remark: finite flag varieties with trivial loop rotation}; note that $\Gm^{\rot}$ acts trivially on each of them.
\end{proof}

\begin{remark}\label{remark: replacing formally completed fibers by actual fibers} 
In fact, in the first reduction step, we may reduce further.
Firstly, the map in question is injective by \S \ref{subsubsection: injectivity}, so we only need to check surjectivity. 

Secondly, since we are speaking about a map of quasi-coherent sheaves on an affine scheme, it is enough to check that for each closed point $(x, \zeta)$, the following map between the fibers is surjective:
\begin{equation*}
    \begin{tikzcd}
        \widehat{\O}(\eN_{q=\zeta}(\eZ, T \times T \sslash W))_{(x, \zeta)} \arrow[r] & K^{\top, 0}_{T \times \Gm^{\rot}}(\Gr_G; \C)_{(x, \zeta)}.
    \end{tikzcd}
\end{equation*}    
Consequently, all occurrences of the formal completion $(-)^{\wedge}_{(x, \zeta)}$ may be replaced by the actual fiber $(-)_{(x, \zeta)}$ throughout the proof. This variant of the proof thus needs a slightly weaker technical input from \S \ref{section: K-theory and HP and cohomology of fixed points}, which can be covered by more standard results.
\end{remark}

\begin{remark}
The horizontal identifications in the proof may be also compared through the GKM presentations, as outlined in Remarks \ref{remark: GKM for cohomology of finite flag varieties}, \ref{remark: informal GKM localization}, \ref{remark: Upsilon and GKM}. This is another sanity check for the compatibility of all the identifications that we made.

Since the GKM presentation works over more general coefficient rings $\k$, this variant of the argument is applicable. In particular, the variant of Theorem \ref{theorem: equivariant K-theory as deformation to normal cone -- G sssc} with coefficients in $\k = \Q$ holds without without modifications. On the other hand, for $\k=\Z$, further blowups over all primes $p \in \Z$ are necessary; we will discuss this elsewhere.
\end{remark}

\subsubsection{Finite generation as complete topological $\lambda$-ring.}
As we saw, the equivariant $K$-theory ring of the affine Grassmannian is quite big -- it is a topological completion of an already non-noetherian ring of functions on our affine blowup. Nevertheless, it is finitely generated in the following appropriate sense.
\begin{corollary}\label{corollary: topological Adams generation}
Let $G$ be a simply connected group over $\C$. As complete topological $\lambda$-algebra over the equivariant base, $K^{\top, 0}_{T \times \Gm^{\rot}}(\Gr_G; \C)$ is generated by the finitely many elements
    \begin{equation*}
        (b_{j, 1} \mid j = 1, \dots, n).
    \end{equation*}
\end{corollary}
\begin{proof}
By Theorem \ref{theorem: equivariant K-theory as deformation to normal cone -- G sssc} and Lemma \ref{lemma: fundamental bezrukavnikov classes}, we know that $K^{\top, 0}_{T \times \Gm^{\rot}}(\Gr_G; \C)$ is topologically generated by the fundamental classes $b_{j, \ell}$. Under $\lambda$-operations, these are in turn generated by $b_{j, 1}$ for $j=1, \dots, n$ by Lemma \ref{lemma: adams operations on bezrukavnikov classes}.
\end{proof}

\subsubsection{Some intuition.}
One could try to approach the surjectivity in the proof of Theorem \ref{theorem: equivariant K-theory as deformation to normal cone -- G sssc} in another way as follows. By Lemma \ref{lemma: completed deformation as functions on Upsilon} and Proposition \ref{proposition: GKM description of K-theory of Gr} we want to prove that the following map is an isomorphism
\begin{equation*}
   \O(\overline{\Upsilon}) \to \GKM^{0}_{T \times \Gm^{\rot}}(\Gr_G; \k). 
\end{equation*}
Both sides compatibly embed into
\begin{equation}\label{equation: product of copies of torus}
    \prod_{\lambda \in X_{\bullet}} \O(T \times \Gm^{\rot}).
\end{equation}
and we are left to check that the images agree. For the right hand side, the image is given by the GKM conditions, saying that functions on different copies of $T \times \Gm^{\rot}$ agree over subtori of the form $T_{\beta + m\hbar} = \ker(\beta + m\hbar)$.
For the left-hand side, we are considering functions on an arrangement of copies of $T \times \Gm^{\rot}$ inside $T \times \Gm^{\rot} \times T\sslash W$. Their intersections combinatorially match the subvarieties imposing the equalizer conditions in the GKM picture.

From what we already said, one would be tempted to conclude. However, this is not a complete reasoning, as Example \ref{example: seminormality issues} below shows. To complete the above intuition, one needs to, in one way or another, study precisely how $\overline{\Upsilon}$ sits inside $\Bl^{\circ}_{\eZ}( T \times \Gm^{\rot} \times T \sslash W)$ via its defining embedding -- this is what the above proof accomplishes.

\begin{example}\label{example: seminormality issues}
 Consider the following example. Let $X_3 := \Spec \k[x,y, z] / (xy, yz, zx)$ be three lines in general position inside the three-space $\A^3$ intersecting in the origin. Let $X_2 := \Spec \k[x, y] / xy(x-y)$ be three lines inside the plane $\A^2$ intersection in the origin.  There is a natural squishing map $X_3 \to X_2$. This map is a homeomorphism -- in fact the same arrangement of lines. However, it is not an isomorphism of schemes: the induced injection on the rings of functions $\O(X_2) \to \O(X_3)$ is not surjective. The problem is related to the notion of \textit{semi-normality} ($X_3$ is seminormal but $X_2$ is not), or even more explicitly to that of \textit{embedding dimension} (which is $3$ and $2$, respectively).   
\end{example}

\subsection{Variants and generalizations}\label{section: variants and generalizations}

\subsubsection{The general answer.}
We work with coefficients $\k=\C$. As stated in Remark \ref{remark: variant of the K-theory of affine Grassmannian}, our Theorem \ref{theorem: equivariant K-theory as deformation to normal cone -- G sssc} has several direct generalizations. These may be either proved by a direct modifications of the above argument, or bootstrapped from what we already have. Putting these remarks together, the generalized statement is as follows. Recall that $G^{\sc}$ denotes the simply connected cover of the derived group of $G$ from \S \ref{subsubsection: simply connected covers and adjoint groups}. We use the modified family of blowup centers $\widetilde{\eZ}$ from \S \ref{subsubsection: modifications in the reductive case}, \S \ref{section: action of the fundamental group}.

\begin{theorem}\label{theorem: equivariant K-theory as deformation to normal cone -- general case}
Let $G$ be a reductive group over $\C$. Denote by $T^{\sc}$ the maximal torus of $G^{\sc}$. Pick two parahoric subsgroups $\eP_1, \eP_2 \leq \lo G$. Consider the affine flag variety $\Fl^J_G := \lo G / \eP_2$, together with the action of $\eP_1 \rtimes \Gm^{\rot} \acts \Fl^J_G$.
Then 
\begin{equation*}
\widehat{\O}(\Bl^{\circ}_{\widetilde{\eZ}}(\Gm^{\rot} \times T\sslash W_1 \times T^{\sc} \sslash W_2 \times \pi_1(G)) = K^{\top, 0}_{\eP_1 \rtimes \Gm^{\rot}}(\Fl^J_G, \C).
\end{equation*}
\end{theorem}

\begin{proof}
The setup of theorem has four degrees of freedom -- we can independently choose:
(i) which parahoric will act on the left;
(ii) which partial affine flag variety are we taking;
(iii) which isogeny class of $G$ we take for $\lo G$, or more precisely which connected components of $\Gr_{G^{\ad}}$ we take into account; 
(iv) which isogeny class of $G$ will be used for the $\lop G$-action, or alternatively which torus $T$ with $T^{\der}$ in between $T^{\sc} \to T^{\der} \to T^{\ad}$ will act.

We will treat these independently: (i) and (ii) in Lemma \ref{lemma: equivariant K-theory as deformation to normal cone -- flag varieties and equivariance}; (iii) in Lemma \ref{lemma: equivariant K-theory as deformation to normal cone -- G reductive}; (iv) in Lemma \ref{lemma: equivariant K-theory as deformation to normal cone -- isogeneous torus} and Remark \ref{remark: equivariant K-theory as deformation to normal cone -- isogeneous torus}. It is quite clear that all this bookkeeping can be done simultaneously; we leave the details to the reader. 
\end{proof}

In what follows, we also write $\Bl^{\circ}_{\widetilde{\eZ}}(X)$ for the affine blowup at the relevant quotients and restriction of the family $\widetilde{\eZ}$ which lives in the given variant $X$ of the scheme $\Gm^{\rot} \times T\sslash W_1 \times T^{\sc} \sslash W_2 \times \pi_1(G)$. This should not lead to any confusion.

\subsubsection{Affine flag varieties and positive loop-group equivariance.}
Assume $G = G^{\sc}$ is simply connected and consider its torus $T=T^{\sc}$. Take two parahoric subgroups $\eP_1, \eP_2 \leq \lo G$. Consider the corresponding affine flag variety $\Fl^J = \lo G / \eP_2$, together with the action
$\eP_1 \rtimes \Gm^{\rot} \acts  \Fl^J_G$. The equivariant $K$-theory of this action is the $K$-theory of the double quotient (quantum Hecke stack)
\begin{equation*}
    \eP_1 \rtimes \Gm^{\rot} \backslash \lo G \rtimes \Gm^{\rot} / \eP_2 \rtimes \Gm^{\rot}
\end{equation*}
which has an apparent symmetry. 

\begin{lemma}\label{lemma: equivariant K-theory as deformation to normal cone -- flag varieties and equivariance}
For $G$ simply connected, we have a ring isomorphism 
\begin{equation*}
    \widehat{\O}(\Bl^{\circ}_{\widetilde{\eZ}}(T\sslash W_1 \times \Gm^{\rot} \times T\sslash W_2))
    =K^{\top, 0}_{\eP_2 \rtimes \Gm^{\rot}}(\Fl^J_G; \C).
\end{equation*}   
\end{lemma}
\begin{proof}
By homotopy invariance and equivariant formality, we have
\begin{equation*}
K^{\top, 0}_{\eP_1 \rtimes \Gm^{\rot}}(\Fl^J_G; \C) = K^{\top, 0}_{L_1 \times \Gm^{\rot}}(\Fl^J_G; \C) = (K^{\top, 0}_{T \times \Gm^{\rot}}(\Fl^J_G; \C))^{W_1}.    
\end{equation*}
Here, the $W_1$-action is on the first copy of the torus. 

For the other part, note that $\Fl^J_G \to \Gr_G$ is a fpqc locally trivial fibration with fiber $G/P_{2}$. The $K$-theory of $\Fl^J_G$ is thus given by the following base change:
\begin{equation*}
    K^{\top, 0}_{T\times \Gm^{\rot}}(\Fl^J_G; \C) = K^{\top, 0}_{T\times \Gm^{\rot}}(\Gr_G; \C) \otimes_{\O(T \sslash W)} \O(T \sslash W_2).
\end{equation*}
Since blowups are compatible with base change, we have reduced the desired statement to Theorem \ref{theorem: equivariant K-theory as deformation to normal cone -- G sssc}.(Alternatively, one can directly repeat our computation of $T \times \Gm^{\rot}$-equivariant $K$-theory of $\Gr$ in the more general case of $\Fl^J_G$.)
\end{proof}

\subsubsection{Affine Grassmannian of reductive $G$.}
Assume now $G$ to be merely reductive, but let us work equivariantly with respect to the torus $T^{\sc}$ of $G^{\sc}$, see \S \ref{subsubsection: simply connected covers and adjoint groups}.

\begin{lemma}\label{lemma: equivariant K-theory as deformation to normal cone -- G reductive}
Assume $G$ is connected reductive. Then the above constructed map induces a ring isomorphism
\begin{equation*}
   \widehat{\O}(\Bl^{\circ}_{\widetilde{\eZ}}(T^{\sc} \times \Gm^{\rot} \times T^{\sc}\sslash W \times \pi_1(G))) 
   = K^{\top, 0}_{T^{\sc} \times \Gm^{\rot}}(\Gr_G; \C).
\end{equation*}
\end{lemma}
\begin{proof}
The affine Grassmannian $\Gr_G$ has connected components $\Gr^{\tau}_G$ labeled by elements $\tau \in \pi_1(G)$. The identity component $\Gr^1_G$ is $T^{\sc} \times \Gm^{\rot}$-equivariantly isomorphic to $\Gr_{G^{\sc}}$, which we have treated already in our main theorem.

In general, we can repeat the proof of the main theorem for all components at once, using the corresponding variant of cocharacter graphs from \S \ref{section: cocharacter graphs for reductive groups}. Alternatively, we can follow \cite{BF07} and bootstrap the current case as follows. For each $\tau \in \pi_1(G)$ choose a lift $\underline{\tau} \in X_{\bullet}$. Translation by $\tau$ then identifies the identity component $\Gr^1_G$ with $\Gr^{\tau}_G$. This isomorphism is twisted-equivariant for the $T^{\sc} \times \Gm^{\rot}$-action; the twisting is given by the conjugation with $t^\tau$. Under this twisting, we identify the copies of $T^{\sc} \times \Gm^{\rot} \times T \sslash W$ compatibly with the cocharacter graphs inside them as in \S \ref{subsubsection: cocharacter graphs for reductive groups}. The conclusion follows.
\end{proof}

\begin{remark}
The connected components of $\Gr_G$, labeled by the fundamental group $\pi_1(G)$, are all abstractly isomorphic; such isomorphisms are realized by translating with elements of the loop group $\lop G$, for example by an element corresponding to some $\lambda \in X_{\bullet}$. Each of these components is isomorphic to the affine Grassmannian $\Gr_{G^{\sc}}$ of $G^{\sc}$.

However, we would like to emphasize the following caveat: these components are \textit{not} isomorphic equivariantly, not even with respect to the extended torus $T \times \Gm^{\rot}$. Indeed, given two components $\Gr^{\sigma}_G$, $\Gr^{\tau}_G$ for some $\sigma$, $\tau \in \pi_1(G)$, their equivariant $K$-theory rings are usually different as $\O(T \times \Gm^{\rot})$-algebras.

Indeed, as we explained, the fiber of the $K$-theory ring over a point $(x, \zeta) \in T \times \Gm^{\rot}$ is given by the cohomology of fixed-points. Now, take $(1, \zeta)$ for generic $\zeta$, and consider the affine Grassmannian of $PGL_2$. The fixed points of $(1, \zeta)$ are given by a union of partial flag varieties. In the even component, the fixed-point scheme contains a unique connected component isomorphic to a point. In the odd component, the fixed-point scheme has only components given by partial flag varieties of positive dimension. The cohomology rings of the fixed-points are not the same.

The same situation occurs for cohomology in \cite[\S 3, proof of Theorem 1.(b)]{BF07}.
\end{remark}

\subsubsection{Equivariance with respect to actual torus.}
Above, we have always considered equivariance with respect to the maximal torus $T^{\sc}$ of the simply connected $G^{\sc}$. However, one can well ask about a description with respect to the actual maximal torus $T \leq G$. 

We first assume $G$ is semisimple, giving isogenies $T^{\sc} \to T \to T^{\ad}$. To simplify notation, we take only the identity component $\Gr_{G^{\sc}}$ of $\Gr_G$. 
\begin{lemma}\label{lemma: equivariant K-theory as deformation to normal cone -- isogeneous torus}
Assume $G$ is semisimple. We have 
\begin{equation*}
    \widehat{\O}(\Bl^{\circ}_{\widetilde{\eZ}}(T \times \Gm^{\rot} \times T^{\sc}\sslash W ))
    =\widehat{\O}(\Bl^{\circ}_{\widetilde{\eZ}}(\tfrac{T^{\sc} \times T^{\sc} \sslash W}{\pi_1(G)} \times \Gm^{\rot}))
    =K^{\top, 0}_{T \times \Gm^{\rot}}(\Gr_{G^{\sc}}; \C).
\end{equation*}
\end{lemma}
\begin{proof}
By Iwahori pavings, the equivariant $K$-theory of the affine Grassmannian is equivariantly formal. Since the action of $T^{\sc}$ factors through $T$, we know that $K^{\top, 0}_{T^{\sc}}(\Gr_G; \k)$ is the pullback of $K^{\top, 0}_{T \times \Gm^{\rot}}(\Gr_G; \k)$ along the $\pi_1(G)$-torsor $T^{\sc} \to T$. Conversely, the $T$-equivariant $K$-theory is given by descending the $T^{\sc}$-equivariant one, in other words by taking the quotient by the $\pi_1(G)$-action.

Now, on the inclusion $\Gamma \hookrightarrow T^{\sc} \times \Gm \times T^{\sc}$, this action is given by \S \ref{section: action of the fundamental group}. On this level, the quotient by $\pi_1(G)$ is \eqref{equation: descended Gamma eq 2}. Since the affine blowup construction and its Schubert completion is compatible with quotients by finite groups, we deduce the result.    
\end{proof}

\begin{remark}
Note the apparent break of symmetry between the two actions. This is due to our desire to explicitly see the equivariant base $T$ as the first copy. If we wish to keep the symmetry, the appropriate description via $\tfrac{T^{\sc} \times T^{\sc}}{\pi_1(G)}$ comes from \eqref{equation: descended Gamma eq 1}.
\end{remark}

\begin{remark}\label{remark: equivariant K-theory as deformation to normal cone -- isogeneous torus}
Now assume $G$ is reductive. Then we have
\begin{equation*}
\widehat{\O}(\Bl^{\circ}_{\widetilde{\eZ}}(T \times \Gm^{\rot} \times T^{\sc}\sslash W ))
=K^{\top, 0}_{T \times \Gm^{\rot}}(\Gr_{G^{\sc}}; \C).    
\end{equation*}
Indeed, this follows by first applying Lemma \ref{lemma: equivariant K-theory as deformation to normal cone -- isogeneous torus} with respect to the the maximal torus $T^{\ad}$ of the semisimple group $G^{\ad}$, and then base-changing along $T \to T^{\ad}$.
\end{remark}

%% file: A_some_explicit_computations.tex
\section{Some explicit localization formulas}\label{appendix: some explicit formulas in type A}
This appendix is dedicated to spelling out some of our formulas for $GL_n$. We also recall the determinant line bundle and show that it only appears after Schubert completion.

\subsection{Formulas in \texorpdfstring{$GL_n$}{GLn}}
For the convenience of the reader, we record the description of the $K$-theory of $\Gr_{GL_n}$ together with its localization formulas explicitly in usual coordinates.

\subsubsection{Localization formulas in \texorpdfstring{$GL_n$}{GLn}.}\label{equation: localization formulas in GLn}
Consider $G = GL_n$. The connected components of $\Gr_{GL_n}$ are labeled by $\pi_1(GL_n) = \Z$.
We want to describe the localization map from functions on the corresponding union of affine blowups.
\begin{align*}
  \loc:  \prod_{\Z} \O(\eN_{\mu_{\infty}}(\eZ, T \times T \sslash W)) &\to \prod_{\lambda \in X_{\bullet}} \O(T \times \Gm^{\rot})
\end{align*}
Let $\lambda \in X_{\bullet}$ and denote $d_1, \dots, d_n$ the heights of the columns of its Young diagram. The localization $\loc_{\lambda}$ to the corresponding fixed point $t^{\lambda} = (t^{d_1}, \dots, t^{d_n})$ is given as follows. First project to the component labeled by $\overline{\lambda} = d_1 + d_2 + \dots + d_n \in \Z$. On generators of this copy, the localization map reads as
\begin{align*}
   \gamma_{\lambda}: \O(\eN_{\mu_{\infty}}(\eZ, T \times T \sslash W)) & \xrightarrow{} \O(T \times \Gm^{\rot}) \\
    t_i &\mapsto t_i \\
    q &\mapsto q \\
    s_i &\mapsto t_i q^{d_i} \\    
    e_j(s_1, \dots, s_n) &\mapsto e_j(t_1 q^{d_1}, \dots, t_n q^{d_n}) \\
    b_{j, \ell} = \tfrac{e_j(t^{\ell}_1, \dots, t^{\ell}_n) - e_j(s^{\ell}_1, \dots, s^{\ell}_n)}{1-q^{\ell}} &\mapsto \tfrac{e_j(t^{\ell}_1, \dots, t^{\ell}_n) - e_j(t_1 \cdot q^{d_1}, \dots, t_n \cdot q^{d_n})}{1-q^{\ell}}.
\end{align*}
Also note that these formulas work uniformly with respect to the connected components of $\Gr_{GL_n}$.

\subsubsection{Determinant line bundle \texorpdfstring{$\eL_{\det}$}{L det}.}
Let us first recall the construction of the determinant line bundle $\eL_{\det}$ on $\Gr_{GL_n}$.

\begin{recollection}
To a perfect complex $\eE \in \Perf(\eX)$ one can associate a graded line bundle $\det(\eL) \in \Pic(\eX)^{\Z}$. If we represent $\eE$ by an actual bounded complex of vector bundles on $\eX$, the determinant line bundle is given by an alternating product of the top wedge powers of the terms of $\eE$, while the grading records the alternating sum of their ranks. This descends to a map $K(\eX) \to \Pic(\eX)^{\Z}$. See \cite[Tag 0FJI]{Sta}, \cite[Proposition 5.3, \S 5, \S 12]{BS16}. 

In particular, to any class in $K^{0}(\eX)$ one can associate the isomorphism class of the corresponding line bundle; similarly in the topological setting.
\end{recollection}

\begin{construction}
We obtain the \textit{determinant line bundle} $\eL_{\det}$ on $\Gr_{GL_n}$ as follows. On each $\Gr_{\leq \mu}$, set
\begin{equation*}
    \eL_{\leq \mu, \det} := \det([\eE_{\leq \mu}] - [\eG_{\leq \mu}]).
\end{equation*}
These are compatible in $\mu$, and hence give a well-defined equivariant line bundle on the whole $\Gr_{GL_n}$. For each $\sigma \in \pi_1(G)$, it is known to be an ample generator of the Picard group $\Pic(\Gr^{\sigma}) = \Z$.

Also note that $\eL_{\det}$ is naturally $\lop GL_n \rtimes \Gm^{\rot}$-equivariant. It is not equivariant for the whole loop group $\lo G \rtimes \Gm^{\rot}$; this is corrected by working with respect to its central extension. 

For fixed $\mu$, $[\eL_{\det}] = \det(b_{1,1})$ is expressible in terms of $\lambda$-operations on $b_{1,1}$. With rational coefficients, the restriction of $[\eL_{\det}]$ to any $\Gr_{\leq \mu}$ may thus be expressed in terms of $b_{j, \ell}$. However, this expression is not uniform in $\mu$ -- the number of terms grows. (We will see shortly that the class $[\eL_{\det}] \in K^{\top, 0}_{T \times \Gm^{\rot}}(\Gr)$ can be seen only after the Schubert completion of our affine blowup.)
\end{construction}

\begin{claim}\label{equation: localization of determinant line bundle in GLn} \label{claim: localization of determinant line bundle in GLn}
The class of the determinant line bundle $\eL_{\det}$ inside the GKM description is
\begin{equation*}
   \loc_{\lambda}( [\eL_{\det}] ) = \prod_{k=1}^n t^{d_k}_k \cdot q^{\binom{d_k}{2}}, \qquad \forall \lambda \in X_{\bullet}.
\end{equation*}
\end{claim}
\begin{proof}
Directly follows from the localization formulas for $b_{1,1}$ in \S \ref{section: bezrukavnikov classes for GL_n from tautological bundles}. 
\end{proof}

\begin{remark}
    One can easily make the sanity check that the above class $[\eL_{\det}]$ indeed satisfies the GKM conditions. Namely, given $\lambda$ and $\lambda' = s_{\beta, m}\cdot \lambda$, where $\beta$ corresponds to the reflection switching $i$ and $j$, we can compute
    \begin{align*}
       \loc_{\lambda}( [\eL_{\det}]) - \loc_{\lambda'}( [\eL_{\det}] ) 
       &= \left( \prod_{k=1}^n t^{d_k}_k \cdot q^{\binom{d_k}{2}} \right) \cdot \left( 1 - t_jt_i^{-1} q^{d_j -d_i +1} \right)
    \end{align*}
 which is divisible by $(1 - {\beta} \cdot q^{m}) = (1 - t_jt_i^{-1} q^{d_j -d_i +1})$.   
\end{remark}

\begin{remark}
For general reductive $G$, the Picard group of each connected component of $\Gr_G$ is given by $\Z$. We denote its ample generator by $\eL_{\ver}$ and call it the \emph{Verlinde line bundle} (we reserve the name \emph{determinant line bundle} for $GL_n$). Pulling back $\eL^{GL_n}_{\det}$ along representations $G \to GL_n$ gives nonnegative powers of its ample generator $\eL_{\ver}$; these can be made explicit. See \cite[Lemma 4.2, Remark 4.3]{YZ09} for a discussion.  
\end{remark}

To describe the ample generator $\eL_{\ver}$ for a given component of the affine Grassmannian $\Gr_G$ of a reductive group $G$, it is sufficient to understand the case of $G=G^{\sc}$. Here, we have the following.
\begin{lemma}\label{lemma: localization of the verlinde line bundle}
Assume $G$ is simply connected.
The localization $\loc ([\eL_{\ver}])$ is given as follows. For any $\lambda \in X_{\bullet}$, we have
\begin{equation*}
    \loc_{\lambda} ([\eL_{\ver}]) = t_{\lambda} \cdot q^{\tfrac{1}{2} \langle \lambda^{\vee}, \lambda \rangle} \in \O(T \times \Gm^{\rot}).
\end{equation*}
\end{lemma}
\begin{proof}
Well-known, for example a special case of \cite[Proposition 7.6]{Yun10}. Note that \cite{Yun10} describes the localization formulas for the first Chern class of powers and twists of $\eL_{\ver}$ on the corresponding affine flag variety. By taking in his notation $\kappa=1$ and $\chi=1$, we get the pullback of the Verlinde line bundle from $\Gr_G$. The equivariant first Chern class of the localization of $\eL_{\ver}$ determines its equivariant $K$-theory class by exponentiation.
\end{proof}

\subsubsection{Necessity of Schubert completion.}
The following observation well illustrates the necessity of Schubert completion.
\begin{lemma}\label{lemma: necessity of Schubert completion}
The class of the determinant line bundle does not lies in the functions on the uncompleted deformation
\begin{equation*}
    [\eL_{\det}] \notin \O(\eN_{\mu_{\infty}}(\eZ, T \times \Gm^{\rot} \times T \sslash W)).
\end{equation*}
\end{lemma}
\begin{proof}
For the sake of contradiction, suppose otherwise. Then we could express $[\eL_{\det}]$ as a polynomial $f$ over $k$ in finitely many elements from
\begin{equation}\label{equation: elements}
 \{ t_i, q, e_j, b_{j, \ell} \mid i = 1, \dots, n; j = 1, \dots, n; \ell \in \N \}   
\end{equation}
For any such fixed $f$, the localization map $\gamma = (\gamma_{\lambda} \mid \lambda \in \Z^d)$ has the following property: The powers of $q$ in the family of polynomials
\begin{equation*}
 \gamma_{\lambda}(f) \in \k[t_1^{\pm 1}, \dots, t^{\pm 1}_n, q^{\pm 1}]   
\end{equation*}
grow linearly in $\lambda = (d_1, \dots, d_n)$. Indeed, this is true for the elements from \eqref{equation: elements} by \S \ref{equation: localization formulas in GLn} and is preserved under polynomial combinations.

However, the powers of $q$ in the family
\begin{equation*}
 \gamma_{\lambda}([\eL_{\det}]) \in \k[t_1^{\pm 1}, \dots, t^{\pm 1}_n, q^{\pm 1}]   
\end{equation*}
grow quadratically in $\lambda = (d_1, \dots, d_n)$ by \eqref{equation: localization of determinant line bundle in GLn}.

Therefore, there exists a sufficiently big $\lambda$, such that $\gamma_{\lambda}([\eL_{\det}]) \neq \gamma_{\lambda}(f)$, a contradiction.
\end{proof}

\begin{remark}
The statement of Lemma \ref{lemma: necessity of Schubert completion} follows for any reductive group $G$. Indeed all the classes in question are base-changed from the $GL_n$-case; the statement about linear vs. quadratic growth of $q$-powers for varying $\lambda$ still holds -- the localizations may be also expressed purely Lie theoretically.
\end{remark}

In other words, $[\eL_{\det}]$ is not present in the uncompleted ring $\O(\eN_{\mu_{\infty}}(T \times T \sslash W))$. At the same time $[\eL_{\det}]$ is as nice a line bundle on $\Gr_G$ as one could hope for. This makes it challenging to give a purely $K$-theoretical interpretation of the decompleted ring $\O(\eN_{\mu_{\infty}}(T \times T \sslash W))$.

\subsubsection{\texorpdfstring{$\lambda$}{Lambda}-rings and affine Grassmannian of \texorpdfstring{$GL_{\infty}$}{GL infinity}.}

\begin{remark}
When $G=\GL_n$ and the coefficient ring $\k$ is a field of characteristic zero, it is possible to do with less topological generators. As complete topological $\lambda$-ring over the equivariant base, $K^{\top, 0}_{T \times \Gm^{\rot}}(\Gr; \C)$ is generated by the single tautological element $b_{1,1}$.

For $GL_n$, the representation ring $R(GL_n; \k) = \O(T \sslash W)$ is a quotient of the free $\lambda$-ring in one variable. Since we are working over a field of characteristic zero, Newton's formulas imply that we may further express any $\lambda$ operation (an elementary symmetric polynomial) purely in terms of Adams operations (power sums).
\end{remark}

\begin{remark}
Assume $G = GL_n$. In our presentation of $K$-theory of $\Gr_{GL_n}$, we have used the tautological class $b_{1,1}$ and $\psi$-operations (or $\lambda$-operations) on it. However, for fixed $n$, there are relations between these classes. 

This is closely related to the following phenomenon. The representation ring $R(GL_n) = \O(T \sslash W)$ is a quotient of the free $\lambda$-ring in one variable. The free $\lambda$-ring in one variable is given by putting all $R(GL_n)$ together for varying $n$.

To have a uniform presentation for the $K$-theory, one could consider the ind-scheme
\begin{equation*}
    \Gr_{GL_{\infty}} := \colim_n \Gr_{GL_n}
\end{equation*}
equivariant with respect to the extended loop group of the ind-group scheme $GL_{\infty} := \colim_n GL_n$.
\end{remark}